\documentclass[12pt]{amsart}
\usepackage{fullpage}
\usepackage{amsfonts, amsmath, amssymb, amsthm}
\usepackage{paralist}
\usepackage{graphicx}
\usepackage{tabmac}
\usepackage{lmodern, xcolor}
\usepackage[normalem]{ulem}
\usepackage{hyperref}

\newtheorem{thm}{Theorem}[section]
\newtheorem{prop}[thm]{Proposition}
\newtheorem{lemma}[thm]{Lemma}

\newtheorem{cor}[thm]{Corollary}
\newtheorem{propdef}[thm]{Proposition-Definition}
\theoremstyle{definition}
\newtheorem{defn}[thm]{Definition}
\newtheorem{definition}[thm]{Definition}
\theoremstyle{remark}
\newtheorem{rmk}[thm]{Remark}
\newtheorem{ex}[thm]{Example}

\newcommand{\ol}[1]{\overline{#1}}
\newcommand{\fw}[1]{{\operatorname{fw}\mathopen{}\left(#1\right)\mathclose{}}}
\newcommand{\str}[2]{{\operatorname{\mathfrak{st}}_{#1}\mathopen{}\left(#2\right)\mathclose{}}}
\def\st{{\mathfrak{st}}} 
\newcommand{\Z}{\mathbb{Z}}
\newcommand{\A}{\mathcal{A}}
\newcommand{\B}{\mathcal{B}}

\newcommand{\inv}{\operatorname{inv}}
\newcommand{\sgn}{\operatorname{sgn}}
\newcommand{\ch}{\operatorname{charge}}
\newcommand{\lch}{\operatorname{lch}}
\newcommand{\dom}{\Omega_{\mathrm{dom}}}
\newcommand{\Gup}{G^{\mathrm{up}}}
\newcommand{\fix}{\operatorname{fix}}

\newcommand{\abs}[1]{\left\lvert#1\right\rvert}
\newcommand{\ceil}[1]{\left\lceil#1\right\rceil}

\newcommand{\typeA}{\textrm{A}}

\graphicspath{{figures/}}

\author{Michael Chmutov}
\address{Department of Mathematics, University of Minnesota, Minneapolis, MN 55455, USA}
\email{mchmutov@umn.edu}
\author{Joel Brewster Lewis}
\address{Department of Mathematics, University of Minnesota, Minneapolis, MN 55455, USA}
\email{jblewis@umn.edu}
\author{Pavlo Pylyavskyy}
\address{Department of Mathematics, University of Minnesota, Minneapolis, MN 55455, USA}
\email{ppylyavs@umn.edu}
\thanks{M.C. was partially supported by NSF grant DMS-1503119. J.B.L. was partially supported by NSF grant DMS-1401792.  P.P. was partially supported by NSF grants DMS-1148634, DMS-1351590, and a Sloan Fellowship.}

\title{Monodromy in Kazhdan-Lusztig cells in affine type A.}
\begin{document}
\begin{abstract}
We use the affine Robinson-Schensted correspondence to describe the structure of bidirected edges in the Kazhdan-Lusztig cells in affine type A. Equivalently, we give a comprehensive description of the Knuth equivalence classes of affine permutations. 
\end{abstract}
\maketitle

\setcounter{tocdepth}{1}
\tableofcontents

\section{Introduction}

\subsection{Cells in Kazhdan-Lusztig theory}
In a groundbreaking paper \cite{KL}, Kazhdan and Lusztig laid a basis for a new approach to representation theory of Hecke algebras. Since then, this approach has been significantly developed, and is called \emph{Kazhdan-Lusztig theory}. (For a nice introduction to Kazhdan-Lusztig theory, see \cite[Ch.~7]{humphreys}.) Of particular importance in this theory are the objects called \emph{cells}.   Briefly, their definition is as follows. Each Hecke algebra is associated with a Coxeter group $W$. Kazhdan and Lusztig define a pre-order $\leq_L$ on
elements of $W$. Some pairs $v,w$ of elements of $W$ satisfy both $v \leq_L w$ and $w \leq_L v$, in which case we say that they are left-equivalent, denoted $v \sim_L w$. Similarly one can define right equivalence $\sim_R$. The respective equivalence classes are called the \emph{left cells} and the \emph{right cells}.

Another way to describe cells is via the Kazhdan-Lusztig $W$-graph; it is a certain directed graph whose vertices are the elements of $W$. The graph has the property that $v \leq_L w$ precisely when there is a directed path from $v$ to $w$. Thus the cells are the strongly connected components of the $W$-graph. Some edges of the $W$-graph are bidirected, i.e., between a pair of vertices $v$ and $w$ there is an edge $v\to w$ and an edge $w\to v$. In this case, of course, $v$ and $w$ belong to the same Kazhdan-Lusztig cell.

\subsection{Type A}
In (finite) type A, when $W$ is the symmetric group, the Kazhdan-Lusztig cell structure corresponds to something very familiar to combinatorialists, the \emph{Robinson-Schensted correspondence}. This is a bijective correspondence between elements of the symmetric group and pairs $(P, Q)$ of standard Young tableaux of the same shape. It is well known \cite{BV, KL, GM, A} that
\begin{itemize}
 \item two permutations lie in the same left cell if and only if they have the same \emph{recording tableau} $Q$, and
 \item two permutations lie in the same right cell if and only if they have the same \emph{insertion tableau} $P$.
\end{itemize}
The bidirected edges of the Kazhdan-Lusztig graph in this case are called \emph{Knuth moves}, and one can go between any two permutations with the same insertion tableau (and hence between any two choices of the recording tableau $Q$) via a series of Knuth moves.

\subsection{Affine type A}
In affine type A, when $W$ is an affine symmetric group, Chmutov, Pylyavskyy, and Yudovina \cite{cpy} described, via a combinatorial algorithm called the Affine Matrix Ball Construction (AMBC), a bijection $W\to\dom$, where $\dom$ is the set of triples $(P,Q,\rho)$ such that $P$ and $Q$ are tabloids of the same shape and $\rho$ is an integer vector (called a \emph{dominant weight}) satisfying certain inequalities that depend on $P$ and $Q$. Relying on the work of Shi on Kazhdan-Lusztig cells in affine type A, they show that this bijection affords a description of cells analogous to the non-affine case: fixing the tabloid $Q$ gives all affine permutations in a left cell while fixing the tabloid $P$ gives all affine permutations in a right cell. The bidirected edges in this case (what Shi called \emph{star operations} \cite{shi}) are natural analogues of Knuth moves -- see Section \ref{sec:knuth on perms} for the definition.

\subsection{Summary of results}
The main accomplishment of this paper is to precisely describe the Knuth equivalence classes of affine permutations, i.e., the equivalence classes of the relation generated by Knuth moves. In the language of \cite{stembridge}, we describe the \emph{Kazhdan-Lusztig molecules}. Unlike the case of finite type A, most Kazhdan-Lusztig cells are composed of many molecules.  This multiplicity comes in two varieties.  First, while Knuth moves preserve the $P$ tabloid, not all $Q$ tabloids can be reached from a given one using Knuth moves; the description of which ones can be reached, given in Section \ref{sec:charge}, is in terms of a variant of the \emph{charge statistic}. Second, even among affine permutations having the same $P$ and $Q$ tabloids, one cannot reach every dominant weight $\rho$ from every other using Knuth moves; the different vectors $\rho$ that can be reached are the subject of Section \ref{sec:monodromy}, and they depend on the shape of the tabloids.

Along the way, we improve our understanding of some combinatorial aspects of AMBC analogous to the combinatorics of the Robinson-Schensted correspondence. We show how to read off the left and right descent sets of an affine permutation from its image under AMBC (Proposition \ref{prop:descents}), as well as how to read off its sign (Theorem \ref{thm:main conjecture v2}). We describe precisely how taking the inverse of or doing a Knuth move on an affine permutation affects its image under AMBC (Proposition~\ref{prop:inverses} and Theorem \ref{thm:Knuth move action}). In Section~\ref{sec:symmetries}, we briefly discuss how the symmetry of the Dynkin diagram appears in our setting. Finally, we give a better description, in terms of charge, for the inequalities satisfied by an integer vector for it to be dominant (Theorem~\ref{thm:dominance constants}).  The appearance of charge suggests a connection to crystal graphs, which we describe in Section~\ref{sec:crystals}.

Section~\ref{sec:background} contains the background information, including a summary of the main results of \cite{cpy}, essential to understanding the rest of the paper. Sections~\ref{sec:Knuth moves} and~\ref{sec:signs} improve our understanding of the various combinatorial aspects of AMBC. Sections \ref{sec:charge intro}, \ref{sec:monodromy} and  \ref{sec:charge} describe the Knuth equivalence classes of permutations in terms of AMBC; these sections rely heavily on the material in Section~\ref{sec:Knuth moves} but may be read independently of Section~\ref{sec:signs}. 

There is a program available to compute AMBC \cite{Program}; the reader may want to use it to explore additional examples.

\section{Background}
\label{sec:background}
\subsection{Notational preliminaries}
\label{sec:notational preliminaries}
For the duration of this paper, $n$ will be a fixed positive integer. Let $[n] := \{1,\dots, n\}$. For each $i\in\Z$, denote by $\ol{i}$ the residue class $i+n\Z$, and let $[\ol{n}]:= \{\ol{1},\dots, \ol{n}\}$.

The \emph{symmetric group} $S_n$ is the Weyl group of type $\typeA_{n-1}$.  We may variously think of its elements as bijections $[n] \to [n]$, as words of length $n$ containing each element of $[n]$ exactly once, or as $n \times n$ permutation matrices.  The \emph{extended affine symmetric group} $\widetilde{S_n}$ is the extended affine Weyl group of type $\widetilde{\typeA}_{n-1}$; it consists of bijections $w:\Z\to\Z$ such that 
\[
w(i + n) = w(i) + n \qquad \textrm{ for all } i.
\]
The elements of $\widetilde{S_n}$ are called \emph{extended affine permutations}. We typically abbreviate this term to \emph{permutations}, and we distinguish the elements of $S_n$ by the name \emph{finite permutations}.
Denote by $\widetilde{S}_n^0$ the \emph{affine symmetric group}, i.e., the affine Weyl group of type $\widetilde{\typeA}_{n-1}$; it consists of permutations $w\in\widetilde{S_n}$ such that 
\[
\sum_{i=1}^n \left(w(i) - i\right) = 0.
\]
Note that $S_n$ naturally embeds into $\widetilde{S}_n^0\subsetneq\widetilde{S_n}$: a finite permutation $w$ can be sent to the unique affine permutation that takes the same values as $w$ on $[n]$.

A \emph{partial (extended affine) permutation} is a pair $(U, w)$ where $U\subseteq\Z$ has the property that $(x\in U) \Leftrightarrow (x+n\in U)$ and $w:U\to\Z$ an injection such that $w(i+n) = w(i)+ n$. We suppress the explicit mention of the subset $U$ in the notation and just refer to the partial permutation $w$. Any permutation may be viewed as a partial permutation with $U = \Z$. 

A permutation is determined by its values on $1,\dots, n$. The \emph{window notation} for a permutation $w$ is $[w(1), \dots, w(n)]$.	A partial permutation $w$ is also determined by its values on $1,\dots, n$, except it may not be defined on some of them. 

We often think of permutations in terms of pictures such as the one in Figure \ref{fig:proper numbering}, extending the notion of a permutation matrix to the affine case.  More precisely, on the plane we draw an infinite matrix; the rows are labeled by $\Z$, increasing downward, and the columns are labeled by $\Z$, increasing to the right (usual matrix coordinates). The positions in this matrix are called \emph{cells}; there will never be confusion with cells of the Kazhdan-Lusztig variety.  To distinguish the $0$-th row, figures have a solid red line between the $0$-th and $1$-st rows, and similarly for columns. We also put dashed red lines every $n$ rows and columns.  If $w(i)=j$ then we place a ball in the $i$-th row and $j$-th column.  For example, the cell $(1,4)$ in Figure \ref{fig:proper numbering} contains a ball.  The balls of a partial permutation will also be referred to by their matrix coordinates. Thus, formally, both balls and cells are just ordered pairs of integers, and we use the word ``ball'' to indicate that the relevant partial permutation takes a certain value on a certain input. For a partial permutation $w$, we denote by $\B_w$ the collection of balls of $w$ (a subset of $\Z\times\Z$).

For an integer $k$ and a ball $b=(i,j)$, the ball $b' = b + k(n, n) = (i+kn, j+kn)$ is the \emph{$k(n,n)$-translate} of $b$. Two balls $b$ and $b'$ are \emph{translates} if for some $k$ one is a $k(n,n)$-translate of the other.  The set of all translates of a ball or set of balls is a \emph{translation class}.

We often assign numbers to balls of permutations, as well as to other cells of $\Z \times \Z$. For a partial permutation $w$, a \emph{numbering} of $w$ is a function $d:\B_w\to\Z$. A numbering $d$ of $w$ is \emph{semi-periodic} with \emph{period} $m$ if we have $d(b+(n,n)) = d(b) + m$ for every $b\in\B_w$. When referring to a numbering in pictures, we write the number $d(b)$ inside the ball $b$ as in Figure \ref{fig:proper numbering}, where we show a semi-periodic numbering of period 3.

We frequently use compass directions (north, east, etc.)\ to describe relative positions of balls or cells, with north being toward the top of the page (smaller row numbers) and east being toward the right of the page (larger column numbers). Adding the modifier ``directly'' constrains one of the two coordinates: a cell $(i,j)$ is \emph{directly south} of $(i',j')$ if $i\geqslant i'$ and $j=j'$. By a composite direction (e.g., northeast) we mean north and east. The relations are weak by default: a cell $(i,j)$ is \emph{southwest of $(i',j')$} if $i\geqslant i'$ and $j\leqslant j'$. Directions define partial orders on $\Z\times\Z$: we say $(i,j)\leqslant_{SW} (i',j')$ if $(i,j)$ is southwest of $(i',j')$. 

A \emph{partition} $\lambda$ is a finite, weakly decreasing sequence of positive integers.  We typically treat a partition $\lambda = \langle \lambda_1, \lambda_2, \ldots \rangle$ as equivalent to its \emph{Young diagram}, a left-justified collection of rows of boxes with top row having $\lambda_1$ boxes, the row below it having $\lambda_2$ boxes, and so on.  The number of rows of $\lambda$ is denoted $\ell(\lambda)$.

Given a partition $\lambda$ having $n$ boxes, a \emph{tabloid} of \emph{shape} $\lambda$ is an equivalence class of fillings of the Young diagram of $\lambda$ with $[\ol{n}]$, where two fillings are considered equivalent when one is obtained from the other by permuting elements within rows.\footnote{Classically, one considers fillings with integers, but nothing is lost by this slightly nonstandard choice.}  Some examples are shown of Figure \ref{fig:tabloids}; we draw representatives of the equivalence classes and keep in mind that entries within rows can be permuted.  Throughout this paper, all tabloids will be filled with distinct residue classes.

\begin{figure}
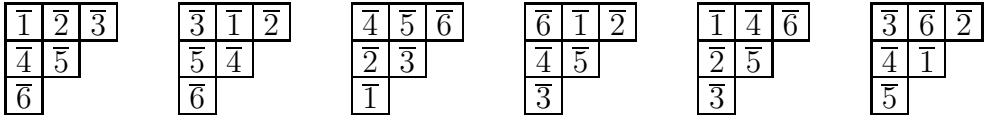

\[
\tableau[sY]{\ol{1}, \ol{2}, \ol{3} \\ \ol{4}, \ol{5}\\ \ol{6}} \qquad 
\tableau[sY]{\ol{3}, \ol{1}, \ol{2} \\ \ol{5}, \ol{4}\\ \ol{6}} \qquad 
\tableau[sY]{\ol{4}, \ol{5}, \ol{6} \\ \ol{2}, \ol{3}\\ \ol{1}} \qquad 
\tableau[sY]{\ol{6}, \ol{1}, \ol{2} \\ \ol{4}, \ol{5}\\ \ol{3}}\qquad 
\tableau[sY]{\ol{1}, \ol{4}, \ol{6} \\ \ol{2}, \ol{5}\\ \ol{3}} \qquad 
\tableau[sY]{\ol{3}, \ol{6}, \ol{2} \\ \ol{4}, \ol{1}\\ \ol{5}} \qquad 
\]
\caption{Several tabloids of shape $\langle 3, 2, 1 \rangle$.  The first two tabloids are equal, since they differ only by permuting elements within rows.}
\label{fig:tabloids}
\end{figure}

Several special tabloids will be important in this paper. The first is the \emph{reverse row superstandard tabloid} of shape $\lambda = \langle \lambda_1, \ldots, \lambda_k \rangle$ with \emph{start} at $i$: this is the tabloid whose last row has entries
\[
\ol{i}, \; \ol{i+1}, \; \ldots, \; \ol{i+\lambda_k-1},
\] 
whose next-to-last row has entries
\[
\ol{i+\lambda_k}, \; \ol{i+\lambda_k+1}, \; \ldots, \; \ol{i+\lambda_k+\lambda_{k - 1}-1},
\]
and so on. Thus the third and fourth tabloids in Figure~\ref{fig:tabloids} are the reverse row superstandard tabloids of shape $\langle 3, 2, 1 \rangle$ with starts at $1$ and $3$, respectively. Similarly, one can define the \emph{column superstandard tabloid} of shape $\lambda$ with start at $i$ as the tabloid that has $\ol{i}$ in the first row, $\ol{i+1}$ in the second row, \ldots, $\ol{i+\ell(\lambda)-1}$ in the last row, $\ol{i+\ell(\lambda)}$ in the first row, and so on. The last two tabloids in the figure are column superstandard tabloids of shape $\langle 3, 2, 1 \rangle$ with starts at $1$ and $3$, respectively. 

If $T$ is a tabloid then we denote by $T_i$ the $i$-th row of $T$, viewed as a one-row tabloid or, equivalently, as a subset of $[\ol{n}]$.  Similarly, we denote by $T_{i, i + 1}$ the tabloid consisting of the $i$-th and $(i + 1)$-st rows of $T$, and by $T_{[i, j]}$ the tabloid consisting of all rows of $T$ with indices between $i$ and $j$, inclusive.

\subsection{An analogue of the Robinson-Schensted correspondence}
\label{sec:intro ars}

The paper \cite{cpy} describes a bijection 
\[
\Phi:\widetilde{S_n}\to\dom,
\]
where $\dom$ is the set of triples $(P,Q,\rho)$ such that $P$ and $Q$ are tabloids of the same shape of size $n$ and $\rho$ is an integer vector satisfying certain inequalities depending on $P$ and $Q$. In this section, we give the relevant definitions and outline the construction of the bijection via an algorithm called the Affine Matrix-Ball Construction (AMBC). A detailed summary of these results occupies the first part of \cite{cpy}; while the present work is not completely independent of that part of \cite{cpy}, we hope that the short version provided here is sufficient to understand most of the new results in the present paper.

Define $\Omega$ to be the collection of triples $(P,Q,\rho)$ where $P$ and $Q$ are tabloids of the same shape $\lambda$ of size $n$ and $\rho\in\Z^{\ell(\lambda)}$. Suppose $(P,Q,\rho)\in\Omega$ and $P$ and $Q$ have shape $\lambda$. For $i$ such that $\lambda_{i-1} = \lambda_i$, there are associated integers $r_i(P, Q)$ called \emph{offset constants} (which are defined precisely in Definition~\ref{def:dominance}). If $(P,Q,\rho)$ satisfies the conditions $\rho_i\geqslant\rho_{i-1} + r_i(P, Q)$ then we say that $\rho$ is \emph{dominant} with respect to $(P,Q)$; the pair $(P,Q)$ is usually clear from the context. We define
\[
\dom = \{(P,Q,\rho)\in\Omega : \rho\text{ is dominant}\}.
\]

A second algorithm \cite[\S4]{cpy}, referred to as the \emph{backward algorithm}, gives a surjection $\Psi:\Omega\to\widetilde{S_n}$. By \cite[Thms.~5.11 \& 6.3]{cpy}, the restriction of $\Psi$ to $\dom$ is equal to $\Phi^{-1}$. 

We now briefly describe AMBC, which is closely related to Viennot's geometric construction \cite{viennot_shadow} for the Robinson-Schensted correspondence, called the Matrix-Ball Construction by Fulton \cite{fulton_yt}.  The first step is to produce a special numbering of the balls of the permutation, called a \emph{channel numbering}, as in Figure~\ref{fig:proper numbering}.  This numbering partitions the balls into equivalence classes having the same number; for each class, we form the \emph{zig-zag} having those balls as inner corners, as in Figures~\ref{fig:forward ex 1} and~\ref{fig:forward ex 2}.  We get a new partial permutation from the outer corners of the zig-zags, and we iterate.  At each step of the iteration, we record basic data about the \emph{back corner-posts} of the zig-zags (illustrated by $*$'s in Figures~\ref{fig:forward ex 1} and~\ref{fig:forward ex 2}): their column-indices are recorded in a row of $P$, the row-indices in a row of $Q$, and their \emph{altitude} (a shift parameter distinguishing among the collections occupying the same collection of rows and columns) in an entry of $\rho$.  The remainder of this section is concerned with providing the details behind this summary.

\subsubsection{Channel numberings}
\label{sec:channel numberings}

\begin{figure}
\[
\resizebox{.6\textwidth}{!}{\input{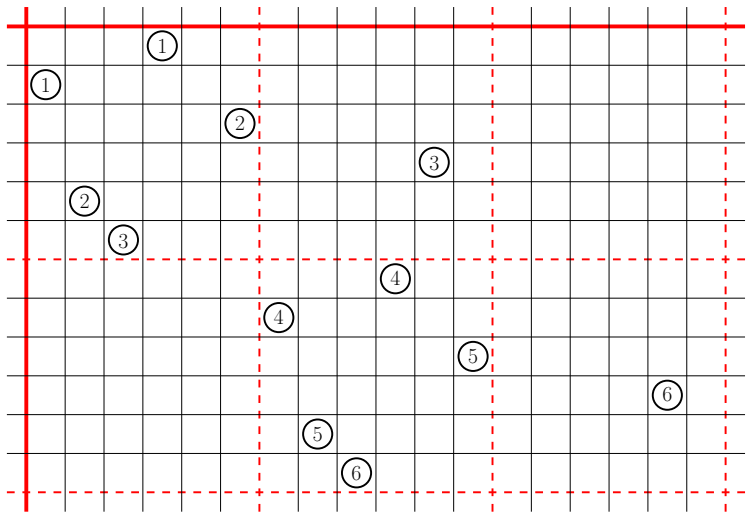}}
\] 
	\caption{A proper numbering for the (extended, affine) permutation $[4,1,6,11,2,3]$. The numbering is semi-periodic with period $3$.}
\label{fig:proper numbering}
\end{figure}

\begin{defn}
Suppose we have a collection $C$ of cells that is invariant under translation by $(n,n)$ and forms a chain in the partial ordering $\leqslant_{SE}$. Then the \emph{density} of $C$ is the number of distinct translation classes in $C$.
\end{defn}
In Figure \ref{fig:proper numbering}, the collection consisting of balls $(2,1), (5,2), (6,3),$ and their translates has density $3$.

\begin{defn}
Suppose $w$ is a partial permutation. Then $C\subseteq \B_w$ is a \emph{channel} if all of the following hold:
\begin{itemize}
\item $C$ is invariant under translation by $(n,n)$,
\item $C$ forms a chain in the partial ordering $\leqslant_{SE}$, and
\item the density of $C$ is maximal among all subsets of $\B_w$ satisfying the first two conditions. 
\end{itemize}
\end{defn}
In Figure \ref{fig:proper numbering}, the collection consisting of balls $(2,1), (5,2), (6,3),$ and their translates is a channel. On the other hand, balls $(1,4)$ and $(3,6)$ are not part of any channel, since no translation-invariant chain of the ordering $\leqslant_{SE}$ with density $3$ passes through both.

A curious fact about channels is that the collection of southwest-most (or northeast-most) balls of a union of two channels is again a channel (see \cite[Prop.~3.13]{cpy} and the comment following it). Thus there is always a southwest-most channel (in Figure \ref{fig:proper numbering}, it is the channel formed by taking the balls $(2,1), (5,2), (6,3)$ and their translates) and a northeast-most channel. In AMBC, we need to pick a distinguished channel at each step; we pick the southwest one.

To perform a step of AMBC requires a special kind of numbering of a permutation, which we define now.

\begin{defn}
For a partial permutation $w$, a function $d:\B_w\to\Z$ is a \emph{proper numbering} if it is
\begin{itemize}
\item monotone: for any $b, b'\in \B_w$, if $b$ lies strictly northwest of $b'$ then $d(b) < d(b')$, and
\item continuous: for any $b'\in \B_w$ there exists $b$ northwest of it with $d(b) = d(b')-1$.
\end{itemize}
\end{defn}
An example of a proper numbering is given in Figure \ref{fig:proper numbering}.

\begin{prop}[{\cite[Prop.~3.4]{cpy}}]
\label{prop: period of proper numbering}
Given a partial permutation $w$, let $m(w)$ denote the density of the channels of $w$.  All proper numberings of $w$ are semi-periodic with period $m(w)$.
\end{prop}

\begin{defn}
\label{defn:path}
Suppose $w$ is a partial permutation and $C$ is a channel of $w$. For a ball $b\in\B_w$ and $k\in\Z^{\geqslant 0}$, a \emph{path} of length $k$ from $b$ to $C$ is a sequence of balls $(b_0, b_1, \ldots, b_k)$ such that $b_0 = b$, $b_k \in C$, and $b_{i + 1}$ lies strictly northwest of $b_{i}$ for all $i$.
\end{defn}

Suppose we have a channel $C$ of some partial permutation $w$. We can ignore all other balls of $w$ and ask for proper numberings of just $C$ itself. It is clear that up to an overall shift there is just one of them, with the balls numbered consecutively by all the integers as one moves from northwest to southeast. We can use this proper numbering of $C$ to produce a proper numbering of all the balls of $w$.

\begin{defn}
Suppose $C$ is a channel of $w$ and $\tilde{d}$ is a proper numbering of $C$. Define the \emph{channel numbering} $d_w^C : \B_w \to \Z$ of $w$ by
\[
d_w^C(b) := \max_{k} \max_{(b_0, b_1, \dots, b_k)} (\tilde{d}(b_k) + k),
\]
where the second maximum is taken over all paths from $b$ to $C$. 
\end{defn}

Because the number of paths is infinite, it seems \emph{a priori} that $\tilde{d}(b_k) + k$ could be unbounded and so $d_w^C(b)$ undefined; however, by \cite[Prop.~3.9]{cpy} this is not the case. Once we know that $d_w^C(b)$ is finite for any ball $b$, it is easy to see that $d_w^C$ is, in fact, a proper numbering.

\begin{rmk}
Given a channel, there is an infinite family of channel numberings associated with it: they differ by shifting the numbers of all balls by the same amount.  In this paper, the distinctions between these numberings never matter, and we will use the phrase ``the channel numbering'' to refer to a (locally fixed, but) arbitrary choice among them.
\end{rmk}

\subsubsection{Zig-zags}

\begin{figure}
\centering
\resizebox{.8\textwidth}{!}{\input{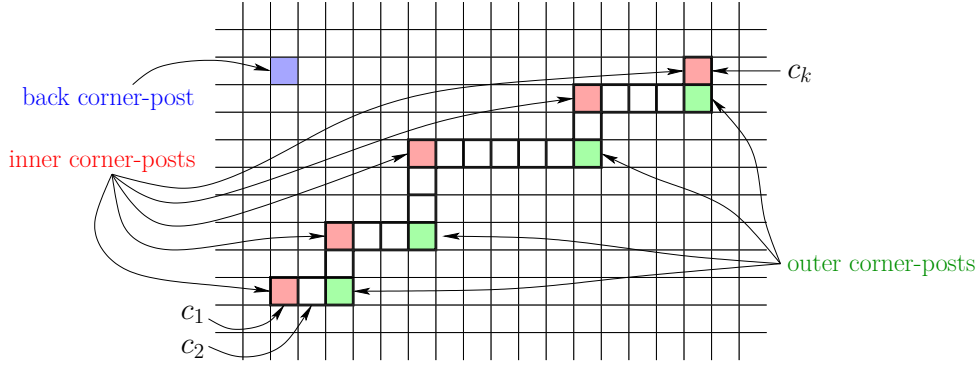}} 
\caption{The different corner-posts of a zig-zag.}
\label{fig:zig-zag posts}
\end{figure}

The definition of AMBC involves certain collections of cells called \emph{zig-zags}.

\begin{defn}
A \emph{zig-zag} is a non-empty sequence $(c_1, c_2, \ldots, c_k)$ of cells such that both of the following hold:
\begin{itemize}
\item for $1\leqslant i< k$, $c_{i+1}$ is adjacent to and either directly north or directly east of $c_i$, and
\item if $k\geqslant 2$, then $c_2$ is directly east of $c_1$ and $c_k$ is directly north of $c_{k-1}$.
\end{itemize}
Given a zig-zag $Z = (c_1, c_2, \ldots, c_k)$, we say that
\begin{itemize}
\item the \emph{back corner-post} is the cell in the same column as $c_1$ and the same row as $c_k$,
\item the \emph{inner corner-posts} are the cells of $Z$ such that no cell directly north or directly west of them is in $Z$, and
\item if $k \geqslant 2$, the \emph{outer corner-posts} are the cells of $Z$ such that no cell directly south or directly east of them is in $Z$; if $k = 1$ then there are no outer corner-posts of $Z$.
\end{itemize}
\end{defn}

These definitions are illustrated in Figure \ref{fig:zig-zag posts}.  Notice that the inner and outer corner-posts are always part of the zig-zag. The back-corner post is not usually part of the zig-zag; the only exception is a degenerate zig-zag with one cell, whose back corner-post coincides with its inner corner-post (and which has no outer corner-post).

The zig-zags that appear in AMBC are attached to a proper numbering of a permutation.
\begin{defn}
Given a proper numbering $d$ of a partial permutation $w$, the collection of \emph{zig-zags corresponding to $d$} is the collection $\{Z_i\}_{i\in\Z}$ where $Z_i$ is the unique zig-zag whose inner corner-posts are precisely the balls of $w$ labeled $i$ by $d$.
\end{defn}

The number of translation classes of zig-zags is the number $m(w)$ appearing in Proposition~\ref{prop: period of proper numbering}.

\subsubsection{Streams}
\label{sec:streams}

\begin{figure}
\centering
\resizebox{.4\textwidth}{!}{\input{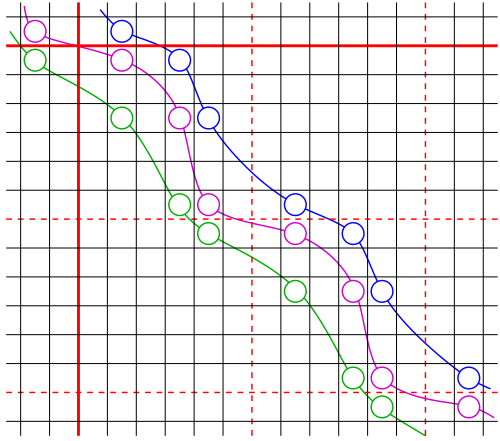}}
\caption{Streams of altitude $-1$ (green), $0$ (purple), and $1$ (blue) for $A = \{\ol{1},\ol{3},\ol{6}\}$ and $B = \{\ol{2},\ol{4},\ol{5}\}$.}
\label{fig:streams}
\end{figure}

The definition of AMBC involves a certain collection of cells called a \emph{stream}:
a set that is invariant under translation by $(n,n)$ and forms a chain in the partial ordering $\leqslant_{SE}$.
Since streams are translation-invariant, if a stream has a cell with column $k$, then it has cells in all the columns of $\ol{k}$. 

\begin{defn}
For any cell $c = (c_1, c_2)$, let $D(c)$ be the \emph{block diagonal} of $c$, 
\[
D(c) := \ceil{\frac{c_2}{n}} - \ceil{\frac{c_1}{n}}.
\] 
\end{defn}
Thus, if one of the translates of $c$ is in $[n]\times [n]$ then $D(c) = 0$; if a translate is in $\{n+1,n+2,\ldots,2n\}\times [n]$ then $D(c) = -1$; etc. 

\begin{defn}
For any translation-invariant collection $X$ of cells, we define $D(X) = \sum_{x\in Y} D(x)$, where $Y$ is a subset of $X$ containing one representative of each translation class.  If $X$ is a stream, we call $D(X)$ the \emph{altitude} of $X$.
\end{defn}

In \cite[\S3.4]{cpy}, a different definition was given of altitude of a stream; however, it is an easy exercise to see that they are equivalent.  The following result shows that the set of rows, set of columns, and altitude uniquely specify a stream.

\begin{prop}[{essentially \cite[Lem.~3.23 \& Prop.~3.24]{cpy}}]
Given two subsets $A, B$ of $[\ol{n}]$ of the same size and an integer $r$, there is a unique stream of altitude $r$ whose balls lie in rows indexed by $A$ and in columns indexed by $B$.
\end{prop}
This stream will be denoted $\st_r(A, B)$; the triple $A, B, r$ is called its \emph{defining data}.

\begin{ex}
 Let $n=6$, $A = \{\ol{1},\ol{3},\ol{6}\}$, and $B = \{\ol{2},\ol{4},\ol{5}\}$. The streams $\st_{-1}(A,B)$, $\st_{0}(A,B)$, and $\st_{1}(A,B)$ are shown in Figure \ref{fig:streams}. The stream $\st_{0}(A,B)$ is the unique choice with all of its cells in translations of $[n]\times [n]$; $\st_{1}(A,B)$ is obtained from $\st_0(A, B)$ by moving the cell in every row from its column east into the next available column in $\bigcup B$.
\end{ex}

The streams that appear in AMBC come from a proper numbering of a permutation.
\begin{defn}
Suppose $w$ is a partial permutation and $d:\B_w\to\Z$ is the southwest channel numbering; let $\{Z_i\}_{i\in\Z}$ be the collection of zig-zags corresponding to $d$. For each $i$, let $b_i$ be the back corner-post of $Z_i$. Then $\st(w) := \{b_i\}_{i\in\Z}$ is a stream.
\end{defn}

\subsubsection{The algorithm}

\begin{figure}
\centering
\resizebox{.7\textwidth}{!}{\input{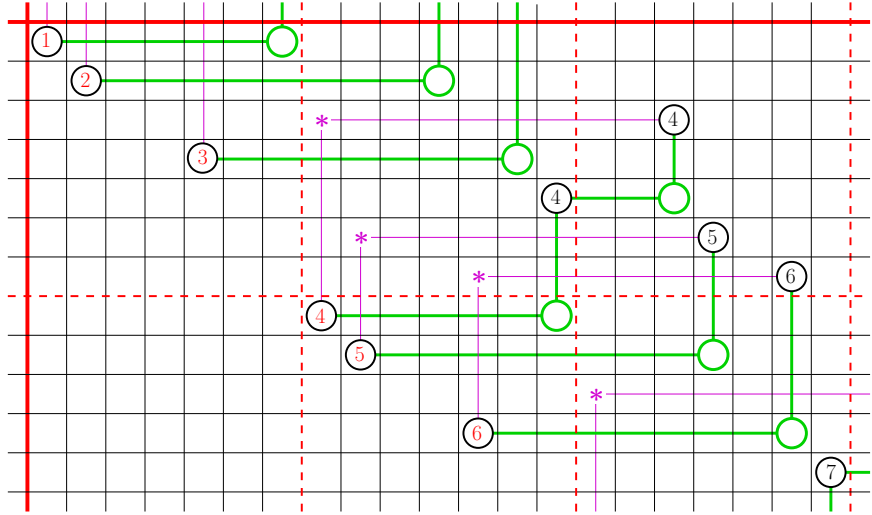}}
\caption{First step of AMBC for $w = [1,2,17,5,14,18,20]$. The southwest channel numbering is shown with the balls in the channel numbered in red. The balls of $\fw{w}$ are shown in green. The cells of $\st(w)$ are marked by $*$'s. }
\label{fig:forward ex 1}
\end{figure}

We now define the algorithm AMBC, giving the map $\Phi:\widetilde{S_n}\to\Omega$. 
\begin{defn}
Suppose $w$ is a partial permutation and $d:\B_w\to\Z$ is the southwest channel numbering. Define $\fw{w}$ to be the permutation whose balls are located at the outer corner-posts of all the zig-zags corresponding to $d$. 
\end{defn}
The algorithm is as follows.
\begin{itemize}
 \item Input $w\in \widetilde{S_n}$.
 \item Initialize $(P,Q,\rho)$ to $(\varnothing,\varnothing,\varnothing)$.
 \item Repeat until $w$ is the empty partial permutation:
\begin{itemize} 
 \item Record the defining data of $\st(w)$ in the next row of $P$, $Q$, and $\rho$.
 \item Reset $w$ to $\fw{w}$. 
\end{itemize}
 \item Output $(P,Q,\rho)\in\Omega$.
\end{itemize}

\begin{ex}
Let $n = 7$ and consider the permutation $w=[1,2,17,5,14,18,20]$.  In the first step of AMBC, shown in Figure~\ref{fig:forward ex 1}, $w$ is numbered in the SW channel numbering; the stream consisting of the back corner-posts for the zig-zags corresponding to this numbering has elements in rows $\ol{3}$, $\ol{6}$, and $\ol{7}$ and columns $\ol{1}$, $\ol{2}$, and $\ol{5}$, and has altitude $3$.  Thus, the first row of $P$ is $\tableau[sY]{\ol{1}, \ol{2}, \ol{5}}$, the first row of $Q$ is $\tableau[sY]{\ol{3}, \ol{6}, \ol{7}}$, and the first row of $\rho$ is $3$.  The second step is shown in Figure~\ref{fig:forward ex 2}, and produces second rows $\tableau[sY]{\ol{4}, \ol{6}, \ol{7}}$, $\tableau[sY]{\ol{2}, \ol{4}, \ol{5}}$ and $3$ for $P$, $Q$, and $\rho$.  Finally, the third step begins with the outer corner-posts from the second step; these form a partial permutation whose balls are exactly the translates of $(1, 10)$.  Thus the stream for this permutation also consists of translates of the single cell $(1, 10)$, and so at the end of this step we set the third row of $P$ to $\tableau[sY]{\ol{3}}$, the third row of $Q$ to $\tableau[sY]{\ol{1}}$, and the third row of $\rho$ to $1$.  The resulting triple $(P,Q,\rho)$ is
\[
\left(
\quad
\tableau[sY]{\ol{1}&\ol{2}&\ol{5}\\\ol{7}&\ol{4}&\ol{6}\\\ol{3}}
\quad,\qquad
\tableau[sY]{\ol{7}&\ol{3}&\ol{6}\\\ol{2}&\ol{4}&\ol{5}\\\ol{1}}
\quad,\qquad
\begin{pmatrix} 3 \\ 3 \\ 1 \end{pmatrix}
\;
\right).
\]
\end{ex}

\begin{figure}
\centering
\resizebox{.7\textwidth}{!}{\input{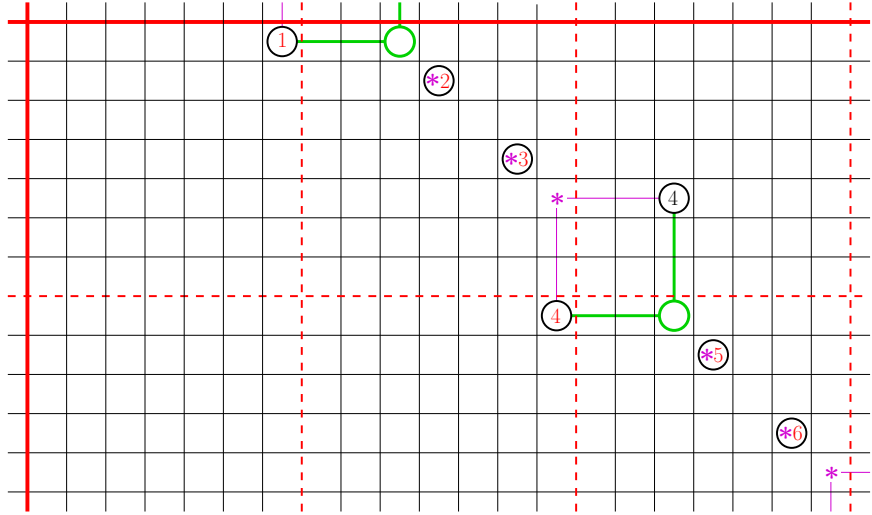}}
\caption{Second step of AMBC for $w=[1,2,17,5,14,18,20]$.}
\label{fig:forward ex 2}
\end{figure}


\subsection{Zig-zags from a proper numbering}

The following simple facts about how the zig-zags associated to a proper numbering lie in the plane will occur repeatedly in the arguments that follow.  
Both halves of the result are illustrated in Figure \ref{fig:zig-zag positions}.
\begin{figure}
\centering
\resizebox{.9\textwidth}{!}{\input{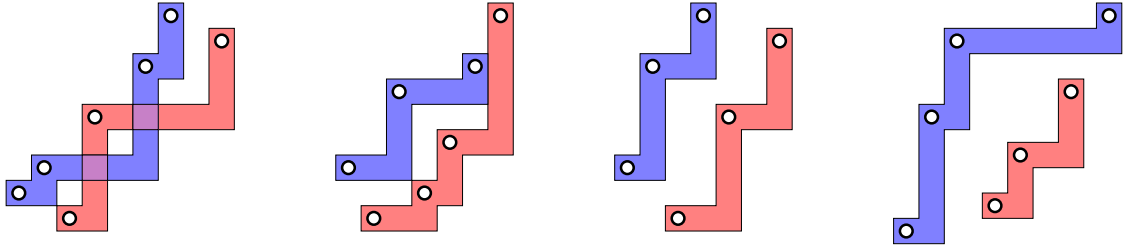}}
\caption{The blue zig-zag has balls of $w$ with lower value of $d$ than the red zig-zag. The two zig-zag positions shown on the left cannot occur: the first satisfies condition (b) of Proposition \ref{prop:zig-zags} but violates condition (a), while the second satisfies condition (a) but violates condition (b). The two positions on the right are valid.}
\label{fig:zig-zag positions}
\end{figure}

\begin{prop}
\label{prop:zig-zags}
Suppose that $w$ is a partial permutation, with balls labeled by
a proper numbering $d$.  Divide the balls into zig-zags according to $d$.
\begin{compactenum}[(a)]
\item If $b = (k, w(k))$ and $c = (\ell, w(\ell))$ are consecutive balls in a zig-zag of $w$ and $k < \ell$, then there are no balls of $w$ having larger value of $d$ northwest of the cell $(\ell, w(k))$.
\item If $b$ is the northeast (resp.\ southwest) ball of one zig-zag and $c$ is the northeast (resp.\ southwest) ball of another zig-zag and $d(b) < d(c)$, then $c$ lies strictly south (resp.\ east) of $b$.
\end{compactenum}
\end{prop}
\begin{proof}
For part (a), suppose for contradiction that there is such a ball $a$.  By continuity of $d$ (possibly applied many times), there is a ball northwest of $a$ with label $d(b) = d(c) < d(a)$.  Since $b$ and $c$ are given as consecutive balls in their zig-zag, this ball cannot be in the rectangle having $b$ and $c$ as vertices; however, by the monotonicity of $d$ it also cannot be northwest of $b$ or of $c$.  But every cell northwest of $(\ell, w(k))$ falls into one of these three sets; this is a contradiction, so no such $a$ exists, as claimed.

For part (b), suppose that $c$ is the northeast ball in its zig-zag and $d(b) < d(c)$.  By continuity of $d$, there is some ball $b'$ of $w$ northwest of $c$ such that $d(b') = d(b)$.  The northeast ball in this zig-zag is at least as far north as $b'$, hence is north of $c$, as claimed.
\end{proof}

\section{Knuth moves, descents, and inverses under AMBC}
\label{sec:Knuth moves}

The main result of this section is Theorem~\ref{thm:Knuth move action}, which describes precisely how making a Knuth move to a permutation affects its image under AMBC.  It was already shown in \cite[Lem.~12.3]{cpy} that the $P$-tabloid does not change; we furthermore show that the $Q$-tabloid changes by a Knuth move and that $\rho$ changes in a predictable way.  This result is crucial to our analysis of the Kazhdan-Lusztig dual equivalence graph in Sections~\ref{sec:monodromy} and~\ref{sec:charge}.  

Theorem~\ref{thm:Knuth move action} is also interesting in that it provides evidence that AMBC is the ``correct'' analogue of the Robinson-Schensted correspondence, as it interacts in the appropriate ways with other combinatorial constructions.  The first steps towards the proof also are of this form, showing that AMBC respects inverses and descent sets of permutations in the same way as the Robinson-Schensted correspondence.

\subsection{Inverses}
\label{sec:inverses}

 In the finite case, it is a classical result of Sch\"utzenberger \cite{Schutzenberger} (or see \cite[Thm.~7.13.1]{EC2}) that taking the inverse of a permutation $w$ exchanges the insertion and recording tableaux in its image under the Robinson-Schensted correspondence. Here we show the analogous result for AMBC.

\begin{prop}
\label{prop:inverses}
Suppose $\Phi(w) = (P,Q,\rho)$. Then $\Phi(w^{-1}) = (Q,P, (-\rho)')$, where $(-\rho)'$ is the dominant representative of $-\rho$ in the fiber under $\Psi(Q, P, -)$.
\end{prop}
\begin{proof}
Inverting a permutation reflects its matrix across the main diagonal; this switches rows for columns, switches ``SW'' for ``NE,'' and switches ``E'' for ``S.'' Moreover, given a stream, reflecting it in the main diagonal negates the altitude of the stream. Thus, when we apply each step of AMBC to the inverse permutation $w^{-1}$, \emph{using the NE channel numbering} (instead of the usual SW channel numbering) at each step, we recover exactly the same steps as applying AMBC to $w$ using the SW channel numbering, with the following adjustments: first, because the roles of rows and columns are switched, data recorded for $w$ in the $P$-tabloid is recorded for $w^{-1}$ in the $Q$-tabloid and vice-versa; and second, because altitudes of streams are negated, each entry of $\rho(w^{-1})$ is the negation of the corresponding entry in $\rho(w)$. It follows from a strong form of the inverse relationship between $\Phi$ and $\Psi$ \cite[Prop.~5.2]{cpy} that $\Psi(Q,P,-\rho) = w^{-1}$.  Finally, we have by \cite[Thm.~6.3]{cpy} that $\Phi(w^{-1}) = \Phi(\Psi(Q, P, -\rho)) = (Q, P, (-\rho)')$ where $(-\rho)'$ is the dominant representative of $-\rho$.
\end{proof}

\begin{rmk}
\label{rmk:not obvious}
It is not obvious from the definition of dominance that applying the operation $(P, Q, \rho) \mapsto (Q, P, (-\rho)')$ twice returns the original triple (as it must do, since $(w^{-1})^{-1} = w$): the operation ``take the dominant representative'' depends on $P$ and $Q$ in a nontrivial way.  This oddity is explained in Section \ref{sec:charge intro}, following the description of the offset constants in Theorem~\ref{thm:dominance constants}.
\end{rmk}

\subsection{Descent sets and the $\tau$-invariant}

In the finite case, the descent set (appropriately defined) of the $Q$-tableau of $w$ is equal to the (right) descent set of $w$ (i.e., the set of integers $i$ such that $w(i) > w(i+1)$), and the descent set of the $P$-tableau is the descent set of $w^{-1}$ \cite[Lem.~7.23.1]{EC2}.  Here we show the analogous result for affine permutations and AMBC.

\begin{defn}
For a partial permutation $w$, the \emph{right descent set} $R(w)$ of $w$ is defined by 
\[
R(w) = \{\ol{i}\in[\ol{n}] : w(i) > w(i+1)\}.
\]
Similarly, the \emph{left descent set} $L(w)$ is defined by
\[
L(w) = \{\ol{i}\in[\ol{n}] : w^{-1}(i) > w^{-1}(i+1)\}.
\]
\end{defn}
So $\ol{i}$ is in $R(w)$ (resp.~$L(w)$) precisely when the ball in the $(i+1)$-st row (resp.~column) is west (resp.~north) of the ball in the $i$-th row (resp.~column).  These definitions agree with the usual notions from Coxeter theory.

We call the analogue of the descent set for tabloids the \emph{$\tau$-invariant}, in reference to Vogan's (generalized) $\tau$-invariant \cite{vogan-tau}.  
\begin{defn}
For a tabloid $T$ filled with all the elements of $[\ol{n}]$, define the \emph{$\tau$-invariant} by
\[
\tau(T) := \{\ol{i}\in [\ol{n}]: \ol{i} \text{ lies in a strictly higher row of } T \text{ than } \ol{i + 1}\}.
\]
\end{defn}
\begin{ex}
\label{ex:tabloid tau}
The tabloid 
\[
T := \tableau[sY]{\ol{2}, \ol{9}, \ol{5}\\ \ol{8}, \ol{7}, \ol{6}\\ \ol{3}, \ol{1}\\ \ol{4}}
\]
has $\tau(T) = \{\ol{2}, \ol{3}, \ol{5}, \ol{9}\}$.
\end{ex}

\begin{prop}
\label{prop:descents}
For any permutation $w$, $L(w) = \tau(P(w))$ and $R(w) = \tau(Q(w))$.
\end{prop}

\begin{proof}
It is sufficient to show $L(w) = \tau(P(w))$; the other statement follows from Proposition~\ref{prop:inverses}. 

First, suppose $\ol{i}\in L(w)$, so that the ball $b$ in the $i$-th column is south of the ball $c$ in the $(i+1)$-st column.  It follows that every ball strictly northwest of $c$ is also northwest of $b$, and so by the continuity and monotonicity of the southwest channel numbering $d$ we have $d(c)\leqslant d(b)$.  

Consider the application of a single forward step of AMBC to $w$.  The numbering $d$ induces a set of zig-zags.  Since $b$ is west of $c$, we have by Proposition~\ref{prop:zig-zags}(b) that $c$ is not the southwest ball in its zig-zag, and so there is a ball of $\fw{w}$ in the $(i+1)$-st column; moreover, by Proposition~\ref{prop:zig-zags}(a) this ball is north of $b$. Either $b$ is the southwest ball of its zig-zag, or not. In the first case, $\ol{i}$ appears in the first row of $P(w)$ while $\ol{i+1}$ appears in a lower row. In the second case, neither $\ol{i}$ nor $\ol{i + 1}$ appears in the first row of $P(w)$, and the ball in the $i$-th column of $\fw{w}$ is south of the ball in the $(i+1)$-st column; thus we can repeat the argument until we end up in the first case.

Conversely, if $\ol{i}\notin L(w)$ then the ball $b$ in the $i$-th column is north of the ball $c$ in the $(i+1)$-st column. In this case $d(b) < d(c)$ by monotonicity. If $c$ is the southwest ball of its zig-zag then $\ol{i + 1}$ appears in the first row of $P(w)$; thus $\ol{i} \notin \tau(P(w))$, as claimed.  Alternatively, $c$ is not the southwest ball of its zig-zag. Then the ball $c'$ southwest of $c$ in its zig-zag is west of $b$, and by Proposition~\ref{prop:zig-zags}(b) $b$ is not the southwest ball of its zig-zag. Moreover, it follows from Proposition~\ref{prop:zig-zags}(a) that the ball of $\fw{w}$ in the $i$-th column must be north of the ball in the $(i+1)$-st column. Thus we can repeat the argument to see that $\ol{i+1}$ cannot lie in a lower row of $P(w)$ than $\ol{i}$, as claimed.
\end{proof}

\subsection{Knuth moves}
Graphs whose vertices are combinatorial objects such as permutations or tableaux and whose edges are analogues of Knuth moves occur frequently in the study of the symmetric group -- see, for example, \cite{hrt, assaf2, blasiak-fomin}.  In the affine setting, we deal with two kinds of such graphs, one on the set of permutations and one on the set of tabloids.

\subsubsection{Knuth moves for permutations and tabloids}
\label{sec:knuth on perms}

In the finite setting, Knuth moves are certain elementary operations on finite permutations, interchanging adjacent entries if one of the neighboring entries is numerically between them -- see \cite[Ch.~7 App.~1]{EC2}.  The definition of Knuth moves in the affine case is extremely similar.

\begin{defn}
\label{def:Knuth move perm}
Let $i\in\Z$. Two partial permutations $w$ and $w'$ are \emph{connected by a Knuth move at position $\ol{i}$} if all of the following hold:
\begin{itemize}
\item for all $j$ such that $j\equiv i\pmod{n}$, we have $w'(j) = w(j+1)$ and $w'(j+1) = w(j)$;
\item for all $j$ such that $j\not\equiv i\pmod{n}$, $j\not\equiv i+1\pmod{n}$, we have  $w'(j) = w(j)$; and
\item at least one of $w(i+2)$ and $w(i-1)$ is numerically between $w(i)$ and $w(i+1)$.
\end{itemize}
\end{defn}
For example, the permutation $w = [3, 1, 2]$ is connected by a Knuth move to $[1, 3, 2]$ (because $w(3) = 2$ has value between $w(1) = 3$ and $w(2) = 1$) and to $[-1, 1, 6]$ (because $w(2) = 1$ has value between $w(0) = -1$ and $w(1) = 3$) but not to $[3, 2, 1]$ (because neither $w(1) = 3$ nor $w(4) = 6$ have value between $w(2) = 1$ and $w(3) = 2$).

We call the translation class of the ball $(i+2,w(i+2))$ or $(i-1,w(i-1))$ mentioned in the last part of the definition a \emph{witness} to the Knuth move.  

An alternative way to describe the condition when exchanging the $i$-th and $(i+1)$-st entries of a finite permutation constitutes a Knuth move is via the right descent set. It is not difficult to see that the condition on the witness precisely means that the right descent sets of of the two permutations are incomparable under the containment partial ordering. Similarly, we can give an alternate condition in the affine case.

\begin{prop}
\label{prop:Knuth on descents}
Suppose $w$ is a partial permutation such that $w(i), w(i+1)$, and $w(i+2)$ are defined, and precisely one of $\ol{i}$ and $\ol{i+1}$ is in $R(w)$. Then there exists a unique permutation $w'$ connected to $w$ by a Knuth move whose right descent set contains precisely the other of $\ol{i}$ and $\ol{i+1}$.
\end{prop}
\begin{proof}
If exactly one of $\ol{i}, \ol{i + 1}$ belongs to $R(w)$ then $w(i + 1)$ is either the largest or smallest of $w(i), w(i + 1), w(i + 2)$, and so one of $w(i)$, $w(i + 2)$ lies numerically between $w(i + 1)$ and the other.  Thus a Knuth move that switches $w(i + 1)$ with one of $w(i), w(i + 2)$ is possible, witnessed by the other of $w(i), w(i + 2)$.  It is straightforward to verify that this has the desired effect on the descent set $R$, and that no other Knuth move can have this effect.
\end{proof}


In the finite case, Knuth moves preserve the Robinson-Schensted insertion tableau and induce transpositions of entries (which are also called Knuth moves) in the recording tableau. We give the corresponding definition of Knuth moves on tabloids.

\begin{defn}
Two tabloids $T$ and $T'$ are \emph{connected by a Knuth move} if $T'$ is obtained from $T$ by exchanging $\ol{i}$ and $\ol{i+1}$ and $\tau(T)$ and $\tau(T')$ are incomparable with respect to inclusion. 
\end{defn}

\begin{ex}
The tabloid 
\[
T' := \tableau[sY]{\ol{2}, \ol{1}, \ol{5}\\ \ol{8}, \ol{7}, \ol{6}\\ \ol{3}, \ol{9}\\ \ol{4}}
\]
is connected by a Knuth move to the tabloid $T$ from Example \ref{ex:tabloid tau}: it is obtained by exchanging $\ol{9}$ and $\ol{10} = \ol{1}$, and 
$\tau(T') =  \{\ol{2}, \ol{3}, \ol{5}, \ol{8}\}$ is not contained in and does not contain $\tau(T) = \{\ol{2}, \ol{3}, \ol{5}, \ol{9}\}$.
\end{ex}

\subsubsection{Knuth moves and AMBC}
In this section we describe how a Knuth move on permutations looks after taking images under AMBC.

\begin{thm}
\label{thm:Knuth move action}
Suppose $w$ is a partial permutation and $w'$ differs from $w$ by a Knuth move. Then $P(w) = P(w')$ and
$Q(w')$ differs from $Q(w)$ by a Knuth move.
Moreover, if the Knuth move on the $Q$-tabloids exchanges $\ol{i}$ and $\ol{i+1}$ for some $\ol{i}\neq \ol{n}$ then $\rho(w) = \rho(w')$; if instead $\ol{i} = \ol{n}$ is in row $k$ of $Q(w)$ while $\ol{i+1} = \ol{1}$ is in row $k'$, then $\rho(w')$ differs from $\rho(w)$ by subtracting $1$ from row $k$ and adding $1$ to row $k'$. 
\end{thm}

\begin{ex}
\label{ex:Knuth move pre-example}
Consider the permutation $w = [1,4,6,2,5,3]$; under AMBC it corresponds to 
\[
P = Q = \tableau[sY]{\ol{1}, \ol{2}, \ol{3} \\ \ol{4}, \ol{5} \\ \ol{6}},\quad \rho = \begin{pmatrix}0\\0\\0\end{pmatrix}.
\]
Since $w(4) < w(2) < w(3)$, we can apply a Knuth move at position $3$.  This yields $[1,4,2,6,5,3]$, which corresponds under AMBC to the triple
\[
P = \tableau[sY]{\ol{1}, \ol{2}, \ol{3} \\ \ol{4}, \ol{5} \\ \ol{6}},\quad Q' = \tableau[sY]{\ol{1}, \ol{2}, \ol{4} \\ \ol{3}, \ol{5} \\ \ol{6}},\quad \rho = \begin{pmatrix}0\\0\\0\end{pmatrix}.
\]
\end{ex}
\begin{ex}
\label{ex:Knuth move example}
Again consider $w = [1,4,6,2,5,3]$.  Since $w(6) < w(5) < w(7) = 7$, we can apply a Knuth move to $w$ at position $6$.  This yields $[-3,4,6,2,5,7]$, which corresponds under AMBC to the triple
\[
P = \tableau[sY]{\ol{1}, \ol{2}, \ol{3} \\ \ol{4}, \ol{5} \\ \ol{6}},\quad Q'' = \tableau[sY]{\ol{6}, \ol{2}, \ol{3} \\ \ol{4}, \ol{5} \\ \ol{1}},\quad \rho' = \begin{pmatrix}1\\0\\-1\end{pmatrix}.
\]
\end{ex}

\begin{rmk}
\label{rmk:change of streams under Knuth}
The condition on how the weight changes may be formulated in terms of the streams encoded by the rows. Suppose that $Q$ and $Q'$ differ by a Knuth move exchanging $\ol{i}$ and $\ol{i+1}$, that $\ol{i}$ is in row $k$ of $Q$, and that and $\ol{i+1}$ is in row $k'$ of $Q$. Suppose $\rho'$ differs from $\rho$ as described in the theorem. Then the stream $S'$ represented by row $k$ of $(P,Q',\rho')$ differs from the stream $S$ represented by row $k$ of $(P,Q,\rho)$ by shifting the elements with row indices in $\ol{i}$ south by one cell, and similarly the stream $T'$ encoded by row $k'$ of $(P, Q', \rho')$ differs from the stream $T$ represented by row $k'$ of $(P,Q,\rho)$ by shifting the elements with row indices in $\ol{i + 1}$ north by one cell.

\begin{figure}
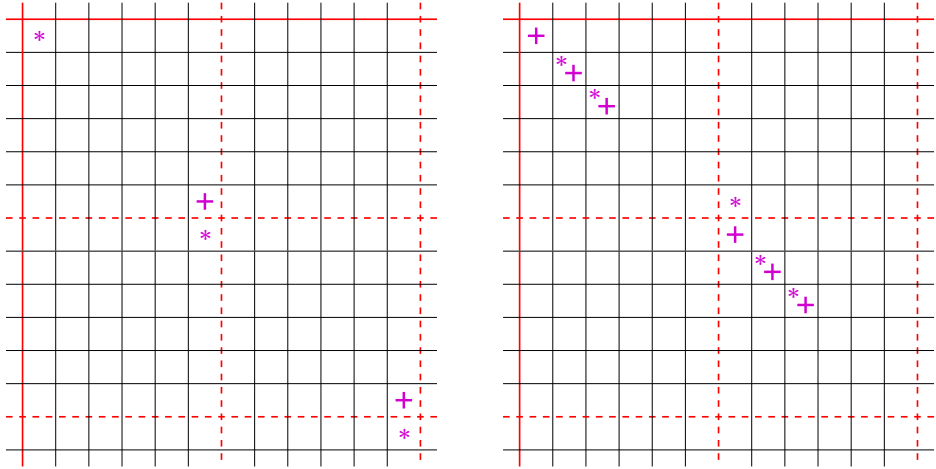

\centering
\resizebox{.35\textwidth}{!}{\input{figures/knuth_on_streams.pspdftex}}\qquad \resizebox{.35\textwidth}{!}{\input{figures/knuth_on_streams_2.pspdftex}}
\caption{Streams for the permutation $w = [1, 4, 6, 2, 5, 3]$ ($+$'s) and its image $[-3, 4, 6, 2, 5, 7]$ ($*$'s) under a Knuth move; see Remark~\ref{rmk:change of streams under Knuth}.  On the left, the streams $S, S'$ encoded by the third row of the images under AMBC. On the right, the streams $T, T'$ encoded by the first row.}
\label{fig:knuth on streams}
\end{figure}

For example, with $w$ as in Example~\ref{ex:Knuth move example}, the streams $S$ and $S'$ (shown on the left of Figure~\ref{fig:knuth on streams}) are encoded by the third row of the triples, while the streams $T$ and $T'$ (shown on the right of the figure) are encoded by the first row of the triples. The stream $S$ contains only one translation class of cells, in row $\ol{n}$; to obtain the stream $S'$, one shifts each of those elements one cell south. The stream $T$ contains three translation classes of cells; to obtain the stream $T'$, one shifts the elements in rows indexed by $\ol{1}$ one cell north.
\end{rmk}

\subsubsection{The proof of Theorem \ref{thm:Knuth move action}}
\label{sec:knuth proof}
We fix some notation for the duration of this section.  Let $w$ be a partial permutation and let $w'$ be the permutation obtained from it by a Knuth move, affecting balls in rows $\ol{k}$ and $\ol{k + 1}$. Without loss of generality, assume $w(k+1) < w(k)$. Let $b = (k+1, w(k+1))\in\B_w$, $c = (k, w(k))\in\B_w$, $b' = (k, w(k+1))\in\B_{w'}$, and $c' = (k+1, w(k))\in\B_{w'}$. This situation is shown in Figure \ref{fig:knuth positions}.

\begin{figure}
\centering
\resizebox{.4\textwidth}{!}{\input{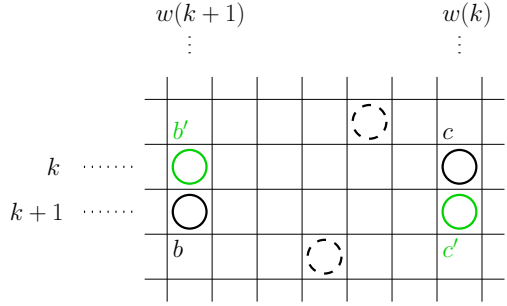}}
\caption{The positions of the balls in Section \ref{sec:knuth proof}. At least one of the dashed balls must be present as a witness for the Knuth move.}
\label{fig:knuth positions}
\end{figure}

Let $d$ be the southwest channel numbering of $w$.  Since every ball strictly northwest of $b$ is also northwest of $c$, we have by the properness of the numbering that $d(c) \geqslant d(b)$. Define a semi-periodic numbering $d'$ of $w'$ as follows: for $a\in\B_{w'}$,
\begin{align*}
d'(a) & = d(a) \quad\qquad \text{ if $a$ is not a translate of } b' \text{ or } c',\\
d'(b') & = d(b), \quad\qquad \text{and}\\
d'(c') & = \begin{cases}
           d(c) & \text{if } d(c) > d(b),\\
           d(c) + 1 & \text{if } d(c) = d(b),
           \end{cases}
\end{align*}
and extend the numbering to all other balls by semi-periodicity.  
In other words, to get from $w$ to $w'$, we shift $b$ (together with its translates) one cell north and shift $c$ one cell south; the numbering $d'$ is obtained by keeping the numbering of the balls being shifted, unless $d(b) = d(c)$ (in which case keeping the numbering would yield a non-monotone numbering).  As the next lemma shows, this adjustment results in a proper numbering.

\begin{defn}
The proofs that follow are divided into cases. We denote by {\bf A} (for \emph{after}) the case when $w(k+2)$ is between $w(k)$ and $w(k+1)$, and by {\bf B} (for \emph{before}) the case when $w(k-1)$ is between $w(k)$ and $w(k+1)$. We denote by {\bf S} (for \emph{same}) the case when $d(b) = d(c)$, and by {\bf D} (for \emph{different}) the case when $d(c) > d(b)$. 
\end{defn}
{\it A priori} the preceding definition allows four sub-cases.  However, the case {\bf BS} is impossible when $d$ is proper: using Proposition~\ref{prop:zig-zags}, one can show that for a witness $a$ in the row before $b$ and $c$, one must have $d(b) \leqslant d(a) < d(c)$, and this contradicts $d(b) = d(c)$.

\begin{lemma}
\label{lemma: a certain numbering is proper}
The numbering $d'$ described above is proper.
\end{lemma}
\begin{proof}
We must check the monotonicity and continuity of the function $d'$.  In case {\bf D}, this is straightforward: no ball numbers change and no relation $a \leqslant_{SE} a'$ between two balls is destroyed, so $d'$ is continuous; and the only newly created relations in the SE order are translates of $b' >_{SE} c'$, and these are compatible with monotonicity since $d'(b') = d(b) < d(c) = d'(c')$.

Suppose instead that we are in case {\bf AS}. Let $a = (k+2, w(k+2))$.  Every ball strictly northwest of $a$ is strictly northwest of $b$, strictly northwest of $c$, or is $b$ itself.  Since $d(b) = d(c)$, the largest label of such a ball is $d(b)$, and thus $d'(a) = d(a) = d(b) + 1 = d'(c')$.  Then to check continuity and monotonicity of $d'$ it suffices to observe that every ball southeast of $c'$ in $w'$ is also southeast of $a$.
\end{proof}

\begin{defn}
We denote by $f$ the natural bijection $f:\B_w\to\B_{w'}$ that commutes with translation by $(n,n)$, has $f(b) = b'$ and $f(c) = c'$, and is the identity on balls that do not lie in rows congruent to $k$ or $k+1$ modulo $n$.
\end{defn}

\begin{lemma}
\label{lemma:two channels intersect}
If $C$ is the southwest channel of $w$ then $f(C)$ is a channel of $w'$ that intersects the southwest channel of $w'$.
\end{lemma}
\begin{proof}
Let $C_1$ be any channel of $w$.  For any two balls $a, a'$ in $\B_w$ such that $a <_{SE} a'$, one also has $f(a) <_{SE} f(a')$.  Thus $C'_1 := f(C_1)$ forms a chain in the southeast ordering.  By Lemma~\ref{lemma: a certain numbering is proper}, the numbering $d'$ is proper, so by Proposition~\ref{prop: period of proper numbering}  $w'$ has the same channel density as $w$.  Since $C_1$ has this density, $C'_1$ does as well, and so $C'_1$ is a channel for $w'$.  

Let $C'_{SW}$ be the southwest channel of $w'$.  By \cite[Prop.~3.13]{cpy} and the surrounding discussion, it follows that every ball of $C'_{SW}$ is weakly southwest of some ball in $C'_1$.  Suppose furthermore that $C'_1 \cap C'_{SW} = \varnothing$, so that we may replace the word ``weakly'' in the preceding sentence with ``strictly.''  We will show that there is a channel of $w$ that contains a ball strictly southwest of some ball of $C_1$.

Suppose first that $C'_{SW}$ contains at most one of $b'$ and $c'$.  Then $f^{-1}(C'_{SW})$ is a channel of $w$.  The map $f^{-1}$ is order-preserving for the order $<_{SW}$, so in this case every ball of $f^{-1}(C'_{SW})$ is strictly southwest of some ball of $C_1 = f^{-1}(C'_1)$, as desired.

Suppose instead that $C'_{SW}$ contains both $b'$ and $c'$. Then $C'_1$ contains neither, and so $C'_1 = C_1$. In case {\bf B}, removing $b'$ and $c'$ from $C'_{SW}$ and replacing them with $(k-1,w(k-1))$ and $c$ gives the desired channel of $w$. Similarly, in case {\bf A}, removing $b'$ and $c'$ from $C'_{SW}$ and replacing them with $b$ and $(k+2,w(k+2))$ gives the desired channel of $w$. 

We now take the contrapositive: if $C_1 = C$ is the southwest channel of $w$ then (again by \cite[Prop.~3.13]{cpy}) there is no channel of $w$ that contains a ball strictly southwest of some ball of $C$, and so $C'_1 = f(C)$ must have nonempty intersection with $C'_{SW}$, as claimed.
\end{proof}

Recall from Definition~\ref{defn:path} 
that a \emph{path} in a permutation $w$ is a sequence of balls of $w$ in which each ball is strictly northwest of the preceding one.

\begin{lemma}
\label{lemma:this is the southwest channel numbering}
The numbering $d'$ is the southwest channel numbering of $w'$.
\end{lemma}
\begin{proof}
Let $C$ be the southwest channel of $w$.  First, we show that for any $x'\in\B_{w'}$, there is a path in $\B_{w'}$ from $x'$ to $f(C)$ along which $d'$ decreases by $1$ at each step.  

Fix $x' \in\B_{w'}$ and let $x = f^{-1}(x')$.  By \cite[Rem.~11.7]{cpy}, there is a path $p = (x_0 = x, x_1, \dots, x_k)$ in $w$ beginning at $x$ and ending at $x_k \in C$ with $d$ decreasing by $1$ at each step.  By \cite[Rem.~11.2]{cpy}, we may choose $p$ so that it does not contain any pair of balls in the same translation class.  Then $p'=(f(x_0),\dots, f(x_k))$ is a path from $x' = f(x)$ to $f(C)$ in $w'$. 

In case {\bf D}, we have $d(x_i) = d'(f(x_i))$ and so $d'$ decreases by $1$ at each step of $p'$, as desired. 

In case {\bf AS}, if $p$ does not contain a translate of $c$ then $d'(f(x_i)) = d(x_i)$ and so $d'$ decreases by $1$ at each step of $p$, as desired.  Otherwise, we assume without loss of generality that $c$ is part of $p$.  In this case, $c$ must be the first entry of $p$: if $c = x_{i + 1}$ were northwest of $x_i$ and $d(c) + 1 = d(x_i)$, then $c'$ would be northwest of $f(x_i)$ and $d'(c') = d(c) + 1 = d(x_i)$; but this contradicts the monotonicity of $d'$ (Lemma~\ref{lemma: a certain numbering is proper}).  Finally, we consider the possibility $x' = c'$.  In this case, $b'$ is northwest of $c'$ and $d'(b') = d'(c') - 1$, and it follows from the preceding arguments in case {\bf AS} that there is a path from $b'$ to $f(C)$ with $d'$ decreasing by $1$ at each step; prepending $c'$ to this path gives the desired one.

By \cite[Rem.~11.7]{cpy}, it follows that $d'$ is the channel numbering of $w'$ with respect to $f(C)$.  Finally, since (by Lemma~\ref{lemma:two channels intersect}) $f(C)$ intersects the southwest channel of $w'$, we have by \cite[Rmk.~11.7 \& Lem.~3.15]{cpy} that $d'$ is the southwest channel numbering of $w'$, as claimed.
\end{proof}

\begin{lemma}
\label{lem:first stage}
Precisely one of the following holds:
\begin{itemize}
\item $\st(w) = \st(w')$ and $\fw{w}$ differs from $\fw{w'}$ by a Knuth move, or
\item for some $\ell\in\Z$, $\st(w)$ differs from $\st(w')$ by raising or lowering the cells in rows $\ol{\ell}$ by one row, and $\fw{w}$ differs from $\fw{w'}$ by respectively lowering the balls in rows $\ol{\ell-1}$ or raising those in rows $\ol{\ell+1}$ by one row.  
\end{itemize}
\end{lemma}

\begin{proof}
We begin with some preliminaries before moving on to case-analysis.  By Lemma~\ref{lemma:this is the southwest channel numbering}, the numbering $d'$ is the southwest channel numbering of $w'$, so the stream $\st(w')$ and permutation $\fw{w'}$ are computed in terms of this numbering.  Thus, in all cases, we consider the balls of $w$ and $w'$ to be divided into zig-zags corresponding to the numberings $d$ and $d'$, respectively. The position of every ball of $\fw{w}$ and every cell of $\st(w)$ is determined by two balls of $w$, and similarly for $w'$; we call these the \emph{parents}. Finally, since $d(b) \leqslant d(c)$ and $c$ is northeast of $b$, it follows from Proposition~\ref{prop:zig-zags}(b) that $b$ is never the northeast ball in its zig-zag.

\textbf{Case AD.}
\begin{figure}
\centering
\resizebox{.35\textwidth}{!}{\input{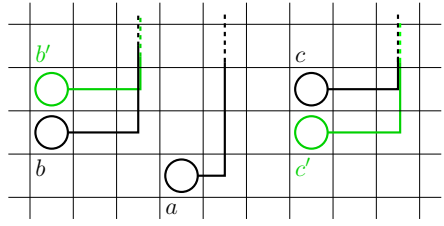}}
\caption{The case {\bf AD} of the proof of Lemma~\ref{lem:first stage} when $c$ is not the northeast ball of its zig-zag.}
\label{fig:knuth_ad_proof}
\end{figure}
Since we are in case ${\bf D}$, for every $x\in\B_w$, we have $d(x) = d'(f(x))$. Therefore, if a zig-zag $Z$ induced by $d$ does not contain a translate of $b$ or $c$, the same zig-zag (with the same inner-, outer-, and back-corner posts) is induced by $d'$.

Let $Z_b$ be the zig-zag containing $b$ induced by balls of $w$ and let $Z_{b'}$ be the zig-zag containing $b'$ induced by balls of $w'$.  As noted above, $b$ is not the northeast ball in $Z_b$.  Thus, the east parents of the back-corner posts of $Z_b$ and $Z_{b'}$ are the same, while their south parents are either the same (if $b$ is not the southwest ball of $Z_b$) or lie in the same column; thus the back-corner posts are equal.  Similarly, except the ball whose west parent is $b$, every ball of $\fw{w}$ in $Z_b$ has west parent unchanged by $f$ and north parent whose column is unchanged by $f$; thus, each of these balls is also a ball of $\fw{w'}$.  Finally, the balls of $\fw{w}, \fw{w'}$ with west parents $b, b'$ have the same north parent, so any ball with row index in $\ol{k + 1}$ in $\B_{\fw{w}}$ moves up one cell to be a ball of $\B_{\fw{w'}}$.

Now consider the zig-zags $Z_c$ and $Z_{c'}$.  Suppose first that $c$ is \emph{not} the northeast ball of its zig-zag. Then the same is true of~$c'$.  As in the previous paragraph, in this case $Z_c$ and $Z_{c'}$ have the same back-corner post, while any ball with row index in $\ol{k}$ in $\B_{\fw{w}}$ moves down one cell to be a ball of $\B_{\fw{w'}}$.  Therefore in this case $\st(w) = \st(w')$ and the permutations $\fw{w}$ and $\fw{w'}$ differ by moving the balls in rows $\ol{k+1}$ up into row $\ol{k}$ and moving the balls in rows $\ol{k}$ down into row $\ol{k + 1}$, just as $w$ and $w'$ do (see Figure \ref{fig:knuth_ad_proof}).  It remains to show that this is a Knuth move.  Let $a = (k+2, w(k+2)) \in \B_w$ be the witness for the Knuth move between $w$ and $w'$.  Aside from $b$, every ball northwest of $a$ is also northwest of $c$; thus by continuity of $d$ we have $d(b) < d(a)$ while by monotonicity we have $d(a)\leqslant d(c)$.  Using this last inequality together with Proposition~\ref{prop:zig-zags}(b), $a$ is not the northeast ball in its zig-zag.  Then it follows from Proposition~\ref{prop:zig-zags}(a) that $(\fw{w})(k) < (\fw{w})(k+2) <(\fw{w})(k+1)$. Thus the ball in row $k + 2$ of $\fw{w}$ and $\fw{w'}$ is a witness for the fact that $\fw{w}$ differs from $\fw{w'}$ by a Knuth move in position $\ol{k}$.

If instead $c$ \emph{is} the northeast ball of its zig-zag, then 
the outer-corner posts of $Z_c$ and $Z_{c'}$ are equal, so $\fw{w'}$ differs from $\fw{w}$ by shifting the ball in row $k+1$ north by one cell, while the back-corner post of $Z_{c'}$ is in row $k + 1$ instead of row $k$, so $\st(w')$ is obtained from $\st(w)$ by shifting the element in row $k$ south by one cell.

\textbf{Case BD.}  This case is similar to {\bf AD}; the difference arises when choosing a witness in the case that $c$ is not the northeast ball in its zig-zag.  Let $a = (k-1, w(k-1)))$.  In this case we have $d(b) \leqslant d(a) < d(c)$.  The ball northeast of $c$ in its zig-zag is also north of $a$, so $a$ must not be the northeast ball in its own zig-zag.  This implies $(\fw{w})(k+1) < (\fw{w})(k - 1) < (\fw{w})(k)$, and so there is a ball in row $k - 1$ of $\fw{w}$ and $\fw{w'}$ that witnesses the claimed Knuth move.

\textbf{Case AS.} This case is more intricate because $f$ affects parenthood in a nontrivial way. Let $a$ be the ball (of both $w$ and $w'$) in row $k + 2$, which witnesses the Knuth move. It was shown in the proof of Lemma~\ref{lemma: a certain numbering is proper} that in this case $d(a) = d(b)+1 = d(c)+1$. We will now use this to show that $\fw{w}$ and $\fw{w'}$ differ at most in rows $\ol{k + 1}$ and $\ol{k + 2}$. 

Zig-zags that do not contain $a, b, c$ or their translates are unaffected by the Knuth move.  Thus, we may restrict our attention to the two pairs of affected zig-zags.  The first pair consists of the zig-zags $Z_{b,c}$, containing $b$ and $c$ induced by balls of $w$, and $Z_{b'}$, containing $b'$ induced by balls of $w'$.  These two zig-zags differ by removing two consecutive balls $(k, w(k))$ and $(k + 1, w(k + 1))$ and replacing them with a single ball $(k, w(k + 1))$.  This operation results in the deletion of the outer-corner post $(k + 1, w(k))$, while the back-corner post and every other outer-corner post stays the same.

The second pair consists of the zig-zags $Z_a$, containing $a$ induced by balls of $w$, and $Z_{a, c'}$, containing $a$ and $c'$ and induced by balls of $w'$.  We have two cases to consider.  
If $a$ is \emph{not} the northeast ball in $Z_a$, suppose that the next ball northeast of $a$ is $(j, w(j))$.  Necessarily $j < k$.  Thus $c'$ is not the northeast ball in $Z_{a, c'}$, and the back-corner posts of $Z_{a, c'}$ and $Z_a$ coincide.  Therefore in this case $\st(w) = \st(w')$.  Moreover, inserting the new inner-corner post $c'$ between $a$ and $(j, w(j))$ results in the deletion of their child at position $(k + 2, w(j))$ and the creation of two new outer-corner posts at positions $(k + 2, w(k))$ and $(k + 1, w(j))$.  Combining this we the previous paragraph, we have that $\fw{w'}$ differs from $\fw{w}$ by moving the ball in row $k + 1$ south one cell into row $k + 2$ and moving the ball in row $k + 2$ north one cell into row $k + 1$.  Finally, we must show that this is a Knuth move.  Since $(j, w(j)) \in Z_a$, we have by Proposition~\ref{prop:zig-zags}(b) that $c$ is not the northeast ball in $Z_{b, c}$, and so also $b'$ is not the northeast ball in $Z_{b'}$.  Thus $\fw{w}$ and $\fw{w'}$ contain some ball $(k, y)$ in row $k$, and by Proposition~\ref{prop:zig-zags}(a) we have $w(k) < y < w(j)$.  Thus this ball is a witness to the fact that $w$ and $w'$ differ by a Knuth move in position $\ol{k+1}$.

Alternatively, suppose that $a$ \emph{is} the northeast ball in its zig-zag in $\B_w$.  By an analysis similar to the preceding one, we see that $\fw{w} \smallsetminus \fw{w'}$ consists of all translates of $(k + 1, w(k))$, while $\fw{w'} \smallsetminus \fw{w}$ consists of all translates of $(k + 2, w(k))$, and that $\st(w)$ has an element in row $k + 2$ while in $\st(w')$ this element appears one cell north, in row $k + 1$.  This is precisely the second situation in the proposition statement.
\end{proof}

With these preliminary results in hand, we move to the proof of the main result of this section.

\begin{proof}[Proof of Theorem \ref{thm:Knuth move action}]
Fix a partial permutation $w$, related to the partial permutation $w'$ by a Knuth move.  Begin applying forward moves of AMBC to both $w$ and $w'$.  By Lemma~\ref{lem:first stage}, we may have for several forward steps that the streams of the two permutations remain equal while the resulting permutations differ by a Knuth move; thus, the corresponding rows of $P(w)$, $Q(w)$ and $\rho(w)$ are respectively equal to those of $P(w')$, $Q(w')$ and $\rho(w')$.  Since $w \neq w'$ and AMBC is a bijection, we must at some point reach a forward step where we fall into the second case of Lemma~\ref{lem:first stage}; without loss of generality, we may assume that this is the first step, and that for some $\ell\in\Z$, $\st(w)$ differs from $\st(w')$ by raising the elements in rows $\ol{\ell}$ by one cell, and $\fw{w}$ differs from $\fw{w'}$ by lowering the balls in rows $\ol{\ell-1}$ by one cell (as in Example~\ref{ex:Knuth move pre-example}).

In this case, the first rows of $P(w)$ and $P(w')$ are equal, the first rows of $Q(w)$ and $Q(w')$ differ by swapping $\ol{\ell - 1}$ with $\ol{\ell}$, and $\fw{w}$ differs from $\fw{w'}$ by moving balls between rows $\ol{\ell - 1}$ and $\ol{\ell}$ (leaving the other row empty).  If $\ol{\ell} \neq \ol{1}$ then the elements of $\st(w)$ have the same block diagonals as the corresponding elements of $\st(w)$, and so in this case the first rows of $\rho(w)$ and $\rho(w')$ are equal.  If instead $\ol{\ell} = \ol{1}$ then the element of the stream that moves also decreases its block diagonal by $1$, and so the first rows of $\rho(w)$ and $\rho(w')$ differ by $1$ in this case (as in Example~\ref{ex:Knuth move example}).

It is straightforward to check from the definition of AMBC that if $u$ and $u'$ are partial permutations such that $u'$ is obtained from $u$ by shifting the the balls in rows $\ol{\ell - 1}$ down by one cell into an empty row, then $P(u) = P(u')$ and $Q(u')$ is obtained from $Q(u)$ by replacing $\ol{\ell - 1}$ with $\ol{\ell}$.  If $\ol{\ell} \neq \ol{1}$ then the corresponding streams of both permutations have the same altitudes and so $\rho(u') = \rho(u)$; if instead $\ol{\ell} = \ol{1}$ then at the forward step that places $\ol{\ell}$ into $Q(u')$ and $\ol{\ell - 1}$ into $Q(u)$, the stream coming from $u'$ has altitude one less than the stream coming from $u$.

Finally, it follows from the preceding discussion that $P(w') = P(w)$, $Q(w')$ differs from $Q(w)$ be exchanging the two entries $\ol{\ell}$ and $\ol{\ell - 1}$, and $\rho(w')$ differs from $\rho(w)$ as described in the theorem. The fact that the change from $Q(w)$ to $Q(w')$ is a Knuth move is a consequence of Proposition \ref{prop:descents}.
\end{proof}

\subsection{Covering} In this short section we prove that any Knuth move on tabloids can be realized as the change in $Q$-tabloids under a Knuth move on permutations.

\begin{defn}
For a partition $\lambda$, the \emph{Kazhdan-Lusztig dual equivalence graph} (KL DEG) $\A_\lambda$ of shape $\lambda$ is the graph whose vertices are the tabloids of shape $\lambda$ and whose edges are the Knuth moves.  Say that a permutation $w$ has \emph{shape $\lambda$} if the tabloids in its image under AMBC are of shape $\lambda$.
\end{defn}

The KL DEG $\A_{\langle 4, 1, 1 \rangle}$ is shown in Figure \ref{fig:KL deg}. 

\begin{figure}
\centering
\resizebox{.9\textwidth}{!}{\input{figures/adeg.pspdftex}}
\caption{The KL DEG $\A_{\langle 4, 1, 1 \rangle}$.  The $\tau$-invariants are shown in red.}
\label{fig:KL deg}
\end{figure}

\begin{lemma}
\label{lem:graph covering}
Consider the graph on permutations of shape $\lambda$ whose edges are Knuth moves. The map $w\mapsto Q(w)$ is a graph covering of $\A_\lambda$. 
\end{lemma}
\begin{proof}
By Theorem~\ref{thm:Knuth move action}, this map is a surjective graph morphism; we only need to check that it is an isomorphism around each vertex. Consider a permutation $w$. Suppose for some $i$, exactly one of $\ol{i}$ and $\ol{i+1}$ is in $R(w)$. By Proposition~\ref{prop:Knuth on descents}, there exists a unique Knuth move to a permutation $w'$ with exactly the other of $\ol{i}$ and $\ol{i+1}$ is in $R(w')$. By considering the various possible relative locations of $\ol{i - 1}$, $\ol{i}$, $\ol{i + 1}$, and $\ol{i + 2}$, it is easy to show that the same is true for tabloids with respect to $\tau$. Combined with the fact that $R(w) = \tau(Q(w))$ (Proposition~\ref{prop:descents}), this finishes the proof.
\end{proof}

In fact, the following stronger statement follows from the proof of Lemma~\ref{lem:graph covering}.	
\begin{prop}
If tabloids $Q$ and $Q'$ are related by a Knuth move, then for any tabloid $P$ and any weight $\rho$, the permutation $w = \Psi(P,Q,\rho)$ is related by a Knuth move to a permutation $w'$ with $Q(w') = Q'$.
\end{prop}

\section{Signs}
\label{sec:signs}

It is well known that recovering the length (i.e., number of inversions) of a finite permutation from its image under the Robinson-Schensted correspondence is hard.  However, Reifegerste showed how to recover the \emph{sign} of a permutation from its image \cite{Reifegerste}.  In this section, we prove an analogous result in the affine case.

Reifegerste's theorem statement requires two definitions.  First, given a standard Young tableau $T$ of size $n$, define the \emph{inversion number} $\inv(T)$ to be the number of pairs $(i, j)$ such that $i < j$ and $i$ appears strictly below $j$ in $T$.  Second, given a partition $\lambda$, define its \emph{inversion number} $\inv(\lambda)$ to be the sum $\lambda_2 + \lambda_4 + \ldots$ of its even-indexed parts.

\begin{thm}[Reifegerste]
\label{thm:reifegerste}
Given a finite permutation $w \in S_n$, let $(P, Q)$ be the pair of standard Young tableaux associated to $w$ by the Robinson-Schensted correspondence, both having shape $\lambda$.  The sign of $w$ is given by
\[
\sgn(w) = (-1)^{\inv(P) + \inv(Q) + \inv(\lambda)}.
\]
\end{thm}

Define the \emph{length} $\ell(w)$ of a permutation $w$ in $\widetilde{S_n}$ to be the number of \emph{inversions}, that is, translation classes of pairs $(a, b)$ of balls of $w$ with $a$ northeast of $b$.  Define the \emph{sign} $\sgn(w)$ of $w$ to be $(-1)^{\ell(w)}$. If $w$ belongs to the (non-extended) affine Weyl group $\widetilde{S}_n^0$, these definitions agree with the usual Coxeter length and sign of $w$; moreover, they extend naturally to give the sign and the length of any partial permutation.  

To state our generalization of Theorem~\ref{thm:reifegerste}, we need some additional definitions.  
\begin{definition}
\label{def:broken order}
There is a natural cyclic order on $[\ol{n}]$. In this section it is convenient to ``break'' this order into a linear order, as follows:
\[
\ol{1}<\ol{2}<\dots<\ol{n}.
\]
We refer to this total ordering as the \emph{broken order}. (See Section~\ref{sec:symmetries} for more on this symmetry-breaking.)
\end{definition}
\begin{definition}
Given a tabloid $T$, define the \emph{inversion number} $\inv(T)$ to be the number of pairs $(\ol{i}, \ol{j})$ such that $\ol{1} \leqslant \ol{i} < \ol{j} \leqslant \ol{n}$ in the broken order and $\ol{i}$ appears strictly below $\ol{j}$ in $T$. Given a shape $\lambda$ and a weight $\rho\in\Z^{\ell(\lambda)}$, define the \emph{inversion number} $\inv_\lambda(\rho)$ to be the sum of the entries of $\rho$ that correspond to rows of $\lambda$ of odd length. The inversion number $\inv(\lambda)$ of a partition continues to have the same meaning as above.
\end{definition}

\begin{thm}
\label{thm:main conjecture v2}
Given $w \in \widetilde{S_n}$, let $(P, Q, \rho)$ be the triple associated to $w$ by AMBC, with $P$ and $Q$ of shape $\lambda$. The sign of $w$ is given by
\begin{equation}
\label{sign equation}
\sgn(w) \cdot (-1)^{\sum_{i = 1}^n (w(i) - i)} = (-1)^{\inv(P) + \inv(Q) + \inv(\lambda) + \inv_\lambda(\rho)}.
\end{equation}
\end{thm}

\begin{rmk}
In the case that $w$ belongs to the (non-extended) affine Weyl group $\widetilde{S}^0_n$, we have that $\sum_{i=1}^n (w(i) - i) = 0$ and that $\sgn(w)$ agrees with the usual notion of sign in a Coxeter group.  In this case, Theorem~\ref{thm:main conjecture v2} becomes
\[
\sgn(w) = (-1)^{\inv(P) + \inv(Q) + \inv(\lambda) + \inv_\lambda(\rho)}.
\]
In the case that $w$ belongs to the finite symmetric group $S_n$, we have that $\rho$ is the all-zero vector, so $\inv_\lambda(\rho) = 0$, and we recover Theorem~\ref{thm:reifegerste}.
\end{rmk}

The proof makes use of a variation of the following result from \cite{cpy}.  Denote by $\sim$ the translation equivalence relation.
\begin{lemma}
\label{lemma:sum of diagonals}
For a cell $c = (c_1, c_2)$, let $\ol{D}(c)$ denote the diagonal of $c$, i.e., $\ol{D}(c) = c_2-c_1$. If $w\in\widetilde{S_n}$ with $\Phi(w) = (P,Q,\rho)$, $P$ and $Q$ having shape $\lambda$, then
\[
\sum_{i = 1}^n (w(i) - i) = 
\sum_{b\in\B_w/\sim} \ol{D}(b) = 
\sum_{i=1}^{\ell(\lambda)} \left(\sum_{c\in\str{\rho_i}{P_i,Q_i}/\sim} \ol{D}(c)\right) = 
n\left(\sum_{i=1}^{\ell(\lambda)}\rho_i\right).
\]
\end{lemma}
\begin{proof}
The first equality is the definition of $\ol{D}$.  The second equality is a repeated use of \cite[Lem.~10.3]{cpy}, as described in the paragraph following its proof. The third equality is \cite[Lem.~10.4]{cpy}.
\end{proof}

The same exact arguments can be applied to block diagonals (the reader may wish to recall the relevant definition from Section~\ref{sec:streams}) to get the following result.
\begin{lemma}
\label{lemma:sum of block diagonals}
Suppose $w\in\widetilde{S_n}$ and $\Phi(w) = (P,Q,\rho)$, where $P$ and $Q$ have shape $\lambda$. Then 
\[
D(w) = \sum_{i=1}^{\ell(\lambda)} D(\str{\rho_i}{P_i, Q_i}) = \sum_{i=1}^{\ell(\lambda)} \rho_i.
\]
That is, the sum of block diagonals of translate classes of balls of a permutation is equal to the sum of block diagonals of translate classes of cells of all the streams involved in the application of AMBC to the permutation.  
\end{lemma}

We now give the proof of the main theorem of this section.  The reader may wish to consult Example~\ref{ex:signs} below while reading the proof.

\begin{proof}[Proof of Theorem \ref{thm:main conjecture v2}]
In order to make an argument by induction, we carefully track how the two sides of \eqref{sign equation} change in the application of a single forward step of AMBC.

A monotone numbering of a partial permutation corresponds to a division of the balls into zig-zags.  This also induces a division of the inversions of the partial permutation into two classes: those that occur between balls in the same zig-zag and those that occur between balls in different zig-zags. We give the corresponding counts names, as follows: for the partial permutation $w$ with numbering $d: \B_w\to\Z$ and corresponding zig-zags $\{Z_i\}_{i \in \Z}$, the \emph{internal inversion number} is 
\[
\operatorname{int}(w) := \sum_{Z / \sim} \abs{\{a, b \in Z\cap \B_w : a \textrm{ is strictly northeast of } b \}},
\]
a sum over translation classes of zig-zags, and the \emph{external inversion number} is 
\[
\operatorname{ext}(w) := \sum_{(Z, Z') / \sim} \abs{\{a\in Z\cap \B_w, b \in Z'\cap \B_w : a \textrm{ is strictly northeast of } b \}},
\]
a sum over translation classes of pairs of zig-zags.

Fix a partial permutation $w$.  The balls of $w$ will be numbered by its southwest channel numbering $d$.  Let $w' = \fw{w}$ be the partial permutation that results from a single forward step of AMBC.  Let $d'$ be the numbering of balls of $w'$ induced from this step, i.e., the balls of $w'$ numbered $i$ are precisely the ones contained in $Z_i$.  (This numbering is monotone by Proposition~\ref{prop:zig-zags}(a).)  We count internal and external inversions of $w$ relative to $d$, and of $w'$ relative to $d'$.

First, we consider the internal inversions of $w$.  Choose a zig-zag $Z$ of $w$, and suppose it includes $k$ balls in $\B_w$.  Any pair of these balls forms an inversion.  Thus $Z$ contributes $\binom{k}{2}$ internal inversions to $w$.  Obviously $Z$ includes exactly $k - 1$ balls in $\B_{w'}$, and so the number of inversions of $w'$ in $Z$ is $\binom{k - 1}{2}$. The difference between these two numbers is $k - 1 = \abs{Z\cap\B_w} - 1$.  Let $\lambda$ be the shape of $w$.  The number of translation classes of balls in $w$ is $\abs{\lambda}$ and the number of translation classes of zig-zags is $\lambda_1$, so
\begin{align}
\operatorname{int}(w) & = \sum_{Z / \sim} \binom{\abs{Z\cap\B_w}}{2} \notag\\
                      & = \sum_{Z / \sim} (\abs{Z\cap\B_w} - 1) + \sum_{Z / \sim} \binom{\abs{Z \cap\B_{w'}}}{2} \notag\\
                      & = \abs{\lambda} - \lambda_1 + \operatorname{int}(w')\notag\\
                      & = \operatorname{int}(w') + \sum_{i>1}\lambda_i. \label{eq:internal}
\end{align}

Next, we consider external inversions.  For a pair of integers $i < j$, we partition the collection of inversions of $w$ between the $i$-th and $j$-th zig-zags into four parts. Recall that, by definition, to say that $(a, b)$ is an inversion is to say that $a$ is strictly northeast of $b$, and recall also the notation $a <_{SW} b$ means that $a$ is strictly southwest of $b$.  Define 
\begin{align*}
E^{i,j}_1 & = \{(a, b) \in \B_w^2 : b <_{SW} a, a \in Z_i, b\text{ is the southwest ball of } Z_j \},\\
E^{i,j}_{2} & = \{(a, b) \in \B_w^2 : b <_{SW} a, a \in Z_i, b\text{ is not the southwest ball of } Z_j \},\\
E^{i,j}_{3} & = \{(b, a) \in \B_w^2 : a <_{SW} b, a \in Z_i, b\text{ is the northeast ball of } Z_j \}, \textrm{ and}\\
E^{i,j}_{4} & = \{(b, a) \in \B_w^2 : a <_{SW} b, a \in Z_i, b\text{ is not the northeast ball of } Z_j \}.
\end{align*}
Note that this really is a partition, in that $\bigcup_{k=1}^4 E^{i,j}_k$ is the collection of all inversions between the $i$-th and $j$-th zig-zags and the $E^{i,j}_k$ pairwise do not intersect.

Suppose that $(b', a')$ is an inversion of $w'$ with $a' \in Z_i$, $b' \in Z_j$.  Let $a, b$ be the balls of $w$ in the same rows as $a', b'$, respectively, so $a$ is directly west of $a'$ and $b$ directly west of $b'$.  It follows from Proposition \ref{prop:zig-zags}(a) that $(b, a)$ is an inversion in $w$.  Moreover, $b$ is not the northeast ball of $Z_j$, so $(b,a)\in E^{i,j}_{4}$. Conversely, it follows from Proposition~\ref{prop:zig-zags} that every inversion $(b,a)\in E^{i,j}_{4}$ arises from an inversion of $w'$ in this way.

By the same argument, inversions $(a', b')$ of $w'$ with $a' \in Z_i$, $b' \in Z_j$ are in bijection with $E^{i,j}_{2}$. Thus the number of inversions of $w'$ with one ball in $Z_i$ and the other one in $Z_j$ is 
$\abs{E^{i,j}_{4}} + \abs{E^{i,j}_{2}}$.  Summing this over all translation classes of pairs $(Z_i, Z_j)$ of zig-zags with $i < j$ gives
\begin{equation}
\label{eq:external}
\operatorname{ext}(w') = \sum \left(\abs{E^{i,j}_{4}} + \abs{E^{i,j}_{2}} \right) 
   = \operatorname{ext}(w) - \sum \left(\abs{E^{i,j}_1} + \abs{E^{i,j}_{3}} \right). 
\end{equation}
We turn our attention to the summands in this last expression.

If $a, b$ are balls of $w$ and $d(a) < d(b)$ then, by monotonicity of $d$, $a$ cannot lie southeast of $b$.  Consequently $(a, b)$ is an inversion if and only if $a$ lies to the east of $b$.  When $b$ is the southwest ball of its zig-zag $Z_j$, a ball $a$ lies to the east of $b$ if and only if it lies to the east of the back-corner post $c_j$ of $Z_j$.  Thus, $\abs{E^{i,j}_1}$ is equal to the number of balls in $Z_i\cap\B_w$ lying strictly east of $c_j$.  Similarly, when $b$ is the northeast ball of $Z_j$, $(b, a)$ is an inversion if and only if $a$ lies to the south of $b$, and so $\abs{E^{i,j}_{3}}$ is equal to the number of balls of $Z_i\cap\B_w$ lying strictly south of $c_j$. Also note that $Z_i$ has no balls which share a row or a column with $c_j$ (since $i\neq j$). Combining the above observations, we have that $\abs{E^{i,j}_1} + \abs{E^{i,j}_{3}}$ is congruent modulo $2$ to the number of balls of $Z_i\cap\B_w$ that lie either strictly northeast or strictly southwest of $c_j$.

Now we sum this expression over all translation classes of pairs $(Z_i, Z_j)$ with $i < j$.  We choose as representatives the pairs $i < j$ such that $1 \leqslant j \leqslant m$, where $m$ is the number of translation classes of zig-zags (i.e., the number appearing in Proposition~\ref{prop: period of proper numbering}).  By the preceding paragraph, 
$\displaystyle\sum\left(\abs{E^{i,j}_1} + \abs{E^{i,j}_{3}}\right)$
is congruent modulo $2$ to the number of pairs $(j, b)$ such that $1 \leqslant j \leqslant m$, $b \in \B_w$, $d(b) < j$, and $b$ is either strictly northeast or strictly southwest of $c_j$.  By Proposition~\ref{prop:zig-zags}(b), if $i \geqslant j$ then $Z_i$ contains no balls either strictly northeast or strictly southwest of $c_j$.  Hence we may drop the condition $d(b) < j$ and, using the notation $\equiv$ for congruence modulo $2$, conclude that 
\[
\sum \left(\abs{E^{i,j}_1} + \abs{E^{i,j}_{3}}\right) \equiv \abs{ \{ (j, b) \in [m] \times \B_w \colon b <_{SW} c_j \textrm{ or } b >_{SW} c_j \}}.
\]

Now fix $j$ and a ball $b = (b_1, b_2)$ of $w$.  Let the coordinates of $c_j$ be given by $c_j = (c_{j, 1}, c_{j, 2})$.  We claim that the number of translates of $b$ that lie either strictly northeast or strictly southwest of $c_j$ is equal to $\abs{D(b) - D(c_j) + \fix(b, j)}$, where 
\begin{equation}
\label{eq:def fix}
\fix(b, j) := \left(\begin{cases} 1 & \text{if } \ol{b_1} < \ol{c_{j, 1}} \\ 0 & \text{if } \ol{b_1} \geqslant \ol{c_{j, 1}} \end{cases} \right) - \left(\begin{cases} 1 & \text{if } \ol{b_2} < \ol{c_{j, 2}} \\ 0 &\text{if } \ol{b_2} \geqslant \ol{c_{j, 2}} \end{cases} \right)
\end{equation}
and the inequalities between equivalence classes are in the broken order.  
(This is illustrated in Figure \ref{fig:fixing}: the coordinates of $b$ and $c_j$ are reduced modulo $n$ so that they lie in the same $n \times n$ square, and then the value of $\fix(b, j)$ is determined by the relative positions of the two.) Since each of row and column congruence class of $b$ can be either smaller, equal, or greater than the corresponding class for $c_j$, in principle there are nine cases of this claim to consider. The cases are very similar, so we only discuss two of them in detail.

\begin{figure}
\centering
\begin{minipage}[t]{0.4\textwidth}
\centering
\resizebox{!}{.19\textheight}{\input{figures/fixing.pspdftex}}
\caption{The dependence of $\fix(b, j)$ on the congruence classes of rows and columns of $b$ relative to $c_j$.}
\label{fig:fixing}
\end{minipage}\hfill
\begin{minipage}[t]{0.6\textwidth}
\centering
\resizebox{!}{.19\textheight}{\input{figures/how_many_translates.pspdftex}}
\caption{Counting translates of $b$ lying northeast of $c_j$ when $D(b)-D(c_j) = 1$ (solid), $D(b)-D(c_j) = 2$ (dashed), and $D(b)-D(c_j) = 3$ (dotted).}
\label{fig:how many translates does it take}
\end{minipage}
\end{figure}

First, consider the case when $\ol{b_1} > \ol{c_{j, 1}}$ and $\ol{b_2} > \ol{c_{j, 2}}$ in broken order.  Without loss of generality, assume that $D(b) \geqslant D(c_j)$, so that translates of $b$ can be northeast but not southwest of $c_j$ (see Figure \ref{fig:how many translates does it take} for examples with $D(b) - D(c_j)$ being $1,2$, and $3$).  The translate $b + k(n, n)$ is northeast of $c_j$ if and only if $c_{j, 1} > b_1 + kn$ and $c_{j, 2} < b_2 + kn$ (where now we are comparing integers, not residue classes).  Thus, the number of translates of $b$ northeast of $c_j$ is equal to the number of integer solutions $k$ to the inequalities
\[
\frac{c_{j, 1} - b_1}{n} > k > \frac{c_{j, 2} - b_2}{n}.
\]
Since $\ol{b_1} > \ol{c_{j, 1}}$, the largest integer smaller than $(c_{j, 1} - b_1)/n$ is $\ceil{c_{j, 1}/n} - \ceil{b_1/n} - 1$, while since $\ol{b_2} > \ol{c_{j, 2}}$ the smallest integer larger than $(c_{j, 2} - b_2)/n$ is $\ceil{c_{j, 2}/n} - \ceil{b_2/n}$; thus, the displayed inequalities are equivalent to
\[
\ceil{\frac{c_{j, 1}}{n}} - \ceil{\frac{b_1}{n}} - 1 \geqslant k \geqslant \ceil{\frac{c_{j, 2}}{n}} - \ceil{\frac{b_2}{n}}.
\]
The number of solutions $k$ is
\[
\left(\ceil{\frac{c_{j, 1}}{n}} - \ceil{\frac{b_1}{n}} - 1\right) - \left(\ceil{\frac{c_{j, 2}}{n}} - \ceil{\frac{b_2}{n}}\right) + 1 = D(b) - D(c_j) = \abs{D(b) - D(c_j) + 0}.
\]
Finally, in this case $\fix(b, c_j) = 0$, as needed.

Second, consider the case that $\ol{b_1} < \ol{c_{j, 1}}$ in broken order and $\ol{b_2} = \ol{c_{j, 2}}$.  Since $\ol{b_2} = \ol{c_{j, 2}}$, the ball $b$ is a translate of the ball in the same column as $c_j$.  By the definition of AMBC, this ball lies directly south of $c_j$, so necessarily $D(b) < D(c_j)$ and we seek translates of $b$ that are strictly southwest of $c_j$.  By the same analysis as in the previous case, these translates correspond to integer solutions of   
\[
\frac{c_{j, 1} - b_1}{n} < k < \frac{c_{j, 2} - b_2}{n}.
\]
By the hypotheses on the relative values of the $b_i$ and $c_{j, i}$, this is equivalent to
\[
\ceil{\frac{c_{j, 1}}{n}} - \ceil{\frac{b_1}{n} } + 1 \leqslant k \leqslant \ceil{\frac{c_{j, 2}}{n}} - \ceil{\frac{b_2}{n}} - 1.
\]
Thus in this case the number of solutions is $D(c_j) - D(b) - 1 = \abs{D(b) - D(c_j) + 1}$.  Since $\fix(b, c_j) = 1$, this gives the desired result.

The other seven cases are extremely similar to these two.

We now plug the claim in to \eqref{eq:external}.  Since $\abs{x} \equiv x\equiv -x$, we can drop absolute values and extra negative signs to get
\begin{align*}
\operatorname{ext}(w) - \operatorname{ext}(\fw{w})
 & \equiv  \sum_{b \in (\B_w/\sim)}  \sum_{j=1}^m \abs{D(b) - D(c_j) + \fix(b, j)}\\ 
 & =       \lambda_1 \cdot \sum_{b \in (\B_w/\sim)}  D(b) + \abs{\lambda} \cdot \sum_{j=1}^m D(c_j)  + \sum_{b \in (\B_w/\sim)}  \left(\sum_{j=1}^m \fix(b, j)\right).
\end{align*}

By Lemma~\ref{lemma:sum of block diagonals}, the first sum is equal to $\sum_i \rho_i$.  The second sum is the definition of the altitude of $\st(w)$, and so is equal to $\rho_1$.  For the third sum, we break it into two pieces as in~\eqref{eq:def fix}, one involving the row indices and one involving the column indices. The sum involving the row indices is equal to the number of pairs of a ball $b$ and a stream cell $c_j$ such that $\ol{b_1} < \ol{c_{j, 1}}$. Note that the numbers $\ol{c_{j, 1}}$ constitute the first row of $Q$, while the other rows of balls will necessarily appear in the lower rows of $Q$. Thus the desired number of pairs is precisely equal to $\inv_1(Q)$, where $\inv_i(T)$ is the number of inversions of a tabloid $T$ such that the top row involved is row $i$. The same analysis applies to the other piece, and so the last sum in the previous displayed equation simplifies to $\inv_1(Q) - \inv_1(P)$. Therefore
\[
\operatorname{ext}(w) - \operatorname{ext}(\fw{w}) \equiv \abs{\lambda}\rho_1  + \lambda_1\sum_{i\geqslant 1}\rho_i + \inv_1(Q) + \inv_1(P).
\]
Combining this with \eqref{eq:internal} yields
\begin{equation}
\label{diffcong}
\ell(w) - \ell(\fw{w})\equiv \sum_{i>1}\lambda_i + \lambda_1\sum_{i\geqslant 1}\rho_i 
+ \abs{\lambda}\rho_1  + \inv_1(Q) + \inv_1(P).
\end{equation}

Now suppose $w\in\widetilde{S_n}$. The congruence \eqref{diffcong} holds at each step of AMBC, with appropriate adjustments to the indexing (i.e., the first row of $P(\fw{w})$ should be numbered $2$, not $1$); after the last step of AMBC, we are left with the empty permutation, which has $0$ inversions. Thus the sum of the left side of \eqref{diffcong} over all the steps gives $\ell(w)$. By Lemma \ref{lemma:sum of diagonals}, we have $n\sum_{i\geqslant 1}\rho_i = \sum_{i=1}^n (w(i) - i)$, and so
\begin{align*}
 \ell(w) & \equiv 
    \sum_{j=1}^{\ell(\lambda)}\left(\sum_{i>j}\lambda_i + \lambda_j\sum_{i\geqslant j}\rho_i + \rho_j\sum_{i\geqslant j}\lambda_i  + \inv_j(Q) + \inv_j(P)\right)\\
	& =  \sum_{i=1}^{\ell(\lambda)} (i-1)\lambda_i + \sum_{j=1}^{\ell(\lambda)}\left(\lambda_j\sum_{i\geqslant j}\rho_i + \rho_j\sum_{i\geqslant j}\lambda_i\right) + \inv(Q) + \inv(P)\\
	& \equiv  
             \inv(\lambda) + \sum_{j=1}^{\ell(\lambda)}\left(\lambda_j\rho_j+\lambda_j\sum_{i\geqslant 1}\rho_i\right) + \inv(Q) + \inv(P)\\
	& = \inv(\lambda) + \sum_{j=1}^{\ell(\lambda)}\lambda_j\rho_j + n\sum_{i\geqslant 1}\rho_i + \inv(Q) + \inv(P)\\
	& \equiv \inv(\lambda) + \inv_\lambda(\rho) + \sum_{i=1}^n (w(i) - i) + \inv(Q) + \inv(P).
\end{align*}
Raising $-1$ to the power of each side yields the desired result. 
\end{proof}

\begin{ex}
\label{ex:signs}
The permutation $w = [7, 2, 4, 1]$ in $\widetilde{S_4}$ is shown in Figure~\ref{fig:inversions}.  For readability, we label inversions using only the column index, so that the seven inversions of $w$ are $(7, 2)$, $(7, 4)$, $(7, 1)$, $(7, 6)$, $(7, 5)$, $(2, 1)$, and $(4, 1)$.
Under AMBC, $w$ corresponds to the triple 
\[
P = \tableau[sY]{\ol{1}, \ol{4} \\ \ol{2} \\ \ol{3}} , \qquad 
Q = \tableau[sY]{\ol{1}, \ol{3} \\ \ol{2} \\ \ol{4}},  \qquad 
\rho = \begin{pmatrix} 0 \\ 0 \\ 1 \end{pmatrix}
,
\]
having $\inv(P) = 2$, $\inv(Q) = 1$, $\inv(\lambda) = 1$, and $\inv_\lambda(\rho) = 1$.  Thus in this case, Theorem~\ref{thm:main conjecture v2} asserts that $(-1)^7 \cdot (-1)^4 = (-1)^5$.
\begin{figure}
\centering
\resizebox{.45\textwidth}{!}{\input{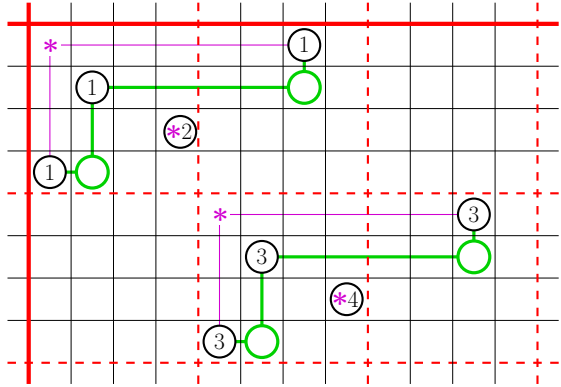}}
\caption{One forward step of AMBC, starting with the permutation $[7, 2, 4, 1]$.}
\label{fig:inversions}
\end{figure}
Of the seven inversions, $(7, 2)$, $(7, 1)$, and $(2, 1)$ are internal, while $w' = \fw{w}$ has a single internal inversion $(7, 2)$; the difference $3 - 1$ is equal to $|\lambda| - \lambda_1 = 2$.  The four external inversions of $w$ are partitioned as follows: $E_{1}^{1,2} = \{(7, 4)\}$, $E_2^{1, 3} = \{(7, 6)\}$, $E_1^{1, 3} = \{(7, 5)\}$, and $E_3^{1, 2} = \{(4, 1)\}$.  The inversion $(7, 6)$ corresponds to the unique external inversion $(7, 6)$ in $w'$.  Finally, we note that $\fix(b, j) = 0$ except for the case $b = (4, 1)$ and $j = 2$ (with $c_2 = (3, 4)$).
\end{ex}

\begin{rmk}
Reifegerste proves Theorem~\ref{thm:reifegerste} by induction, using a result of Beissinger \cite{Beissinger} to show that it holds for an element in each Knuth equivalence class and then proving that its validity is preserved under Knuth moves.  The second half of this argument (the inductive step) is straightforward using Theorem~\ref{thm:Knuth move action}; however, it is not clear what Knuth class representatives could play the role of Beissinger and Reifegerste's involutions for the base case.
\end{rmk}

\section{Charge}
\label{sec:charge intro}

The purpose of this section is to introduce a statistic on tabloids we call \emph{local charge}. Charge (originally defined in \cite{ls}; see also \cite{manivel} for an exposition) is a classical statistic for tableaux, which generalizes naturally to tabloids. What is most useful to us is a ``local'' version that depends on a pair of adjacent rows, rather than the entire tabloid, and in this section we only define it in this context. This statistic arises in two ways in our theory. First, the instance we deal with in this section, is that it plays a crucial role in understanding the offset constants the the definition of dominance (see Proposition~\ref{prop:inverses} and Remark~\ref{rmk:not obvious}).  This is captured in Theorem~\ref{thm:dominance constants}, the main result of this section.  Second, as will be described in Section \ref{sec:charge}, it appears in the description of the connected components of the KL DEG $\A_\lambda$.

\subsection{Definitions and basic properties}
\begin{defn}
Suppose $T$ is a tabloid of shape $\langle m, m \rangle$. An \emph{activation ordering} for $T$ is a bijection $o : [m] \to T_1 \subsetneq [\ol{n}]$. We think of activation orderings as linear orderings of the entries of the top row of $T$. The \emph{standard activation ordering} is the activation ordering in which the entries of the top row are arranged to increase in the broken order (see Definition~\ref{def:broken order}). 
\end{defn}
Whenever an activation ordering is needed but not specified, we assume that the standard activation ordering is used.

\begin{defn}
Suppose $T$ is a tabloid whose $k$-th and $(k+1)$-st rows have the same size $m$, and $o$ is an activation ordering for $T_{k,k+1}$ (the tabloid consisting of the $k$-th and $(k+1)$-st rows of $T$). Then the \emph{charge matching} between rows $k$ and $k+1$ of $T$ with ordering $o$ is the matching of entries in row $k$ with entries in row $k + 1$ defined as follows. Suppose
\begin{equation}
\label{eq:twoRowTabloid}
T_{k,k+1} = \tableau[sY]{\ol{a_1}, \ol{a_2}, \ldots, \ol{a_m} \\ \ol{b_1}, \ol{b_2}, \ldots, \ol{b_m}} \, ,
\end{equation}
where $\ol{a_i} = o(i)$ for all $i$. For $i = 1, 2, \ldots, m$, match $\ol{a_i}$ with the smallest (in broken order) unmatched $\ol{b_j}$ such that $\ol{b_j} > \ol{a_i}$, if such a $\ol{b_j}$ exists; otherwise, match $\ol{a_i}$ to the smallest unmatched $\ol{b_j}$.
\end{defn}

\begin{defn}
\label{def:local charge}
Suppose $T$ is a tabloid whose $k$-th and $(k+1)$-st rows have the same size $m$, and $o$ is an activation ordering of $T_{k,k+1}$. The \emph{local charge in row $k$} of $T$, denoted $\lch^{o}_k(T)$, is the number of elements $\ol{a}$ of the $k$-th row of $T$ such that $\ol{a}$ is matched to $\ol{b}$ and $\ol{a} > \ol{b}$.  We may say that such an entry $\ol{a}$ (or the corresponding $\ol{b}$, or the pair $(\ol{a}, \ol{b})$) \emph{contributes} to the charge.  If $o$ is the standard ordering, we omit $o$ from the notation and write $\lch_k(T)$.
\end{defn}

\begin{ex}
Consider the tabloid
\[
T = \tableau[sY]{\ol{2}, \ol{4}, \ol{6}, \ol{10} \\ \ol{3}, \ol{7}, \ol{8} \\ \ol{1}, \ol{5}, \ol{9}}.
\]
The local charge $\lch_1(T)$ is not defined since the first and second rows of $T$ have different lengths. With the standard activation ordering, the charge matching between the second and third rows is $\ol{3}\leftrightarrow \ol{5}$, $\ol{7}\leftrightarrow \ol{9}$, $\ol{8}\leftrightarrow \ol{1}$. The only pair that contributes to charge is $(\ol{8}, \ol{1})$, so $\lch_2(T) = 1$.

Suppose now we use non-standard activation ordering between the second and third rows, say, $o(1) = \ol{8}$, $o(2) = \ol{7}$, $o(3) = \ol{3}$. Then the charge matching is $\ol{8}\leftrightarrow \ol{9}$, $\ol{7}\leftrightarrow \ol{1}$, $\ol{3}\leftrightarrow \ol{5}$. For this $o$, the only pair that contributes to the charge is $(\ol{7}, \ol{1})$, and so $\lch^o_2(T) = 1$.
\end{ex}

We now show that the coincidence in the example above is not an accident: local charge does not depend on the choice of activation ordering.

\begin{lemma}
\label{lem:local independent of activation}
Suppose $T$ is a tabloid whose $k$-th and $(k+1)$-st rows have the same size $m$.  Then $\lch^o_k(T) = \lch_k(T)$ for every activation ordering $o$.
\end{lemma}
\begin{proof}
Suppose $T_{k,k+1}$ is as in \eqref{eq:twoRowTabloid} with $o(i) = \ol{a_i}$ for $1\leqslant i\leqslant m$.  Since any permutation is a product of adjacent transpositions, it is sufficient to prove that $\lch^o_k(T) = \lch^{o'}_k(T)$, where $o'$ is the ordering defined by
\begin{align*}
o'(j) & = \ol{a_{j+1}}, \\
o'(j + 1) & = \ol{a_j}, &&\textrm{and} \\
o'(i) & = \ol{a_i} &&\textrm{if } i\neq j \textrm{ or } j+1.
\end{align*}
Furthermore, with this choice, the entries $\ol{a_1}, \ldots, \ol{a_{j - 1}}$ match to the same elements of the second row in the charge matchings with orderings $o, o'$.  Thus, we may as well remove these $j - 1$ entries from both the first and second rows, and consider the case $j = 1$.  As a final simplification, we may assume that $\ol{a_1} < \ol{a_{2}}$: if not, switch the roles of $o$ and $o'$. 

Let $\ol{b_{i_1}}$ and $\ol{b_{i_{2}}}$ be the elements of the second row with which $\ol{a_1}$ and $\ol{a_{2}}$ match in the charge matching with ordering $o$.  We seek to show that in the charge matching with ordering $o'$, $\ol{a_1}$ and $\ol{a_{2}}$ still match (in some order) to  $\ol{b_{i_1}}$ and  $\ol{b_{i_{2}}}$ in a way that preserves the number of contributing pairs.

\emph{A priori}, there are $12$ possible orders for the four values $\ol{a_1} , \ol{a_{2}} , \ol{b_{i_1}} , \ol{b_{i_{2}}}$ in broken order;
however, six of these
are incompatible with the fact that $\ol{a_1}$ matches to $\ol{b_{i_1}}$ in the matching with ordering $o$.  This leaves six cases to check; we present the two most complicated here, and leave the other four to the reader.  


Suppose $\ol{a_1} < \ol{b_{i_1}} < \ol{b_{i_{2}}} < \ol{a_{2}}$.  Since $\ol{a_{2}}$ matches to $\ol{b_{i_{2}}}$ in the matching with ordering $o$, there are no values in $T_{k + 1}$ larger than $\ol{a_{2}}$, and the only value in $T_{k + 1}$ smaller than $\ol{b_{i_{2}}}$ is $\ol{b_{i_1}}$.  It follows that in the matching with ordering $o'$, $\ol{a_{2}}$ matches to $\ol{b_{i_1}}$ and $\ol{a_1}$ matches to $\ol{b_{i_{2}}}$.  Then $\ol{a_1}$ does not contribute to charge in either matching and $\ol{a_{2}}$ contributes to the charge in both matchings.



Suppose $\ol{b_{i_{2}}} < \ol{a_1} < \ol{a_{2}} < \ol{b_{i_1}}$.  Since $\ol{a_{1}}$ matches to $\ol{b_{i_{1}}}$ in the matching with ordering $o$, $\ol{b_{i_{1}}}$ is the smallest value in $T_{k + 1}$ larger than $\ol{a_1}$, and hence is the smallest value in $T_{k + 1}$ larger than $\ol{a_{2}}$.  Moreover, since $\ol{a_{2}}$ matches to $\ol{b_{i_{2}}}$, there can be no other values in $T_{k + 1}$ larger than $\ol{a_{2}}$, and therefore none larger than $\ol{a_1}$; and $\ol{b_{i_{2}}}$ is the smallest value in $T_{k + 1}$.  Thus in the charge matching with ordering $o'$, $\ol{a_{2}}$ matches to $\ol{b_{i_1}}$ and $\ol{a_1}$ matches to $\ol{b_{i_{2}}}$.  Thus $\ol{b_{i_{2}}}$ contributes to the charge in both matchings, while $\ol{b_{i_1}}$ does not contribute to the charge in either.


Since the same two elements in the second row are paired to $\{a_1, a_2\}$ in the matchings corresponding to $o$ and $o'$, it follows that all pairs involving $\ol{a_i}$ for $i > 2$ are also the same in both matchings.  Thus the two local charges are equal, as claimed.
\end{proof}

\subsection{A formula for the offset constants}

Recall the discussion of the offset constants in Section \ref{sec:intro ars}. The goal of this section is to give a formula for these constants in terms of local charge. Before we can do that, we need to precisely define the constants. We use the following convention: whenever we have a stream $S$ and a proper numbering of its cells, we let $S^{(i)}$ denote the cell numbered $i$.  

\begin{propdef}
Suppose $S$ and $T$ are streams of the same density $m$ such that no cell of $S$ shares a row or a column with a cell of $T$, and $S$ is properly numbered.  Then there is a unique proper numbering of $T$ with the following properties:
\begin{enumerate}
\item for all $i$, $S^{(i)}$ is northwest of $T^{(i)}$, and
\item for some $j$, $S^{({j+1})}$ is not northwest of $T^{(j)}$.
\end{enumerate}
This numbering is called the \emph{backward numbering} of $T$ with respect to $S$.
\end{propdef}
\begin{proof}
A different definition of the backward numbering is given in \cite{cpy}; the proof proceeds by showing they are equivalent.  The proof may be skipped without harming understanding of the rest of this section.

Let $d$ be the backward numbering of $T$ constructed in \cite[\S4.1]{cpy}.  By construction, $d$ is monotone and it satisfies point (1). As described in \cite[Rem.~13.3]{cpy}, it also satisfies point (2).  Thus such numberings exist.  Let $d'$ be any numbering satisfying (1) and (2).  By \cite[Rem.~13.1]{cpy}, $d'(x) \leqslant d(x)$ for every $x \in T$.  Since $d'$ satisfies (2), it must coincide with $d$ on at least one translation class of cells. However, \cite[Prop.~13.2]{cpy} states that $d$ is, in fact, proper (i.e., $T$ is numbered by consecutive integers). Thus any monotone numbering that coincides with it on a translate class must coincide with it everywhere, so $d' = d$.
\end{proof}

Condition (2) in the definition of the backward numbering may be met in two different ways: $T^{(j)}$ might be either north or west of $S^{(j + 1)}$.  The situation when both occur simultaneously is special.

\begin{defn}[{\cite[Def.~5.4]{cpy}}]
Suppose $S$ and $T$ are streams of the same density $m$, and no cell of $S$ shares a row or a column with a cell of $T$.  Number $S$ with a proper numbering and number $T$ with the backward numbering with respect to $S$. Then $T$ is said to be \emph{concurrent to $S$} if there exist $i$ and $j$ (possibly equal) such that $T^{(i)}$ is north of $S^{({i+1})}$ and $T^{(j)}$ is west of $S^{({j+1})}$.
\end{defn}

\begin{prop}[{\cite[Prop.~5.6]{cpy}}] 
Suppose $A,B,A',B'$ are equinumerous subsets of $[\ol{n}]$ with $A\cap A' = B\cap B' =\varnothing$. Then there exists a unique integer $r$ such that $\st_{r}(A',B')$ is concurrent to $\str{0}{A,B}$. 
\end{prop}

\begin{defn}[{\cite[Def.~5.8]{cpy}}] 
\label{def:dominance}
Suppose $P$ and $Q$ are tabloids of the same shape $\lambda$, and that $\lambda_i = \lambda_{i+1}$. The \emph{dominance constant}
$r_{i+1}(P, Q)$ is the unique integer $r$ such that $\str{r}{P_{i+1},Q_{i+1}}$ is concurrent to $\str{0}{P_{i},Q_{i}}$. The weight $\rho$ is \emph{dominant} for $P, Q$ if and only if 
\[
\rho_{i + 1} \geqslant \rho_{i} + r_{i + 1}(P, Q)
\]
for every $i$ such that $\lambda_{i} = \lambda_{i + 1}$.
\end{defn}

\begin{thm}
\label{thm:dominance constants}
Suppose that $P$ and $Q$ are two tabloids of shape $\lambda$ and $\lambda_i = \lambda_{i+1}$.  Then the dominance constants are given by 
\[
r_{i+1}(P,Q) = \lch_i(P) - \lch_i(Q).
\]
\end{thm}

Before proceeding with the proof we need a technical lemma.

\begin{lemma}
\label{lemma:charge_decrement}
Fix $m\in\Z^{>0}$. Suppose we have two sets of integers $A = \{a_1, a_2, \ldots, a_m\}$ and $B = \{b_1, b_2, \ldots, b_m\}$ with $1 \leqslant a_1 < \dots < a_m$ and $1\leqslant b_1 < \dots < b_m$. Define $\ell(A,B)$ to be the maximal integer such that there exist indices $1\leqslant i_1 < i_2 < \dots < i_{\ell(A,B)}\leqslant m$ satisfying 
\[
a_1 < b_{i_1}, \quad a_2 < b_{i_2}, \quad \ldots, \quad a_{\ell(A,B)} < b_{i_{\ell(A,B)}}.
\] 
Choose $b' > a_m, b_m$, and let $B' =  \{b_2, \ldots, b_m, b'\}$. If $\ell(A,B) < m$ then $\ell(A, B') = \ell(A, B) + 1$.
\end{lemma}
\begin{proof}
The proof is illustrated in Figure \ref{fig:charge_decrement}.
\begin{figure}
\resizebox{.8\textwidth}{!}{\input{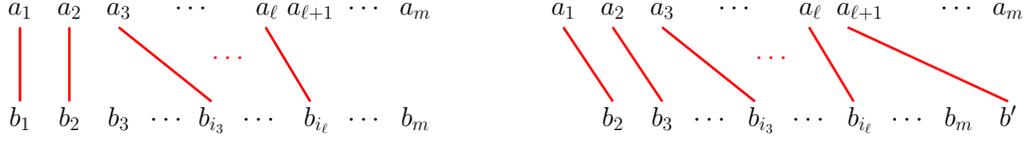}}
\caption{A figure illustrating the proof of Lemma~\ref{lemma:charge_decrement}. The red lines represent inequalities between an element $a_j$ of $A$ and a larger element $B$ (left) or $B'$ (right).  In this example, $i_1 =1$ and $i_2 = 2$, while $i_3>3$.}
\label{fig:charge_decrement}
\end{figure}
Let $k$ be the largest non-negative integer such that $i_k = k$ (and if $i_1 > 1$ then take $k = 0$). For $1\leqslant j\leqslant k$, define $b_{i_j}' = b_{i_j+1} ( = b_{j+1})$. For $k+1\leqslant j\leqslant \ell(A,B)$, define $b_{i_j}' = b_{i_j}$. Define $b_{i_{\ell(A,B)+1}}' = b'$. Since $i_{k + 1} > k + 1 = i_k + 1$, the $b'_{i_j}$ are distinct.  
For $1 \leqslant j \leqslant \ell(A,B)$, we have $a_j < b_{i_j} \leqslant b'_{i_j}$, and also $a_{\ell(A, B) + 1} \leqslant a_m < b_{i_{\ell(A, B)} + 1}'$.  
Thus $\ell(A, B')\geqslant \ell(A,B) + 1$.  

Conversely, if for some integer $k$ there are inequalities 
\[
a_1 < b'_{i_1}, \quad a_2 < b'_{i_2}, \quad \ldots, \quad a_{k} < b'_{i_{k}}
\]
such that the right sides are distinct elements of $B'$, then removing the inequality involving $b'_m = b'$, if it exists, leaves $k - 1$ inequalities between elements of $A$ and elements of $B$, so $\ell(A, B') \leqslant \ell(A, B) + 1$.  This completes the proof.
\end{proof}

We are now prepared for the proof of the main result of this section.

\begin{proof}[Proof of Theorem~\ref{thm:dominance constants}]
Both sides of the desired equality depend only on a pair of adjacent rows of the same size. Hence, in the remainder of the proof we assume that $P$ and $Q$ are tabloids of shape $\langle m, m\rangle$ with entries in $[\ol{n}]$ for some $m \leqslant n/2$.  

Consider $S := \str{0}{P_1,Q_1}$.  Number $S$ properly so that the north-most ball southeast of the origin is numbered $1$, and therefore $S^{(1)},\dots, S^{(m)}$ lie in $[n]\times [n]$. Let $r = \lch_1{P} - \lch_1{Q}$ and $T = \str{r}{P_2,Q_2}$. We must show that $T$ is concurrent to $S$.  First, we give a description of the backward numbering of $T$.

Denote $P_1 = \{\ol{a_1}, \ldots, \ol{a_m}\}$, $P_2 = \{\ol{b_1}, \ldots, \ol{b_m}\}$, $Q_1 = \{\ol{a'_1}, \ldots, \ol{a'_m}\}$, and $Q_2 = \{\ol{b'_1}, \ldots, \ol{b'_m}\}$, with representatives chosen so that $1\leqslant a_1 < a_2 < \dots < a_m \leqslant n$, and similarly for the other three rows.  Extend the indices to all of $\Z$ via
$
a_{i + m} = a_i + n,
$
so that $\bigcup P_1 = \{\dots < a_{-1} < a_0 < a_1 < \cdots \}$ and $a_0 \leqslant 0 < a_1$, and similarly for the other rows.  With this choice, we have $S^{(i)} = (a'_i, a_i)$ for all $i$.  Define a numbering of $T$ as follows: the ball numbered $i$ is 
\[
\left(b'_{i + \lch_1(Q)}, b_{i + \lch_1(P)}\right)
=: T^{(i)}. 
\]
This is clearly a proper numbering of $T$; we claim that this is the backward numbering of $T$ with respect to $S$.  To prove this, we must show that $S^{(i)}$ is northwest of $T^{(i)}$ for every $i$, and that $S^{(j + 1)}$ is not northwest of $T^{(j)}$ for some $j$.

Define the local charge on a pair of equinumerous sets of integers (rather than equivalence classes modulo $n$) in the obvious way, so that 
$\lch(\{a_1,a_2,\ldots, a_m\}, \{b_1, b_2,\ldots, b_m\}) = \lch_1(P)$. We apply Lemma~\ref{lemma:charge_decrement} repeatedly, each time replacing the smallest entry $b_i$ in the second row with $b_{i + m} = b_i + n$; this yields
\[
\lch(\{a_1,a_2,\ldots, a_m\}, \{b_{1 + \lch_1(P)}, \ldots, b_{m + \lch_1(P)}\}) = 0.
\]
Therefore $a_i < b_{i + \lch_1(P)}$ for all $i$, and so $S^{(i)}$ is west of $T^{(i)}$ for all $i$.  Applying the same argument to $Q$ (\emph{mutatis mutandis}) yields that $S^{(i)}$ is north of $T^{(i)}$ for all $i$.

For the second half of the claim, we apply Lemma~\ref{lemma:charge_decrement} one fewer time\footnote{There is a subtlety here in the case that the original local charge is $0$.  In this case, we need a variant of Lemma~\ref{lemma:charge_decrement} in which the \emph{largest} element in the bottom row is replaced by a new \emph{minimal} element; the proof is no harder.} to conclude that
\[
\lch(\{a_1, \ldots, a_m\}, \{b_{\lch_1(P)}, \ldots, b_{m - 1 + \lch_1(P)}\}) = 1.
\]
Therefore there is some $j$ such that $b_{j - 1 + \lch_1(P)} < a_j$, and so $S^{({j})}$ is not west (and \emph{a fortiori} not northwest) of $T^{({j - 1})}$.  This completes the proof that the numbering described is the backward numbering of $T$ with respect to $S$.

Moreover, applying the same argument as in the previous paragraph to $Q$ instead of $P$ yields that there is some $j$ such that $S^{(j)}$ is not north of $T^{({j-1})}$. Together, these two paragraphs establish that $T$ is concurrent to $S$.  
\end{proof}

\subsection{Resolving a conundrum}
As described in Section~\ref{sec:intro ars}, the composite map $\Phi \circ \Psi$ is a surjection from $\Omega$ to $\dom$, and the image of $(P, Q, \rho)$ is a triple $(P, Q, \rho')$ where $\rho'$ is the \emph{dominant representative} of $\rho$.  In this section, we use Theorem~\ref{thm:dominance constants} to resolve a mystery mentioned in Remark~\ref{rmk:not obvious}: why does applying the operation $(P, Q, \rho) \mapsto (Q, P, (-\rho)')$ twice return the original triple?

In this context, it is convenient to define the \emph{symmetrized offset constants}
\[
s_i(P, Q) := \sum_{j=i' + 1}^i r_j(P, Q),
\]
where $i'$ is the smallest integer such that $\lambda_{i'} = \lambda_{i}$.  We may reformulate Definition~\ref{def:dominance} in terms of these constants: $\rho$ is dominant if for each $1\leqslant i < \ell(\lambda)$ one of the following holds:
\begin{itemize}
\item $\lambda_i > \lambda_{i+1}$, or
\item $\lambda_i = \lambda_{i+1}$ and $(\rho - \mathbf{s})_{i+1}\geqslant (\rho - \mathbf{s})_i$,
\end{itemize}
where $\mathbf{s} = (s_i(P,Q))_{1\leqslant i\leqslant \ell(\lambda)}$.  Suppose that $\lambda$ has $k$ distinct part sizes, occurring with multiplicities $m_1, \ldots, m_k$; let $m_i' := m_1 + \ldots + m_i$.  According to \cite[\S6]{cpy}, the dominant representative of $\rho$ can be computed as follows:
\begin{itemize}
\item Let $\mathbf{x} = \rho - \mathbf{s}$.
\item For each $1\leqslant j \leqslant k$, let $(y_{m'_{j - 1} + 1}, y_{m'_{j - 1} + 2}, \dots, y_{m'_j})$ be the increasing rearrangement of $(x_{m'_{j - 1} + 1}, x_{m'_{j - 1} + 2}, \dots, x_{m'_j})$.
\item Let $\rho' = \mathbf{y} + \mathbf{s}$.
\end{itemize}

We now resolve the mystery: by Theorem~\ref{thm:dominance constants}, $r_i(Q, P) = - r_i(P, Q)$ for all $i$, and so also $s_k(Q, P) = -s_k(P, Q)$. If $\lambda_i = \lambda_{i+1} = \ldots = \lambda_j$ is a maximal string of equal part sizes then for any $i\leqslant k\leqslant j$, there exists $i\leqslant \ell\leqslant j$ such that $(-\rho)'_k = -\rho_\ell - s_\ell(P, Q) + s_k(P, Q)$. Therefore $(-(-\rho)')_k = \rho_\ell + s_\ell(P, Q) - s_k(P, Q)$.  To take the dominant representative of this vector, we subtract $s_k(Q, P) = -s_k(P, Q)$, giving a vector whose $k$-th entry is $\rho_\ell +s_\ell(P, Q)$; since $\rho$ is dominant, this is the $\ell$-th largest of these numbers, so after rearranging we add $s_\ell(Q, P) = -s_\ell(P, Q)$, returning the value $\rho_\ell$, as desired. 

\section{Symmetries}
\label{sec:symmetries}
The definition of the extended affine symmetric group (as in Section~\ref{sec:notational preliminaries}) gives no special standing to the interval $[n]$ over any other interval of $n$ consecutive integers.  This indifference can be seen in the expression $\widetilde{S_n} = \widetilde{S}_n^0 \rtimes \Z$ for $\widetilde{S_n}$ as a semidirect product of the non-extended affine Weyl group $\widetilde{S}_n^0$ and the infinite cyclic subgroup generated by the ``shift'' permutation $s=[2,3,\ldots,n+1]$. It may also be seen in the Dynkin diagram of the group $\widetilde{S}_n^0$, which is a cycle on $n$ vertices, each vertex representing a simple transposition: conjugation of $\widetilde{S}_n^0$ by $s$ is equivalent to rotating the Dynkin diagram by one position.

The broken order of Sections~\ref{sec:signs} and~\ref{sec:charge intro} breaks this symmetry by giving special standing to the link from $\ol{n}$ to $\ol{1}$.  This corresponds to other broken symmetries in our setting: the definition of altitude of a stream (and so everything that follows from it) gives special standing to the square $[n] \times [n]$.

The shift permutation $s$ has a natural action on tabloids, sending $\ol{i}$ to $\ol{i + 1}$ for all $i$.
\begin{prop}
\label{prop:symmetries}
Fix a partition $\lambda$.  The action of $s$ is a graph automorphism of the KL DEG $\A_\lambda$.
\end{prop}
\begin{proof}
It is immediate from the definitions that $s$ commutes with Knuth moves.
\end{proof}

\section{Monodromy}
\label{sec:monodromy}

In this section and the next, we study the Kazhdan-Lusztig molecules of $\widetilde{S_n}$: given a permutation $w$ corresponding to a triple $(P, Q, \rho)$ under AMBC, we specify exactly which triples $(P', Q', \rho')$ correspond to a permutation $w'$ in the same Knuth equivalence class as $w$.  By Theorem~\ref{thm:Knuth move action}, it is necessary that $P' = P$.  In this section, we ask which values of $\rho'$ are possible if also $Q' = Q$; in the next section, we consider which tabloids $Q'$ are possible.  Together, the answers to these questions provide a complete classification.

Our approach is to use Lemma~\ref{lem:graph covering}: if there is a sequence $(w = w_1, w_2, \ldots, w_k = w')$ of permutations in which each consecutive pair is connected by a Knuth move and $Q(w) = Q(w')$, we may project this sequence  via $w_i \mapsto Q(w_i) =: T_i$ to a sequence $L = (T_1,\dots, T_k)$ of tabloids in the KL DEG $\A_\lambda$ for the associated shape $\lambda$.  By Lemma~\ref{lem:graph covering}, consecutive elements of $L$ are connected by Knuth moves, and by hypothesis $T_1 = T_k$, so this sequence is a loop (i.e., a closed walk on the graph).  Moreover, by  Theorem~\ref{thm:Knuth move action}, the \emph{weight change vector} $\rho(w') - \rho(w)$ depends only on the loop $L$.  This suggests the following definition.

\begin{figure}
\centering
\resizebox{.7\textwidth}{!}{\input{figures/block_action_example.pspdftex}}
\caption{A loop in $\A_{\langle 2, 1, 1\rangle}$ and the corresponding weight change vector.}
\label{fig:block action}
\end{figure}
\begin{defn}
Given a permutation $w \in \widetilde{S_n}$, the \emph{monodromy group} $G_w$ based at $w$ is the set 
\[
G_w := \{\rho(w_k) - \rho(w)\}_{(w = w_1,\dots, w_k)},
\]
where the indexing tuple $(w_1,\dots,w_k)$ runs over all sequences of permutations such that $Q(w_k) = Q(w)$ and $w_{i+1}$ is connected by a Knuth move to $w_{i}$ for all $i$.
\end{defn}
Implicit in this definition is the fact that $G_w$ is an (abelian) group under vector addition. This is easy to see using Theorem~\ref{thm:Knuth move action}: the concatenation of two loops in a KL DEG has a lift whose monodromy group element is the sum of the lifts of the two individual loops.

\begin{ex}
\label{ex:monodromies act in blocks}
Figure \ref{fig:block action} shows a loop in the KL DEG of shape $\langle 2, 1, 1\rangle$, beginning and ending at the column superstandard tabloid with start at $1$.  The red vectors give the weight change at each step of the loop. The monodromy element of any lift of the loop is $(2,-1,-1)$. 
\end{ex}

\begin{rmk}
\label{rmk:basepoint independent}
If $Q(w)$ and $Q(w')$ are connected in the KL DEG, then it follows by the standard topological argument (again, using Theorem~\ref{thm:Knuth move action}) that $G_w = G_{w'}$ as subsets of $\Z^{\ell(\lambda)}$. Thus monodromy is base-point independent on connected components of KL DEGs.  (In fact, one consequence of Theorem \ref{thm:monodromy description}, the main result of this section, is that $G_w$ only depends on the shape of $Q(w)$.)
\end{rmk}

The next lemma shows it suffices to understand $G_w$ when $Q(w)$ is column superstandard.

\begin{lemma}
\label{lem:component contains superstandard}
Every connected component of a KL DEG contains reverse row superstandard tabloid and a column superstandard tabloid.
\end{lemma}
\begin{proof}
The Shi combing procedure (originally from \cite{shi}, and explained in our language in \cite[\S9]{cpy}) produces a sequence of Knuth moves from any permutation to a permutation whose right descent set is a singleton. By Theorem~\ref{thm:Knuth move action} and Proposition~\ref{prop:descents}, this can be used to produce a sequence of Knuth moves from any tabloid to a tabloid whose $\tau$-invariant is a singleton.  The tabloids with a unique descent are exactly the reverse row superstandard tabloids.  This completes the first claim.

By the first claim, each column superstandard tabloid lies in the same connected component as some reverse row superstandard tabloid. The automorphism $s$ of Proposition~\ref{prop:symmetries} acts transitively on the classes of column superstandard and reverse row superstandard tabloids; applying it to our sequence of Knuth move shows that, in fact, every reverse row superstandard tabloid has a column superstandard tabloid in its connected component.  This completes the second claim.
\end{proof}

\subsection{Bounding the monodromy group from above}

In this section, we establish some restrictions on the monodromy group which, in particular, imply that the monodromy group for rectangular shapes in trivial. In the next section, we will prove that these are the only restrictions. 

\begin{prop}
\label{prop:monodromies act in blocks}
Suppose $\lambda$ is a partition and $w$ is a permutation of shape $\lambda$. Suppose $\alpha\in G_w$. If $\lambda_i = \lambda_{i+1}$ then $\alpha_i = \alpha_{i+1}$.
\end{prop}

For example, in Example \ref{ex:monodromies act in blocks} the coordinates of the monodromy element corresponding to the two rows of length $1$ are equal. 

\begin{proof}
If this were not the case then repeating the monodromy or its inverse would eventually lead to a weight $\rho$ for which the difference $\rho_i - \rho_{i + 1}$ is arbitrarily large.  From Definition~\ref{def:dominance}, such a weight is non-dominant. However, since the image of AMBC is $\dom$, it follows from Theorem~\ref{thm:Knuth move action} that Knuth moves preserve dominance.
\end{proof}

This result has an interesting corollary.
\begin{cor}
If $\lambda$ is a rectangle then for any $w$ of shape $\lambda$, the monodromy group $G_w$ is trivial and the Knuth equivalence class of $w$ is finite.
\end{cor}
\begin{proof}
Each Knuth move preserves the sum of the entries of $\rho$.  Thus, every element in $G_w$ has sum of coordinates $0$.  By Proposition~\ref{prop:monodromies act in blocks}, if $\lambda$ has parts of a unique size, then every element in $G_w$ has all coordinates equal.  Putting the two together gives the first half of the claim.  For the second half, there are only finitely many choices of a tabloid $Q$ of a given shape, and by the first half of the claim there is at most one choice of $\rho$ in the equivalence class once $P, Q$ are fixed.
\end{proof}

We will see as part of Theorem \ref{thm:monodromy description} that for any non-rectangular shape, there exists a non-trivial monodromy, so the corresponding Kazhdan-Lusztig molecule is infinite.

\begin{defn}
\label{defn:gup}
Suppose $\lambda = \langle r_1^{m_1}, \ldots, r_k^{m_k} \rangle$ is a partition whose distinct row-lengths are $r_1 > \dots > r_k > 0$, of multiplicities $m_1,\dots, m_k$, and let $m'_i = m_1 + \ldots + m_i$ be the distinct part-sizes of $\lambda'$.  For $1\leqslant i\leqslant k$, let $e_i\in\Z^{\ell(\lambda)}$ be the vector with $1$'s in its first $m'_i$ rows and $0$'s afterward.
Define
\begin{equation}
\label{eq:gup}
\Gup_\lambda = \left\{\sum_{i=1}^k a_i e_i \;\middle|\; a_1,\dots ,a_k\in\Z \text{ and } \sum_{i=1}^k a_i m'_i = 0\right\},
\end{equation}
\end{defn}

\begin{ex}
\label{ex:Gup}
For the shape $\lambda = \langle 3^2, 2^3, 1\rangle$, one has $e_1 = (1, 1, 0, 0, 0, 0)$, $e_2 = (1, 1, 1, 1, 1, 0)$, and $e_3 = (1, 1, 1, 1, 1, 1)$. The space $\Gup_\lambda$ consists of vectors in the $\Z$-span of the $e_i$ whose sum of coordinates is $0$, such as $3e_1 - 6e_2 + 4e_3 = (1, 1, -2, -2, -2, 4)$.
\end{ex}

\begin{prop}
\label{prop:easy containment}
We have $G_w\subseteq \Gup_\lambda$. 
\end{prop}
\begin{proof}
Taking $e_0 = 0$ by convention, the vector $e_i - e_{i - 1}$ has $1$'s exactly in those indices for which the corresponding row of $\lambda$ has length $r_i$.  Thus, $\Gup_\lambda$ contains all integer vectors $\alpha$ with entry sum $0$ such that $\alpha_i = \alpha_{i + 1}$ whenever $\lambda_i = \lambda_{i + 1}$.  Then the result follows from Proposition~\ref{prop:monodromies act in blocks}.
\end{proof}
The next section is devoted to proving the reverse inclusion $G_w\supseteq \Gup_\lambda$, and thus equality.

\subsection{Explicit description of monodromy generators}

In this section, we construct an explicit set of monodromies that generate all of $\Gup_\lambda$. We start by describing a $\Z$-basis for the lattice of solutions of a linear Diophantine equation in many variables.

\subsubsection{Linear Diophantine equations in many variables}

Let $m'_1,\dots, m'_k\in\Z$. Consider the lattice (in $\Z^k$) of solutions to the equation
\begin{equation}
\label{eqn:sumzero}
\sum_{i=1}^k m'_ix_i = 0.
\end{equation}
Some solutions of this equation are obvious.

\begin{defn}
\label{def:x^ij}
For $1\leqslant i < j\leqslant k$, define $\mathbf{x}^{(i,j)}\in\Z^k$ by
\[
x^{(i,j)}_l = \begin{cases}
-\frac{m'_j}{\gcd(m'_i, m'_j)}  & \text{if } l = i,\\
\frac{m'_i}{\gcd(m'_i, m'_j)} & \text{if } l = j,\\
0                           & \text{otherwise}.
\end{cases}
\]
\end{defn}

The main result of this section is that these obvious solutions generate the full lattice.  (We suspect that this result is known, but have not been able to find it in the literature.)
\begin{prop}
\label{prop:spanning_set}
The lattice of solutions of \eqref{eqn:sumzero} is the $\Z$-span of $\mathbf{x}^{(i,j)}$.
\end{prop}

\begin{ex}
Consider the equation 
\[
3x+2y+6z = 0.
\]
The three obvious solutions are $(-2, 3, 0)$, $(0, -3, 1)$, and $(-2, 0, 1)$.
The claim is that these three vectors span the (2-dimensional) lattice of solutions. 
\end{ex}

We begin the proof of Proposition~\ref{prop:spanning_set} with a technical lemma.
\begin{lemma}
\label{lemma:diophantine}
For any integers $m'_1, \ldots, m'_j$,  we have
\[
\gcd\left(\frac{m'_1}{\gcd(m'_1, m'_j)}, \dots, \frac{m'_{j-1}}{\gcd(m'_{j-1}, m'_j)}\right) = \frac{\gcd(m'_1,\dots, m'_{j-1})}{\gcd(m'_1,\dots, m'_j)}.
\]
\end{lemma}
\begin{proof}
Consider a prime $p$. For each $1\leqslant i \leqslant j$, denote by $v_i$ the maximal power of $p$ dividing $m'_i$.  First, suppose that $v_j > v_i$ for some $i$.  Then $p$ does not divide $\frac{m'_i}{\gcd(m'_i, m'_j)}$, and similarly $p$ does not divide $\frac{\gcd(m'_1,\dots, m'_{j-1})}{\gcd(m'_1,\dots, m'_j)}$.  Thus $p$ does not divide either side of the claimed equality.

Second, suppose $v_j\leqslant \min\{v_1,\dots,v_{j-1}\}$. Then the maximal power of $p$ that divides the right-hand side is $\min\{v_1,\dots,v_{j-1}\} - v_j$, while the maximal power of $p$ that divides the left-hand side is $\min\{v_1 - v_j, \ldots, v_{j - 1} - v_j\} = \min\{v_1,\dots,v_{j-1}\} - v_j$.  Since these are equal and $p$ is arbitrary, the result follows.
\end{proof}

\begin{proof}[Proof of Proposition \ref{prop:spanning_set}]
Suppose we have some solution $\mathbf{x}^{(0)} = (x_1^{(0)},\dots, x_k^{(0)})$ of \eqref{eqn:sumzero}. Clearly $m'_kx^{(0)}_k$ is divisible by $\gcd(m'_1,\dots, m'_{k-1})$.  Thus $x^{(0)}_k$ is divisible by $\frac{\gcd(m'_1,\dots, m'_{k-1})}{\gcd(m'_1,\dots, m'_k)}$.  By Lemma~\ref{lemma:diophantine} and the Euclidean algorithm, there are integers $c_1,\dots, c_{k-1}$ such that 
\[
\sum_{i=1}^{k-1}c_i \cdot \frac{m'_i}{\gcd(m'_i, m'_k)} = x^{(0)}_k.
\]
Then 
\[
\mathbf{x}^{(1)} := \mathbf{x}^{(0)} - \sum_{i=1}^{k-1} c_i \cdot \mathbf{x}^{(i,k)}
\]
is a solution of \eqref{eqn:sumzero} with the last entry $x_k^{(1)} = 0$. Repeating the argument to successively eliminate coordinates produces an expression for $\mathbf{x}^{(0)}$ as an integer linear combination of the $\mathbf{x}^{(i,j)}$, as desired.
\end{proof}

\subsubsection{Column rebases}

In this section and the next, we provide the building blocks necessary to produce an explicit set of paths in $\A_\lambda$ that, when lifted to permutations, correspond to elements in the monodromy group $G_w$ and generate $\Gup_\lambda$.  By Remark~\ref{rmk:basepoint independent} and Lemma~\ref{lem:component contains superstandard}, it suffices to understand the situation in which we start with a column superstandard tabloid.  Thus, our constructions are adapted to this case.

It is convenient to work not with tabloids but with associated tableaux, filled with equivalence classes modulo $n$, that arise by ordering the entries of each row into columns.  We use usual matrix coordinates to refer to entries of tableaux, so that $(1, 3)$ is the entry in the first (= top = longest) row and third column.

Our first construction, called a \emph{column rebase}, is defined below; the definition is illustrated in Example~\ref{ex:rebase} and Figure~\ref{fig:rebase_sequence}.

\begin{definition}
Suppose that $T$ is a tabloid and $\widetilde{T}$ is a corresponding tableau such that, for some column indices $c_1, c_2$ and some nonnegative integers $i, k, l$, column $c_1$ of $\widetilde{T}$ is filled (in order from top to bottom) with the elements $\ol{i+1}, \ol{i+2}, \dots, \ol{i+k+l}$ and column $c_2$ is filled with the elements $\ol{i+k+l+1}$, $\ol{i+k+l+2}$, \ldots, $\ol{i+k+2l}$.  Then the \emph{forward column rebase} of \emph{type} $(i, k, l)$ of $\widetilde{T}$ is the sequence of tableaux $(\widetilde{T} = \widetilde{T}_0, \widetilde{T}_1, \ldots, \widetilde{T}_{(k + l) \cdot l})$, where $\widetilde{T}_{j}$ is formed from $\widetilde{T}_{j - 1}$ by switching the elements in positions $((k + l)m - j + 1, c_1)$ and $(m, c_2)$, where $m = \lceil j/(k + l)\rceil$ is the unique integer such that
\[
(k + l) \cdot (m - 1) + 1 \leq j \leq (k + l) \cdot m.
\] 
We say that the reversed sequence is the \emph{backward column rebase} of $\widetilde{T}' := \widetilde{T}_{(k + l) \cdot l}$.  The \emph{forward column rebase} of $T$ is the sequence of tabloids that we get by forgetting the row orders in every tableau in the forward column rebase of $\widetilde{T}$, and similarly for the \emph{backward column rebase} of the tabloid $T'$ associated to $\widetilde{T}'$.
\end{definition}

\begin{ex} 
\label{ex:rebase}
Below is a forward column rebase of type $(6,2,2)$. The affected columns are the first and third ones.
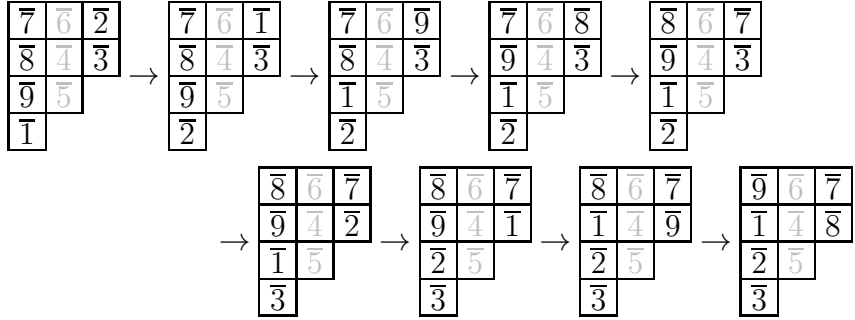
\[
\hspace{-2cm}
\tableau[sY]{\ol{7}, \textcolor{lightgray}{\ol{6}}, \ol{2}\\ \ol{8}, \textcolor{lightgray}{\ol{4}}, \ol{3}\\ \ol{9}, \textcolor{lightgray}{\ol{5}}\\ \ol{1}}\to
    \tableau[sY]{\ol{7}, \textcolor{lightgray}{\ol{6}}, \ol{1}\\ \ol{8}, \textcolor{lightgray}{\ol{4}}, \ol{3}\\ \ol{9}, \textcolor{lightgray}{\ol{5}}\\ \ol{2}}\to
    \tableau[sY]{\ol{7}, \textcolor{lightgray}{\ol{6}}, \ol{9}\\ \ol{8}, \textcolor{lightgray}{\ol{4}}, \ol{3}\\ \ol{1}, \textcolor{lightgray}{\ol{5}}\\ \ol{2}}\to
    \tableau[sY]{\ol{7}, \textcolor{lightgray}{\ol{6}}, \ol{8}\\ \ol{9}, \textcolor{lightgray}{\ol{4}}, \ol{3}\\ \ol{1}, \textcolor{lightgray}{\ol{5}}\\ \ol{2}}\to
    \tableau[sY]{\ol{8}, \textcolor{lightgray}{\ol{6}}, \ol{7}\\ \ol{9}, \textcolor{lightgray}{\ol{4}}, \ol{3}\\ \ol{1}, \textcolor{lightgray}{\ol{5}}\\ \ol{2}}
 \]
\[
\hspace{2cm}
   \to\tableau[sY]{\ol{8}, \textcolor{lightgray}{\ol{6}}, \ol{7}\\ \ol{9}, \textcolor{lightgray}{\ol{4}}, \ol{2}\\ \ol{1}, \textcolor{lightgray}{\ol{5}}\\ \ol{3}}\to
    \tableau[sY]{\ol{8}, \textcolor{lightgray}{\ol{6}}, \ol{7}\\ \ol{9}, \textcolor{lightgray}{\ol{4}}, \ol{1}\\ \ol{2}, \textcolor{lightgray}{\ol{5}}\\ \ol{3}}\to
    \tableau[sY]{\ol{8}, \textcolor{lightgray}{\ol{6}}, \ol{7}\\ \ol{1}, \textcolor{lightgray}{\ol{4}}, \ol{9}\\ \ol{2}, \textcolor{lightgray}{\ol{5}}\\ \ol{3}}\to
    \tableau[sY]{\ol{9}, \textcolor{lightgray}{\ol{6}}, \ol{7}\\ \ol{1}, \textcolor{lightgray}{\ol{4}}, \ol{8}\\ \ol{2}, \textcolor{lightgray}{\ol{5}}\\ \ol{3}}
\]
\end{ex}

\begin{figure}
\centering
\begin{minipage}[t]{0.5\textwidth}
\centering
\resizebox{!}{.19\textheight}{\input{figures/exchange_sequence_1.pspdftex}}
\caption{Sequence in which exchanges are to be performed.}
\label{fig:rebase_sequence}
\end{minipage}\hfill
\begin{minipage}[t]{0.5\textwidth}
\centering
\resizebox{!}{.19\textheight}{\input{figures/cyclic_order_two_col.pspdftex}}
\caption{A cyclic order on the cells of two columns.}
\label{fig:cyclic order on two columns}
\end{minipage}
\end{figure}

\begin{prop}
In the forward column rebase $(T_0, \ldots, T_{(k + l)\cdot l})$ of type $(i, k, l)$ of the tabloid $T_0$, every two consecutive tabloids are either equal or connected by a Knuth move.
\end{prop}
\begin{proof}
Let $(\widetilde{T}_0, \ldots, \widetilde{T}_{(k + l)\cdot l})$ be the associated sequence of tableaux, and let $c_1$ and $c_2$ be the indices of the relevant columns.  Choose $j \in [1, (k + l) \cdot l]$ and write $j = (k + l) \cdot (m - 1) + j'$ for integers $m, j'$ with $1 \leqslant m \leqslant l$ and $1 \leqslant j' \leqslant k + l$.  It is straightforward to show by induction that column $c_1$ of tableau $\widetilde{T}_j$ consists of the entries
\begin{multline*}
\ol{i + m}, \ol{i + m + 1}, \ldots, \ol{i + m + k + l - j' - 1}, \\
\ol{i + m + k + l - j' + 1}, \ol{i + m + k + l - j' + 2}, \ldots, \ol{i + m + k + l}
\end{multline*}
in order from top to bottom, while column $c_2$ consists of the entries
\begin{multline*}
\ol{i + 1}, \ol{i + 2}, \ldots, \ol{i + m - 1}, 
\qquad
\ol{i + m + k + l - j'},\\ 
\ol{i + m + k + l + 1},
\ol{i + m + k + l + 2}, \ldots, \ol{i + k + 2l}.
\end{multline*}
The last swap taken was between the entries $\ol{i + m + k + l - j'}$ and $\ol{i + m + k + l - j' + 1}$ in positions $(k + l - j' + 1, c_1)$ and $(m, c_2)$.  If $k + l - j' + 1 = m$ then the corresponding tabloids are equal.  If $k + l - j' +1 > m$ then $\tau(T_{j - 1})$ includes $\ol{i + m + k + l - j' - 1}$ but not $\ol{i + m + k + l - j'}$ while $\tau(T_{j})$ includes $\ol{i + m + k + l - j'}$ but not $\ol{i + m + k + l - j' - 1}$, and so the swap is a Knuth move.  Similarly, if $k + l - j' + 1 < m$ then $\tau(T_{j - 1})$ includes $\ol{i + m + k + l - j'}$ but not $\ol{i + m + k + l - j' + 1}$ while $\tau(T_{j})$ includes $\ol{i + m + k + l - j' + 1}$ but not $\ol{i + m + k + l - j'}$, and so the swap is a Knuth move. 
This completes the proof.
\end{proof}

It follows from the preceding proof that in the final tabloid $T_{(k + l) \cdot l}$ of a forward column rebase, the shorter column $c_2$ is filled (in order from top to bottom) with the elements $\ol{i + 1}$, $\ol{i + 2}, \ldots, \ol{i + l}$ while the longer column $c_1$ is filled with the elements $\ol{i + l + 1}$, $\ol{i + l + 2}$, \ldots, $\ol{i + k + 2l}$.

We now describe how weight changes under a column rebase. We will need two pieces of notation.
\begin{defn}
Suppose we have a partition $\lambda$ and a positive integer $l \leqslant \ell(\lambda)$. Let $\mathbf{u}^{l}_{\textup{col}}\in\Z^{\ell(\lambda)}$ be the vector with $1$ in the first $l$ rows and $0$ in the remaining rows.  (There will never be ambiguity about the shape $\lambda$, so we suppress it from the notation.)
\end{defn}
The second piece is somewhat more involved. Consider a tableau $\widetilde{T}$ of shape $\lambda$ such that one can apply a column rebase of type $(i, k, l)$. Informally, the vector $\mathbf{u}_{\textup{disp}}^{\widetilde{T},l}$ keeps track of how far $1$ shifts in the process of a column rebase. If $\ol{1}$ does not lie in the two columns of interest, define $\mathbf{u}_{\textup{disp}}^{\widetilde{T},l} := \mathbf{u}^{l}_{\textup{col}}$.  If $\ol{1}$ does lie in one of these columns, consider the cyclic order on the columns shown in Figure \ref{fig:cyclic order on two columns}.  Let $F$ be the set of $l$ cells following the cell containing $\ol{1}$ in this cyclic order, not including the cell itself, and define $\mathbf{u}_{\textup{disp}}^{\widetilde{T},l}\in\Z^{\ell(\lambda)}$ by
\[
\left(\mathbf{u}_{\textup{disp}}^{\widetilde{T},l}\right)_m := \text{number of cells of $F$ in row $m$.}
\]  

\begin{ex}
Consider the column rebase from Example \ref{ex:rebase}, so $l = 2$. In this case the collection $F$ consists of the second and third cell in the first column. Thus
\[
\mathbf{u}^{l}_{\textup{col}} =  \begin{pmatrix}1\\1\\0\\0\end{pmatrix}
\quad
\textrm{ and }
\quad
\mathbf{u}_{\textup{disp}}^{\widetilde{T},l} = \begin{pmatrix}0\\1\\1\\0\end{pmatrix}.
\]
\end{ex}

\begin{prop}
\label{prop:rebase weight change}
Suppose $T$ is a tabloid of shape $\lambda$ to which one can apply a forward column rebase of type $(i,k,l)$, with associated tableau $\widetilde{T}$.  Then the weight change vector of the rebase is 
\[
\mathbf{u}^{l}_{\textup{col}} - \mathbf{u}_{\textup{disp}}^{\widetilde{T},l}.
\]
\end{prop}

\begin{proof}
If $\ol{1}$ is not in either of the two columns involved in the rebase, then by Theorem~\ref{thm:Knuth move action} the weight change is $0$ at each step, as claimed.  Otherwise, $\ol{1}$ lies somewhere in the two columns of interest.  There are three cases, depending on the position of $\ol{1}$; these are pictured in Figure \ref{fig:rebase_proof_cases}.  The three cases involve similar considerations, so we provide a detailed analysis of just one of them.


\begin{figure}
\begin{minipage}[t]{0.4\textwidth}
\centering
\resizebox{!}{.125\textheight}{\input{figures/rebase_proof_cases.pspdftex}}
\caption{The cases for the position of $\ol{1}$ in the proof of Proposition~\ref{prop:rebase weight change}.}
\label{fig:rebase_proof_cases}
\end{minipage}\hfill
\begin{minipage}[t]{0.6\textwidth}
\centering
\resizebox{!}{.125\textheight}{\input{figures/rebase_weights.pspdftex}}
\caption{The action of a forward column rebase on weights, depending on the initial position of $\ol{1}$.}
\label{fig:rebase_weights}
\end{minipage}
\end{figure}


Suppose that $\ol{1}$ is in the shorter of the two relevant columns in $\widetilde{T}$, i.e., the column that does not contain $\ol{i + 1}$, and let $r$ be the index of the row in which it appears.  (In Figure~\ref{fig:rebase_proof_cases}, this is the blue region; in Figure~\ref{fig:rebase_weights}, it corresponds to the right panel.)  The first $(k + l) \cdot (r - 1)$ steps of the column rebase do not involve the cell containing $\ol{1}$, and so by Theorem~\ref{thm:Knuth move action} do not change the weight. In step $(k + l) \cdot (r - 1) + 1$, the entry $\ol{1}$ in row $r$ is swapped with the entry $\ol{n}$ in row $k + l$; this adds $1$ to $\rho_r$ and subtracts $1$ from $\rho_{k + l}$.  The next $k + l - 1$ swaps to do not involve the entry $\ol{1}$.  Finally, for $m = 1, \ldots, l - r$, we have that in the $m$th remaining set of $k + l$ swaps, the entry $\ol{1}$ begins in row $k + l - m + 1$, is swapped with the entry $\ol{2}$ in row $r + m$, and then is swapped with entry $\ol{n}$ in row $k + l - m$.  This last swap is the only one that affects the weight $\rho$: it adds $1$ to $\rho_{m + r}$ and subtracts $1$ from $\rho_{k + l - m}$.  (This is illustrated in Figure~\ref{fig:rebase_weights}.)  The net effect of these additions is to add $1$ to $\rho$ in rows $r, r + 1, \ldots, l$ and to subtract $1$ from $\rho$ in rows $k + l, k + l - 1, \ldots, k + r$.  This coincides exactly with $\mathbf{u}^{l}_{\textup{col}} - \mathbf{u}_{\textup{disp}}^{\widetilde{T},l}$.
\end{proof}

\subsubsection{Column cycles}

In this section, we make a second step towards explicit monodromy generators by stitching column rebases together into longer paths in the KL DEG.  Again, for the terminological convenience of being able to talk about columns, we often deal with tableaux.

\begin{defn}
Suppose $\lambda$ is a partition and $i\in\Z$. Let $\lambda'$ be the conjugate partition of $\lambda$. The \emph{column superstandard tableau} of shape $\lambda$ with \emph{start} $i$ is the tableau with whose first column contains in order from top to bottom the entries $\ol{i}, \ol{i+1}, \dots, \ol{i+\lambda_1'-1}$, whose second column contains
$\ol{i+\lambda_1'}, \ol{i+\lambda_1'+1}, \dots, \ol{i+\lambda_1'+\lambda_2'-1}$, and so on.
\end{defn}
Note that our definition differs slightly from the usual one, in which there is only a single column superstandard tableau of each shape (with start $1$). For every column superstandard tabloid, it is possible to assign the entries of each row to columns in order to produce a column superstandard tableau, and forgetting the row assignments in a column superstandard tableau leaves a column superstandard tabloid.  An example of a column superstandard tableau is given in Figure~\ref{fig:following_cells}.


\begin{defn}
Suppose $T$ is the column superstandard tabloid of shape $\lambda$ with start $\ol{i}$, and $\widetilde{T}$ is the corresponding tableau. The \emph{forward column cycle} of $\widetilde{T}$ with respect to column $j$ is the following sequence of tableaux: it is the concatenation of the forward column rebase of $\widetilde{T}$ on the $(j-1)$-st and $j$-th columns, the forward column rebase of the resulting tableau on the $(j-2)$-nd and $j$-th columns, etc., through a forward column rebase on the first and $j$-th columns, followed by the reverse column rebase on the $j$-th and $\lambda_1$-th (last) columns, then the reverse column rebase on the $j$-th and $(\lambda_1-1)$-st columns, etc., through the reverse column rebase on the $j$-th and $(j+1)$-st columns. 

The \emph{forward column cycle} on $T$ is the sequence of tabloids that one gets by forgetting the row orders in every tableau in the forward column cycle of $\widetilde{T}$.  The reversed sequences of tableaux and tabloids are called \emph{reverse column cycles}.
\end{defn}

\begin{ex} 
\label{ex:column cycle}
A forward column cycle of a tableau with respect to the second column is shown below; the first step is a forward column rebase, and the latter two are reverse column rebases:
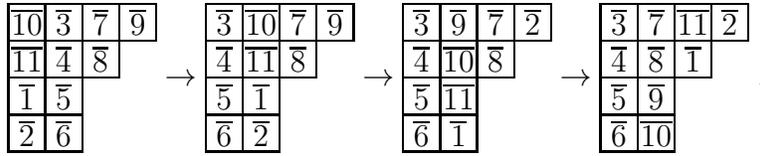
\[
\tableau[sY]{\ol{10}, \ol{3}, \ol{7}, \ol{9}\\ \ol{11}, \ol{4}, \ol{8}\\ \ol{1}, \ol{5} \\ \ol{2}, \ol{6}} 
\to
\tableau[sY]{\ol{3}, \ol{10}, \ol{7}, \ol{9}\\ \ol{4}, \ol{11}, \ol{8}\\ \ol{5}, \ol{1} \\ \ol{6}, \ol{2}} 
\to
\tableau[sY]{\ol{3}, \ol{9}, \ol{7}, \ol{2}\\ \ol{4}, \ol{10}, \ol{8}\\ \ol{5}, \ol{11} \\ \ol{6}, \ol{1}} 
\to
\tableau[sY]{\ol{3}, \ol{7}, \ol{11}, \ol{2}\\ \ol{4}, \ol{8}, \ol{1}\\ \ol{5}, \ol{9} \\ \ol{6}, \ol{10}} 
\; .
\]
\end{ex}

\begin{prop}
\label{prop:column cycles}
Forward column cycles are well defined.  Furthermore, if $\widetilde{T}$ is the column superstandard tableau of shape $\lambda$ and start $\ol{i}$, then the last tableau in the forward column cycle of $\widetilde{T}$ is the column superstandard tableau of shape $\lambda$ and start $\ol{i + \lambda'_j}$.
\end{prop}
\begin{proof}
It is conceptually straightforward but notationally cumbersome to check by induction that, after each rebase, every column consists of an interval in $[\ol{n}]$ and the two columns that are relevant for the next rebase have entries that are related in the appropriate way.  We leave the details to the reader. 
\end{proof}

\begin{figure}
\begin{minipage}[t]{0.5\textwidth}
\centering
\resizebox{!}{.14\textheight}{\input{figures/cyclic_order.pspdftex}}
\caption{The cyclic order on the partition $\langle 4, 3, 2^2\rangle$. }
\label{fig:disp_cyc}
\end{minipage}\hfill
\begin{minipage}[t]{0.5\textwidth}
\centering
\resizebox{!}{.14\textheight}{\input{figures/displacement.pspdftex}}
\caption{The column superstandard tableau of shape $\langle 4, 3, 2^2\rangle$ with start at $\ol{10}$.  The four cells following $\ol{1}$ are boxed in red.}
\label{fig:following_cells}
\end{minipage}
\end{figure}

Now we describe the effect of a column cycle on the weight. As before, we start by introducing a vector tracking the displacement of $\ol{1}$. Suppose $\widetilde{T}$ is a column superstandard tableau of shape $\lambda$ and $k$ is an integer. Consider the cyclic order on the cells of $\lambda$ illustrated in Figure~\ref{fig:disp_cyc}: every cell not in the first row is followed by the cell directly above it, the top cell in the first column is followed by the bottom cell in the last column, and the top cell in every other column is followed by the bottom cell in the previous column. Denote by $F$ the collection of $k$ cells following the cell of $\widetilde{T}$ containing $\ol{1}$ in this order. Define $\mathbf{v}_{\textup{disp}}^{\widetilde{T},k}\in\Z^{\ell(\lambda)}$ by
\[
\left(\mathbf{v}_{\textup{disp}}^{\widetilde{T},k}\right)_m := \text{number of cells of $F$ in row $m$.}
\]  

\begin{ex}
\label{ex:cycle_displacement}
Suppose $\widetilde{T}$ is the column superstandard tableau of shape $\langle 4, 3, 2^2 \rangle$ with start $\ol{10}$; it is shown in Figure \ref{fig:following_cells}. The four cells following the cell with $\ol{1}$ are shown in the figure with red squares. Thus
$
\mathbf{v}_{\textup{disp}}^{\widetilde{T}, 4} = \begin{pmatrix}2\\2\\0\\0\end{pmatrix}.
$
\end{ex}

\begin{prop}
\label{prop:cycle_weichgt_change}
Suppose $T$ is a column superstandard tabloid of shape $\lambda$, with corresponding tableau $\widetilde{T}$. Consider performing a forward column cycle with respect to the $j$-th column. Then the weight change vector is
\[
\Delta_{\textup{wt}} = \mathbf{u}_{\textup{col}}^{\lambda_j'} - \mathbf{v}_{\textup{disp}}^{\widetilde{T},\lambda_j'}.
\]
\end{prop}

\begin{ex}
Consider the column cycle of Example~\ref{ex:column cycle}. Computing the weight change for the three column rebases via Proposition~\ref{prop:rebase weight change}, the total weight change vector is
\[
\Delta_{\textup{wt}} =\begin{pmatrix}0\\0\\0\\0\end{pmatrix} + \begin{pmatrix}-1\\0\\1\\0\end{pmatrix} + \begin{pmatrix} 0\\-1\\0\\1 \end{pmatrix} = \begin{pmatrix}-1\\-1\\1\\1\end{pmatrix}.
\]
On the other hand, the vector $\mathbf{v}_{\textup{disp}}^{\widetilde{T}, 4}$ was computed in Example \ref{ex:cycle_displacement}. Since we are dealing with the longest column, the vector $\mathbf{u}_{\textup{col}}^{\lambda_1'}$ consists of all $1$'s. Thus by Proposition~\ref{prop:cycle_weichgt_change}, the weight change is 
\[
\Delta_{\textup{wt}} =\begin{pmatrix}1\\1\\1\\1\end{pmatrix} - \begin{pmatrix}2\\2\\0\\0\end{pmatrix} = \begin{pmatrix}-1\\-1\\1\\1\end{pmatrix}.
\]
\end{ex}

\begin{proof}[Proof of Proposition \ref{prop:cycle_weichgt_change}]
There are five cases, depending on where $\ol{1}$ is in the tabloid: whether it is in the first $\lambda'_j$ rows or not, and if so whether it is in the first column, between column $1$ and column $j$, in column $j$, or in a larger-numbered column. These are illustrated in Figure \ref{fig:cycle_cases}; we refer to them below by their color in this figure. (If $j = 2$, the blue case does not occur; if $j = 1$, the blue and purple cases do not occur.)

The green case, when $\ol{1}$ is in a row lower than $\lambda_j'$, is the easiest one; in this case $\ol{1}$ is only involved in a single column rebase, and it is easy to see that $\mathbf{u}_{\textup{disp}}^{\widetilde{T},\lambda_j'}=\mathbf{v}_{\textup{disp}}^{\widetilde{T},\lambda_j'}$. This gives the desired result.

In the blue case, when $\ol{1}$ is in the first $\lambda'_j$ rows in a column between column $1$ and column $j$, there are two column rebases affecting $\ol{1}$. The first one is	
of the kind portrayed on the left of Figure \ref{fig:rebase_weights}; in such a rebase, the vectors $\mathbf{u}_{\textup{col}}^{\lambda_j'}$ and $\mathbf{u}_{\textup{disp}}^{\widetilde{T},\lambda_j'}$ are equal, and so the weight does not change. In the second rebase, the shorter column involved has length $\lambda'_j$.  In addition, in this rebase the longer column is directly to the right of the one in which $\ol{1}$ originally started; thus, $\mathbf{u}_{\textup{disp}}^{\widetilde{T},\lambda_j'}$ for the second rebase is equal to $\mathbf{v}_{\textup{disp}}^{\widetilde{T},\lambda_j'}$. This gives the desired result.

If $\ol{1}$ is in the purple region (so $j > 1$) then only one rebase affects the position of $\ol{1}$ and it is exactly the same as the second rebase of the blue case. 

In the red and yellow cases, the position of $\ol{1}$ can be affected by more than two column rebases, making these more complicated to analyze. These cases are very similar, so we will only consider the yellow case in detail.  Suppose that $\ol{1}$ begins in the yellow region, say in column $j' > j$ and in row $r$. Then the weight does not change until the reverse rebase between column $j$ and column $j'$; during this rebase, $\ol{1}$ moves from column $j'$ to column $j$ and remains in row $r$. What follows is several (possibly none) reverse rebases during which $\ol{1}$ moves lower in column $j$: after the first one it will be in row $r+\lambda'_{j'-1}$, then in row $r+\lambda'_{j'-1}+\lambda'_{j'-2}$, etc. Let $j''$ be the last column for which the reverse rebase affects the weight.  This can happen in one of two ways: either $j'' = j + 1$ and the algorithm ends with $\ol{1}$ in column $j$, or the rebase transfers $\ol{1}$ into column $j''$, where it stays until the end. Let $r''$ be the final row of $\ol{1}$.

If $\ol{1}$ ends up in column $j$, then $\lambda'_j \geqslant r'' = r + \lambda_{j + 1} + \ldots + \lambda_{j' - 1}$ and so by direct examination of the cyclic order on cells of $\lambda$ we have
\begin{align*}
\mathbf{u}_{\textup{col}}^{\lambda_j'} - \mathbf{v}_{\textup{disp}}^{\widetilde{T},\lambda_j'} & = \mathbf{u}_{\textup{col}}^{\lambda_j'} - \left(
\mathbf{u}_{\textup{col}}^{r-1} + 
\mathbf{u}_{\textup{col}}^{\lambda_{j'-1}'} + \ldots + \mathbf{u}_{\textup{col}}^{\lambda_{j+1}'} +  \left(\mathbf{u}_{\textup{col}}^{\lambda_{j}'} - \mathbf{u}_{\textup{col}}^{r'' - 1}\right)\right) \\
& = 
  \mathbf{u}_{\textup{col}}^{r''-1} - \mathbf{u}_{\textup{col}}^{\lambda_{j+1}'} - \mathbf{u}_{\textup{col}}^{\lambda_{j+2}'}
   -\ldots - \mathbf{u}_{\textup{col}}^{\lambda_{j'-1}'} - \mathbf{u}_{\textup{col}}^{r-1}.
\end{align*}
Otherwise, 
\[
\mathbf{u}_{\textup{col}}^{\lambda_j'} - \mathbf{v}_{\textup{disp}}^{\widetilde{T},\lambda_j'} = 
 \mathbf{u}_{\textup{col}}^{\lambda_j'} - (\mathbf{u}_{\textup{col}}^{\lambda_{j''}'}-\mathbf{u}_{\textup{col}}^{r''-1}) - \mathbf{u}_{\textup{col}}^{\lambda_{j''+1}'} - \mathbf{u}_{\textup{col}}^{\lambda_{j''+2}'}
   -\ldots - \mathbf{u}_{\textup{col}}^{\lambda_{j'-1}'} - \mathbf{u}_{\textup{col}}^{r-1}.
\]

On the other hand, we can track the weight change at each stage of the column cycle using Proposition~\ref{prop:rebase weight change}.  Since we are doing reverse column rebases, the contributions from the \textup{col}-vectors are negated relative to the proposition, and they sum to
\[
-\mathbf{u}_{\textup{col}}^{\lambda_{j'}'} - \mathbf{u}_{\textup{col}}^{ \lambda_{j'-1}'}-\ldots -\mathbf{u}_{\textup{col}}^{ \lambda_{j''}'}.
\]
By tracking the position of $\ol{1}$ after each column rebase, we have that if $\ol{1}$ ends up in column $j$, then the total contribution from the \textup{disp}-vectors is 
\[
\mathbf{u}_{\textup{col}}^{ r''-1} + \left(\mathbf{u}_{\textup{col}}^{ \lambda_{j'}'} - \mathbf{u}_{\textup{col}}^{ r-1}\right).\] If instead $\ol{1}$ ends up in column $j'' \neq j$, this contribution is $\mathbf{u}_{\textup{col}}^{ r''-1} + \left(\mathbf{u}_{\textup{col}}^{ \lambda_{j'}'} - \mathbf{u}_{\textup{col}}^{ r-1}\right) + \mathbf{u}_{\textup{col}}^{ \lambda_j'}$. In both cases, the resulting weight change matches the computed value $\mathbf{u}_{\textup{col}}^{ \lambda_j'} - \mathbf{v}_{\textup{disp}}^{\widetilde{T},\lambda_j'}$.
\end{proof}

\begin{figure}
\centering
\resizebox{!}{.2\textheight}{\input{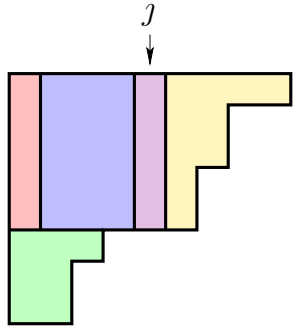}}
\caption{The cases in the proof of Proposition \ref{prop:cycle_weichgt_change}.}
\label{fig:cycle_cases}
\end{figure}

\begin{cor}
\label{cor:doesn't matter which column}
The weight change vectors associated to column cycles with respect to different columns of the same length are equal.
\end{cor}

\subsubsection{Building monodromy group generators from column cycles}

Consider the partition $\lambda = \langle r_1^{m_1}, \ldots, r_k^{m_k} \rangle$.  For $1 \leqslant i \leqslant k$, let $m'_i = m_1 + \ldots + m_i$ be the distinct part-sizes of $\lambda'$.  Recall the space $\Gup_\lambda$ from Definition \ref{defn:gup}. 
%

\begin{thm}
\label{thm:monodromy description}
Suppose $\lambda$ is a partition and $w$ is a permutation of shape $\lambda$. Then $G_w = \Gup_\lambda$.
\end{thm}

\begin{proof}
By Proposition \ref{prop:easy containment}, $G_w\subseteq \Gup_\lambda$. We now prove the opposite containment. 
By \eqref{eq:gup}, there is a natural bijection between elements of $\Gup_\lambda$ and integer solutions $(a_i)_{1 \leq i \leq k}$ of the equation $\sum_{i=1}^k a_i m_i' = 0$. By Proposition \ref{prop:spanning_set}, the solution set of this equation is the $\Z$-span of the vectors $\mathbf{x}^{(i,j)}\in\Z^k$ 
defined in Definition~\ref{def:x^ij}.
The map associating an element of $\Gup_\lambda$ to a solution is obviously $\Z$-linear, so it is sufficient to show that 
\[
\mathbf{v}^{(i,j)} := \frac{m_i'}{\gcd(m_i', m_j')} e_j' - \frac{m_j'}{\gcd(m_i', m_j')} e_i'
\]
is an element of $G_w$.

As explained in Remark \ref{rmk:basepoint independent}, by Lemma \ref{lem:component contains superstandard}, we may restrict ourselves to the case when $Q(w)$ is column superstandard.  Given such a tabloid $T$, consider the following concatenation of column cycles in $\A_\lambda$:
\begin{enumerate}
\item a sequence of $\frac{m_j'}{\gcd(m_i', m_j')}$ reverse column cycles with respect to a column of size $m_i'$, followed by
\item a sequence of $\frac{m_i'}{\gcd(m_i', m_j')}$ forward column cycles with respect to a column of size $m_j'$. 
\end{enumerate} 
(By Corollary~\ref{cor:doesn't matter which column}, it doesn't matter which columns of the given lengths are selected.)  By Proposition~\ref{prop:column cycles}, the final tabloid of this sequence of tabloids is equal to the initial tabloid $T$, i.e., this sequence is a loop in $\A_\lambda$.  By Proposition~\ref{prop:cycle_weichgt_change}, the weight change vector associated to this loop is equal to
\begin{multline*}
\Delta_{\textup{wt}} = 
- \frac{m_j'}{\gcd(m_i', m_j')} \mathbf{u}_{\textup{col}}^{m_i'} + \sum_{k = 0}^{m_j'/\gcd(m_i', m_j') - 1}\mathbf{v}_{\textup{disp}}^{\widetilde{T} + km_i', m_i'} 
+ {} \\
\left(
\frac{m_i'}{\gcd(m_i', m_j')} \mathbf{u}_{\textup{col}}^{m_j'} - \sum_{k = 0}^{m_i'/\gcd(m_i', m_j') - 1}\mathbf{v}_{\textup{disp}}^{\widetilde{T} + m_i'm_j'/\gcd(m_i', m_j') - k m_j', m_j'}\right),
\end{multline*}
where by $\widetilde{T} + m$ we mean the result of adding $m$ to every entry of the tableau $\widetilde{T}$.  Observe that the two sums of \textup{disp}-vectors cancel exactly: the reverse cycles and the forward cycles deal with the same (multi)set of $\frac{m_i'm_j'}{\gcd(m_i', m_j')}$ consecutive cells in the cyclic order.  Thus
\[
\Delta_{\textup{wt}} = \frac{m_i'}{\gcd(m_i', m_j')} \mathbf{u}_{\textup{col}}^{m_j'} - \frac{m_j'}{\gcd(m_i', m_j')} \mathbf{u}_{\textup{col}}^{m_i'} = 
\frac{m_i'}{\gcd(m_i', m_j')} e'_{j} - \frac{m_j'}{\gcd(m_i', m_j')} e'_i = \mathbf{v}^{(i,j)}.
\]
Lifting this cycle from $\A_\lambda$ to the graph on permutations shows that $\mathbf{v}^{(i,j)} \in G_w$ for every $1 \leqslant i < j \leqslant k$. This completes the proof.
\end{proof}

\section{Components of KL DEGs}
\label{sec:charge}

In this section, we finish the description of the Knuth equivalence classes of permutations by describing which tabloids can be reached from a given one by Knuth moves. Using the machinery from the previous section, it is easy to see that certain column superstandard tabloids are connected by Knuth moves; together with the fact that every Knuth class of tabloids contains a superstandard tabloid, this gives an upper bound on the number of connected components. Then we use the charge statistic to produce a matching lower bound.

\begin{defn}
For a partition $\lambda$, define 
\[d_\lambda:=\gcd(\lambda_1', \lambda_2',\dots),\] 
where $\lambda'$ is the conjugate partition of $\lambda$. Equivalently, $d_\lambda$ is the greatest common divisor of the multiplicities of parts of $\lambda$.
\end{defn}

\begin{prop}
\label{prop:upper bound}
The number of connected components of $\A_\lambda$ is no larger than $d_\lambda$.
\end{prop}
\begin{proof}
Column cycles are walks in $\A_\lambda$ that take a column superstandard tabloid to another one, related by adding $\lambda'_j$ to each entry for some $j$. Since $d_\lambda$ is a $\Z$-linear combination of the $\lambda'_j$'s, two column superstandard tabloids that differ by $d_\lambda$ are in the same connected component.  By Lemma~\ref{lem:component contains superstandard}, every connected component of $\A_\lambda$ contains a column superstandard tabloid, so there are at most $d_\lambda$ connected components.
\end{proof}

We seek to prove that this upper bound is in fact the correct value.  Our proof makes use of the charge statistic, which was partially introduced in Section~\ref{sec:charge intro}.


\begin{defn}
Given a tabloid $T$ of shape $\lambda$, extend Definition~\ref{def:local charge} of local charge by defining $\lch_i(T) = 0$ if $\lambda_i \neq \lambda_{i + 1}$.  The \emph{charge} of $T$ is defined by
\[
\ch(T) := \sum_{i = 1}^{\ell(\lambda) - 1} i \cdot \lch_{i}(T).
\]
\end{defn}

\begin{rmk}
The definition of charge given here is not the conventional one (as in, e.g., \cite[\S5]{DalalMorse}). However, it is straightforward to see using Lemma~\ref{lem:local independent of activation} that our definition of the charge of a tabloid $T$ differs by a multiple of $d_\lambda$ from the conventional charge of the word whose $i$-th letter is the index of the row of $T$ containing $\ol{i}$. In light of the statement of Theorem~\ref{thm:connected components} below, this is sufficient for our purposes, and the present definition is easier to work with.
\end{rmk}

Before finishing the description of the connected components of $\A_\lambda$, we need a lemma regarding the charges of column superstandard tabloids.

\begin{lemma}
\label{lem:all charges are non-empty}
Suppose $\lambda$ is a partition of $n$ and $1\leqslant i\leqslant n$. Let $T$ be the column superstandard tabloid of shape $\lambda$ starting at $\ol{i}$, and $T'$ be the column superstandard tabloid of shape $\lambda$ starting at $\ol{i+1}$. Then $\ch(T')\equiv \ch(T) - 1 \pmod{d_\lambda}$.
\end{lemma}
\begin{proof}
By construction, all the local charges for a column superstandard tabloid are $0$, except if $\ol{n}$ is located in the row directly above the row containing $\ol{1}$ and the two rows have the same size, in which case the local charge is $1$.  If $\ol{n}$ is in row $r$ of $T$ and $\lambda_{r - 1} = \lambda_r = \lambda_{r + 1}$ then $\ch(T) = r + 1$ and $\ch(T') = r$, and the result holds.  If $\lambda_{r - 1} = \lambda_{r} > \lambda_{r + 1}$ or if $r = \ell(\lambda)$ then $\ch(T) = 0$ and $\ch(T') = r - 1$.  In this case, $r = \lambda'_i$ for some $i$, so $r \equiv 0 \pmod{d_\lambda}$ and the result holds.  Similar considerations apply if $\lambda_{r - 1} < \lambda_r = \lambda_{r + 1}$ or if $r = 1$.  Finally, if $\lambda_{r - 1} < \lambda_r < \lambda_{r + 1}$ then $d_\lambda = 1$ and the result is trivial.
\end{proof}

\begin{thm}
\label{thm:connected components}
Let $\lambda$ be a partition and $T, T'$ be tabloids of shape $\lambda$. Then $T$ and $T'$ are in the same connected component of $\A_\lambda$ if and only if 
\[
\ch(T)\equiv\ch(T')\pmod{d_\lambda}.
\]
\end{thm}

\begin{proof}
We will show that if 
$T$ and $T'$ are two tabloids of shape $\lambda$ connected by a Knuth move, then $\ch(T)\equiv\ch(T')\pmod{d_\lambda}$.
The ``only if'' direction of the claimed result follows immediately by induction. The ``if'' direction then follows using Proposition \ref{prop:upper bound} and Lemma \ref{lem:all charges are non-empty}.

Suppose the Knuth move between $T$ and $T'$ exchanges the entries $\ol{i}$ and $\ol{i+1}$; without loss of generality, we may assume that $\ol{i}$ is in a higher row of $T$ than $\ol{i+1}$.  We consider two cases, depending on the value of $\ol{i}$.

\textbf{Case $\ol{i} \neq \ol{n}$.}  In this case, $\ol{i}$ and $\ol{i + 1}$ are either both larger or both smaller in broken order than each other entry in the tabloid.  Thus, if these entries lie in non-adjacent rows then the local charge between each pair of adjacent rows is the same in $T$ and $T'$, and so $\ch(T) = \ch(T')$.  Suppose instead that $\ol{i}$ and $\ol{i + 1}$ lie in adjacent rows $r$, $r + 1$.  The only local charge that could differ between $T$ and $T'$ is $\lch_r$, and so the only interesting case is when $\lambda_r = \lambda_{r + 1}$. Moreover, the only way the local charge $\lch_r$ could differ between $T$ and $T'$ is if $\ol{i}$ is matched with $\ol{i+1}$ in the charge matching in $T$.  By Lemma~\ref{lem:local independent of activation}, the local charge does not depend on the activation ordering, so it suffices to produce an activation ordering for $T$ in which $\ol{i}$ is not matched with $\ol{i + 1}$.

In order for the swap of $\ol{i}$ and $\ol{i + 1}$ to be a Knuth move, the descent sets $\tau(T)$ and $\tau(T')$ must be incomparable.  This is possible only if $T_r$ contains $\ol{i-1}$ or $T_{r+1}$ contains $\ol{i+2}$ (or both). Suppose first that $T_r$ contains $\ol{i-1}$.  
In this case, take an activation ordering that begins with $\ol{i - 1}$: if $\ol{i} \neq \ol{1}$, then $\ol{i + 1}$ is the smallest value larger than $\ol{i - 1}$ in broken order, so $(\ol{i - 1}, \ol{i + 1})$ belongs to the charge matching; if $\ol{i} = \ol{1}$ then $\ol{i - 1} = \ol{n}$ matches to the smallest element of $T_2$, which is $\ol{2} = \ol{i + 1}$.  Since $\ol{i - 1}$ matches with $\ol{i + 1}$ in both cases, $\ol{1}$ does not, as needed.

Now suppose instead that $T_{r + 1}$ contains $\ol{i+2}$. In this case, take an activation ordering that ends with $\ol{i}$: if $\ol{i} \neq \ol{n-1}$ then $\ol{i+2}$ gets matched strictly after $\ol{i+1}$ does, so $\ol{i+1}$ is matched to something earlier that $\ol{i}$; if $\ol{i} = \ol{n-1}$ then the second row contains $\ol{i + 2} = \ol{1}$, which again gets matched strictly after $\ol{i + 1} = \ol{n}$.  In both cases, $\ol{i}$ does not match to $\ol{i + 1}$, as needed.

\textbf{Case $\ol{i} = \ol{n}$.} Suppose that $\ol{n}$ is in row $r_1$ in $T$ and $\ol{1}$ is in row $r_2 > r_1$. Then $\lch_j(T) = \lch_j(T')$ for every $j \not \in \{r_1 - 1, r_1, r_2 - 1, r_2\}$, and so the charge difference $\ch(T) - \ch(T')$ depends only on those four or fewer local charges.  For convenience, define $\Delta_{j} := j \lch_{j}(T) - j\lch_{j}(T')$.  We have numerous possible cases to consider: whether or not $\lambda_j = \lambda_{j + 1}$ for each $j \in \{r_1 - 1, r_1, r_2 - 1, r_2\}$, and whether $r_2 = r_1 + 1$ or not.

Suppose that $\lambda_{r_1 - 1} = \lambda_{r_1}$. Then we have by Lemma~\ref{lemma:charge_decrement} that $\lch_{r_1 - 1}(T) = \lch_{r_1 - 1}(T') - 1$, and so $\Delta_{r_1 - 1} = 1 - r_1$.  If instead $\lambda_{r_1 - 1} < \lambda_{r_1}$ or $r_1 = 1$ then $r_1 \equiv 1 \pmod{d_\lambda}$ and so $\Delta_{r_1 - 1} = 0 \equiv 1 - r_1 \pmod{d_\lambda}$ in this case as well.

It is not hard to modify the proof of Lemma~\ref{lemma:charge_decrement} to the situation in which the largest and smallest elements are swapped out of the \emph{top}, rather than bottom, row; it follows from this modified lemma and an argument identical to the preceding paragraph that $\Delta_{r_2} \equiv -r_2 \pmod{d_\lambda}$.

Suppose now that rows $r_1$ and $r_2$ are not adjacent, so that $T_{r_2 - 1} = T'_{r_2 - 1}$.  If $\lambda_{r_2 - 1} = \lambda_{r_2}$ then by Lemma~\ref{lemma:charge_decrement} we have $\lch_{r_2 - 1}(T') = \lch_{r_2 - 1}(T) - 1$, and so $\Delta_{r_2 - 1} = r_2 - 1$.  As before, if $\lambda_{r_2 - 1} > \lambda_{r_2}$ then $r_2 \equiv 1 \pmod{d_\lambda}$ and so $\Delta_{r_2 - 1} = 0 \equiv r_2 - 1 \pmod{d_\lambda}$ in this case as well.  Similarly, as in the preceding paragraph we have $\Delta_{r_1} \equiv r_1 \pmod{d_\lambda}$.

Combining the results of the previous three paragraphs, when $r_1 < r_2 - 1$ we have that
\[
\ch(T) - \ch(T') = \Delta_{r_1 - 1} + \Delta_{r_1} + \Delta_{r_2 - 1} + \Delta_{r_2} \equiv (1 - r_1) + r_1 - r_2 + (r_2 - 1) \equiv 0 \pmod{d_\lambda},
\]
as claimed.

In the remainder of the proof, we consider the case that $r := r_1$ is equal to $r_2 - 1$, so that $\ol{1}$ and $\ol{n}$ lie in adjacent rows in $T$ and $T'$.  In this case, there are (at most) three contributions to the charge difference: $\ch(T) - \ch(T') = \Delta_{r - 1} + \Delta_{r} + \Delta_{r + 1}$.  The analyses of the second and third paragraphs of the case $\ol{i} = \ol{n}$ still apply in this setting, giving $\Delta_{r - 1} \equiv 1 - r$ and $\Delta_{r + 1} \equiv -1 - r$ modulo $d_\lambda$.

First, suppose that $\lambda_r > \lambda_{r + 1}$.  Then $\Delta_r = 0$, and so $\Delta_{r - 1} + \Delta_{r} + \Delta_{r + 1} \equiv -2r \equiv 0 \pmod{d_\lambda}$, as needed.  

Second (and last), suppose that $\lambda_r = \lambda_{r + 1} =: k$.  We claim that $\lch_r(T) - \lch_r(T') = 2$.  Since the swap of $\ol{n}$ and $\ol{1}$ is a Knuth move, either $\ol{n-1}$ is in row $r$ of $T$ or $\ol{2}$ is in row $r+1$. We use different activation orders depending on which situation we are in; they are illustrated in Figure \ref{fig:funny activation}. We analyze the case when $\ol{n-1}$ is in row $r$ (the top half of the figure). Since the other case is similar, we omit the details, but we do provide the relevant illustration in the bottom half of the figure.

\begin{figure}
\begin{center}
\resizebox{.9\textwidth}{!}{\input{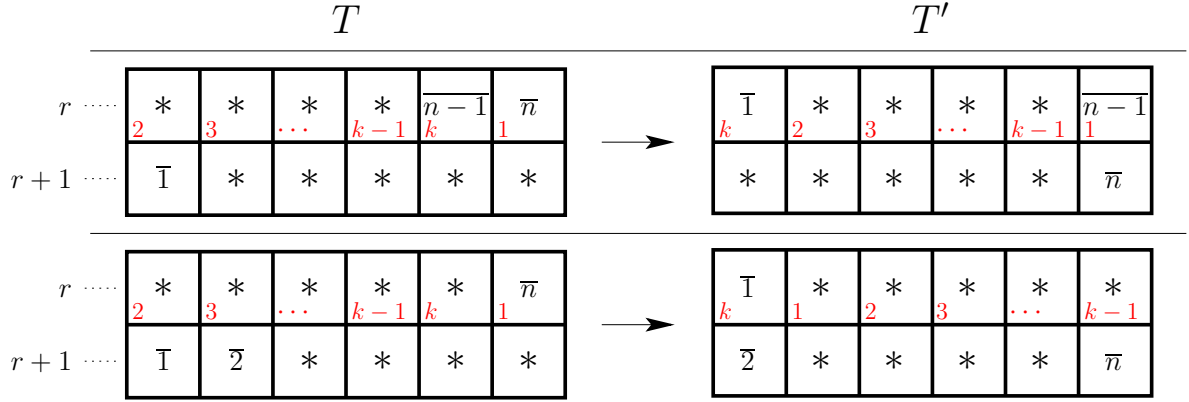}}
\end{center}
\caption{Activation orders used in the proof of Theorem \ref{thm:connected components} in the case that $\ol{n}$ and $\ol{1}$ are swapped and lie in adjacent rows.  Black numbers are tabloid entries; black $*$'s represent entries whose values are irrelevant; and red numbers represent the activation order.}
\label{fig:funny activation}
\end{figure}

In $T$, we use the activation order that begins with $\ol{n}$ first, then proceeds through the other values in some order, then ends with $\ol{n - 1}$.  With this ordering, $\ol{n}$ matches with $\ol{1}$ in row $r + 1$ and contributes $1$ to the local charge.  The next $k - 2$ entries match in some order to $k - 2$ of the elements in row $r + 1$.  Finally, in the last activation in $T$, $\ol{n-1}$ contributes $1$ to the local charge (since the only entry of $T$ larger than $\ol{n-1}$ is $\ol{n}$, and it is not in row $r+1$).  In $T'$, we use the order that begins with $\ol{n - 1}$, then proceeds through the other values in the same order as in $T$, then ends with $\ol{1}$.  With this activation ordering, $\ol{n - 1}$ matches to $\ol{n}$ in row $r + 1$ and does not contribute to the local charge.  At this point, the unmatched entries in row $r + 1$ of $T'$ are exactly the same as the unmatched entries in row $r + 1$ of $T$ after the first pair are matched; hence the next $k - 2$ edges of the charge matching are exactly the same in $T'$ as in $T$, with the same contribution to charge.  Finally, the last match is between $\ol{1}$ and something larger in broken order, with no contribution to local charge.  Thus indeed the local charge between these rows is $2$ less for $T'$. 

Putting everything together, we have in this case that $\Delta_r = 2r$ and so 
\[
\ch(T) - \ch(T') = \Delta_{r - 1} + \Delta_r + \Delta_{r + 1} 
\equiv (1 - r) + 2r + (-1 - r) = 0 \pmod{d_\lambda},
\]
as claimed.  This completes all cases.
\end{proof} 

One interesting corollary is that we can determine exactly when a Kazhdan-Lusztig cell consists of a single Knuth equivalence class.
\begin{cor}
The Kazhdan-Lusztig cell containing the permutation $w$ consists of a single Knuth equivalence class if and only if the shape $\lambda$ of $w$ has all rows distinct.
\end{cor}
\begin{proof}
Choose a permutation $w \in \widetilde{S_n}$ of shape $\lambda$.  If $\lambda$ has a repeated part, say $\lambda_r = \lambda_{r + 1}$, then by Theorem~\ref{thm:monodromy description}, the difference $\rho(w')_r - \rho(w')_{r + 1}$ is the same for every $w'$ in the Knuth equivalence class of $w$.  However, from the definition of $\dom$, this is not the case for every $w'$ in the same Kazhdan-Lusztig cell as $w$.  Thus distinct rows is a necessary condition.  On the other hand, if all rows of $\lambda$ are distinct then by Theorem~\ref{thm:connected components}, the Knuth class contains permutations with every possible $Q$-tabloid of shape $\lambda$, and by Theorem~\ref{thm:monodromy description}, for each $Q$-tabloid the Knuth class contains permutations with every possible weight vector $\rho$.  Thus in this case the Knuth class is equal to the Kazhdan-Lusztig cell.
\end{proof}

\begin{rmk}
The results of this section combined with Proposition \ref{prop:symmetries} imply that the action of the shift permutation $s = [2,3,\ldots, n, n + 1]$ on tabloids of shape $\lambda$ provides graph isomorphisms between the $d_\lambda$ connected components of $\A_\lambda$.
\end{rmk}

\section{Crystals}
\label{sec:crystals}

In this section, we briefly describe a connection to affine crystals that is suggested by the importance of the charge statistic; the finite version of the connection has been developed by Assaf \cite{assaf}. We use notation for affine crystals from \cite[\textsection 4]{crysdumb}, except we label crystal operators $f_i$ and $e_i$ by residue classes modulo $n$ as opposed to by the representatives of these classes from $\{0,1,\ldots,n-1\}$. 

There is a \emph{reading word map} $RW$ from the set of tabloids of shape $\lambda = \langle\lambda_1,\lambda_2,\ldots,\lambda_\ell\rangle$ to the set of vertices of $B := B^{1,\lambda_\ell}\otimes B^{1,\lambda_{\ell-1}}\otimes\cdots\otimes B^{1,\lambda_1}$: it reads the rows from bottom to top and takes representatives in $[n]$ for all residue classes. For example,
\[
RW\left(\quad\tableau[sY]{\ol{2}, \ol{3}, \ol{7}\\ \ol{1}, \ol{4}\\ \ol{5}\\ \ol{6}}\quad\right)
\quad=\quad
\tableau[sY]{6}\otimes\tableau[sY]{5}\otimes\tableau[sY]{1,4}\otimes\tableau[sY]{2,3,7}\;.
\]
The image of this map is precisely the set of vertices of $B$ where each element of $[n]$ is used exactly once; call this set of vertices $V$. For $\ol{i}\in[\ol{n}]$, let $K_{\ol{i}}$ be the composition of crystal operators $e_{\ol{i}}e_{\ol{i+1}}f_{\ol{i}}f_{\ol{i+1}}$.

\begin{prop}
\label{prop:crystal Knuth moves}
Suppose $T$ and $T'$ are tabloids of the same shape. Then $T$ and $T'$ are connected by a Knuth move if and only if $RW(T)\neq RW(T')$ and the two reading words are connected by some $K_{\ol{i}}$.
\end{prop}
\begin{proof}[Proof sketch]
Suppose $T$ and $T'$ are connected by a Knuth move that exchanges $\ol{i}$ and $\ol{i+1}$; without loss of generality $\ol{i}\notin\tau(T)$. Suppose $\ol{i}$ is in row $k$ of $T$ and $\ol{i+1}$ is in row $\ell$; so $k > \ell$. Since the exchange of $\ol{i}$ and $\ol{i+1}$ is a Knuth move, we know that either $\ol{i-1}$ lies in a row weakly below $\ell$ and strictly above $k$, or $\ol{i+2}$ lies in a row strictly below $\ell$ and weakly above $k$. We treat the first of these cases; the second is identical. One easily checks by definition of the crystal operators that $K_{\ol{i-1}}(RW(T)) = RW(T')$ (the computation depends on whether $i = 1$, $i = n$, or neither).

Now suppose that $RW(T)\neq RW(T')$ and $K_{\ol{i}}(RW(T)) = RW(T')$. One needs to check that whatever the positions of (the representatives in $[n]$ of) $\ol{i}$, $\ol{i+1}$ and $\ol{i+2}$ in $RW(T)$, the result of the action by $K_{\ol{i}}$ yields either the reading word of a tabloid related to $T$ by a Knuth move, or $RW(T)$ again, or some vertex of $B$ that is not in $V$. The computation depends on whether $\ol{i} = \ol{n}$, $\ol{i} = \ol{n-1}$, or neither.
\end{proof}

Finally, we give a simple application of the crystal connection. Notice that the definition of KL DEGs $\A_\mu$ does not rely on the shape $\mu$ being a partition.

\begin{prop}
Suppose $\mu$ is a composition and $\lambda$ is the partition formed from $\mu$ by sorting. Then there is a graph isomorphism $\varphi:\A_\lambda\to\A_\mu$.
\end{prop}
\begin{proof}
It is sufficient to show that $A_\mu$ is isomorphic to $A_{\mu'}$ where $\mu'$ is obtained from $\mu$ by switching two adjacent parts. By \cite[Thm.~4.8]{crysdumb}, there is an isomorphism between $B^{1,r}\otimes B^{1,s}$ and  $B^{1,s}\otimes B^{1,r}$ for any $r$ and $s$ (it is called the \emph{combinatorial $R$-matrix}). Moreover, the explicit description of this morphism in the same section shows that it preserves the number of times each letter appears in the tableaux. Applying this morphism to two adjacent factors of $B$ and pulling back the results using Proposition~\ref{prop:crystal Knuth moves} (which also does not rely on the shape being a partition) finishes the proof.  
\end{proof}

\bibliographystyle{alpha}
\bibliography{affine}

\end{document}